\documentclass[reqno, 12pt]{amsart}
\pdfoutput=1

\def\ssign{\textsection\nobreak\hspace{1pt plus 0.3pt}}
\makeatletter
\let\origsection=\section 
\def\mysection{\@mystartsection{section}{1}\z@{.7\linespacing\@plus\linespacing}{.5\linespacing}{\normalfont\scshape\centering\ssign}}
\def\section{\@ifstar{\origsection*}{\mysection}}
\def\appendix{\par\c@section\z@ \c@subsection\z@
	\let\sectionname\appendixname
	\let\section=\origsection
	\def\thesection{\@Alph\c@section}} 
\def\@mystartsection#1#2#3#4#5#6{\if@noskipsec \leavevmode \fi
	\par \@tempskipa #4\relax
	\@afterindenttrue
	\ifdim \@tempskipa <\z@ \@tempskipa -\@tempskipa \@afterindentfalse\fi
	\if@nobreak \everypar{}\else
	\addpenalty\@secpenalty\addvspace\@tempskipa\fi
	\@dblarg{\@mysect{#1}{#2}{#3}{#4}{#5}{#6}}}
\def\@mysect#1#2#3#4#5#6[#7]#8{\edef\@toclevel{\ifnum#2=\@m 0\else\number#2\fi}\ifnum #2>\c@secnumdepth \let\@secnumber\@empty
	\else \@xp\let\@xp\@secnumber\csname the#1\endcsname\fi
	\@tempskipa #5\relax
	\ifnum #2>\c@secnumdepth
	\let\@svsec\@empty
	\else
	\refstepcounter{#1}\edef\@secnumpunct{\ifdim\@tempskipa>\z@ \@ifnotempty{#8}{\@nx\enspace}\else
		\@ifempty{#8}{.}{\@nx\enspace}\fi
	}\@ifempty{#8}{\ifnum #2=\tw@ \def\@secnumfont{\bfseries}\fi}{}\protected@edef\@svsec{\ifnum#2<\@m
		\@ifundefined{#1name}{}{\ignorespaces\csname #1name\endcsname\space
		}\fi
		\@seccntformat{#1}}\fi
	\ifdim \@tempskipa>\z@ \begingroup #6\relax
	\@hangfrom{\hskip #3\relax\@svsec}{\interlinepenalty\@M #8\par}\endgroup
	\ifnum#2>\@m \else \@tocwrite{#1}{#8}\fi
	\else
	\def\@svsechd{#6\hskip #3\@svsec
		\@ifnotempty{#8}{\ignorespaces#8\unskip
			\@addpunct.}\ifnum#2>\@m \else \@tocwrite{#1}{#8}\fi
	}\fi
	\global\@nobreaktrue
	\@xsect{#5}}
\makeatother

\usepackage{amsmath,amssymb,amsthm, nccmath}
\usepackage{mathrsfs}
\usepackage{mathabx}\changenotsign
\usepackage{dsfont}
\usepackage{bbm}
\usepackage{todonotes}

\usepackage{xcolor}
\usepackage[backref=section]{hyperref}
\usepackage[ocgcolorlinks]{ocgx2}
\hypersetup{
	colorlinks=true,
	linkcolor={red!60!black},
	citecolor={green!60!black},
	urlcolor={blue!60!black},
}

\usepackage[open,openlevel=2,atend]{bookmark}

\usepackage[abbrev,msc-links,backrefs]{amsrefs}
\usepackage{doi}

\renewcommand{\PrintDOI}[1]{\doi{#1}}

\usepackage[T1]{fontenc}
\usepackage{lmodern}
\usepackage[babel]{microtype}
\usepackage[english]{babel}

\usepackage{geometry}
\numberwithin{equation}{section}
\numberwithin{figure}{section}

\usepackage{enumitem}
\def\rmlabel{\upshape({\itshape \roman*\,})}

\def\nlabel{\upshape({\itshape \arabic*\,})}

\let\polishlcross=\l
\def\l{\ifmmode\ell\else\polishlcross\fi}

\let\emptyset=\varnothing
\let\setminus=\smallsetminus

\makeatletter
\def\moverlay{\mathpalette\mov@rlay}
\def\mov@rlay#1#2{\leavevmode\vtop{   \baselineskip\z@skip \lineskiplimit-\maxdimen
		\ialign{\hfil$\m@th#1##$\hfil\cr#2\crcr}}}
\newcommand{\charfusion}[3][\mathord]{
	#1{\ifx#1\mathop\vphantom{#2}\fi
		\mathpalette\mov@rlay{#2\cr#3}
	}
	\ifx#1\mathop\expandafter\displaylimits\fi}
\makeatother

\newcommand{\dcup}{\charfusion[\mathbin]{\cup}{\cdot}}
\newcommand{\bigdcup}{\charfusion[\mathop]{\bigcup}{\cdot}}

\DeclareFontFamily{U}  {MnSymbolC}{}
\DeclareSymbolFont{MnSyC}         {U}  {MnSymbolC}{m}{n}
\DeclareFontShape{U}{MnSymbolC}{m}{n}{
	<-6>  MnSymbolC5
	<6-7>  MnSymbolC6
	<7-8>  MnSymbolC7
	<8-9>  MnSymbolC8
	<9-10> MnSymbolC9
	<10-12> MnSymbolC10
	<12->   MnSymbolC12}{}
\DeclareMathSymbol{\powerset}{\mathord}{MnSyC}{180}

\usepackage{tikz}
\usetikzlibrary{calc,decorations.pathmorphing}
\usetikzlibrary{arrows,decorations.pathreplacing}
\pgfdeclarelayer{background}
\pgfdeclarelayer{foreground}
\pgfdeclarelayer{front}
\pgfsetlayers{background,main,foreground,front}

\usepackage{multicol}
\usepackage{subcaption}

\newcommand{\pedge}[9]{
	
	\ifx\relax#6\relax
	\def\qoffs{0pt}
	\else
	\def\qoffs{#6}
	\fi
	
	\def\phedge{
		($#1+#5!\qoffs!-90:#2-#5$) -- 
		($#2+#1!\qoffs!-90:#3-#1$) -- 
		($#3+#2!\qoffs!-90:#4-#2$) -- 
		($#4+#3!\qoffs!-90:#5-#3$) -- 
		($#5+#4!\qoffs!-90:#1-#4$) -- cycle}

	\coordinate (12) at ($#1!\qoffs!90:#2$);
	\coordinate (15) at ($#1!\qoffs!-90:#5$);
	\coordinate (23) at ($#2!\qoffs!90:#3$);
	\coordinate (21) at ($#2!\qoffs!-90:#1$);
	\coordinate (34) at ($#3!\qoffs!90:#4$);
	\coordinate (32) at ($#3!\qoffs!-90:#2$);
	\coordinate (45) at ($#4!\qoffs!90:#5$);
	\coordinate (43) at ($#4!\qoffs!-90:#3$);
	\coordinate (51) at ($#5!\qoffs!90:#1$);
	\coordinate (54) at ($#5!\qoffs!-90:#4$);

	\def\nphedge{
		(15) let \p1=($(15)-#1$), \p2=($(12)-#1$) in 
		arc[start angle={atan2(\y1,\x1)}, delta angle={atan2(\y2,\x2)-atan2(\y1,\x1)-360*(atan2(\y2,\x2)-atan2(\y1,\x1)>0)}, x radius=\qoffs, y radius=\qoffs] --
		(21) let \p1=($(21)-#2$), \p2=($(23)-#2$) in 
		arc[start angle={atan2(\y1,\x1)}, delta angle={atan2(\y2,\x2)-atan2(\y1,\x1)-360*(atan2(\y2,\x2)-atan2(\y1,\x1)>0)}, x radius=\qoffs, y radius=\qoffs] --
		(32) let \p1=($(32)-#3$), \p2=($(34)-#3$) in 
		arc[start angle={atan2(\y1,\x1)}, delta angle={atan2(\y2,\x2)-atan2(\y1,\x1)-360*(atan2(\y2,\x2)-atan2(\y1,\x1)>0)}, x radius=\qoffs, y radius=\qoffs] --
		(43) let \p1=($(43)-#4$), \p2=($(45)-#4$) in 
		arc[start angle={atan2(\y1,\x1)}, delta angle={atan2(\y2,\x2)-atan2(\y1,\x1)-360*(atan2(\y2,\x2)-atan2(\y1,\x1)>0)}, x radius=\qoffs, y radius=\qoffs] --
		(54) let \p1=($(54)-#5$), \p2=($(51)-#5$) in 
		arc[start angle={atan2(\y1,\x1)}, delta angle={atan2(\y2,\x2)-atan2(\y1,\x1)-360*(atan2(\y2,\x2)-atan2(\y1,\x1)>0)}, x radius=\qoffs, y radius=\qoffs] --
		cycle}
	
	\ifx\relax#7\relax
	\def\plwidth{1pt}
	\else
	\def\plwidth{#7}
	\fi
	
	\ifx\relax#9\relax
	\fill \nphedge;
	\else
	\fill[#9]\nphedge;
	\fi
	
	\ifx\relax#8\relax
	\draw[line width=\plwidth,rounded corners=\qoffs]\nphedge;
	\else
	\draw[line width=\plwidth,#8]\nphedge;
	\fi
}

\newcommand{\qedge}[7]{
	
	\ifx\relax#4\relax
	\def\qoffs{0pt}
	\else
	\def\qoffs{#4}
	\fi
	
	\def\qhedge{
		($#1+#3!\qoffs!-90:#2-#3$) --
		($#2+#1!\qoffs!-90:#3-#1$) --
		($#3+#2!\qoffs!-90:#1-#2$) -- cycle}

	\coordinate (12) at ($#1!\qoffs!90:#2$);
	\coordinate (13) at ($#1!\qoffs!-90:#3$);
	\coordinate (23) at ($#2!\qoffs!90:#3$);
	\coordinate (21) at ($#2!\qoffs!-90:#1$);
	\coordinate (31) at ($#3!\qoffs!90:#1$);
	\coordinate (32) at ($#3!\qoffs!-90:#2$);
	
	\def\nqhedge{
		(13) let \p1=($(13)-#1$), \p2=($(12)-#1$) in
		arc[start angle={atan2(\y1,\x1)}, delta angle={atan2(\y2,\x2)-atan2(\y1,\x1)-360*(atan2(\y2,\x2)-atan2(\y1,\x1)>0)}, x radius=\qoffs, y radius=\qoffs] --
		(21) let \p1=($(21)-#2$), \p2=($(23)-#2$) in
		arc[start angle={atan2(\y1,\x1)}, delta angle={atan2(\y2,\x2)-atan2(\y1,\x1)-360*(atan2(\y2,\x2)-atan2(\y1,\x1)>0)}, x radius=\qoffs, y radius=\qoffs] --
		(32) let \p1=($(32)-#3$), \p2=($(31)-#3$) in
		arc[start angle={atan2(\y1,\x1)}, delta angle={atan2(\y2,\x2)-atan2(\y1,\x1)-360*(atan2(\y2,\x2)-atan2(\y1,\x1)>0)}, x radius=\qoffs, y radius=\qoffs] --
		cycle}
	
	\ifx\relax#5\relax
	\def\qlwidth{1pt}
	\else
	\def\qlwidth{#5}
	\fi
	
	\ifx\relax#7\relax
	\fill \nqhedge;
	\else
	\fill[#7]\nqhedge;
	\fi
	
	\ifx\relax#6\relax
	\draw[line width=\qlwidth,rounded corners=\qoffs]\nqhedge;
	\else
	\draw[line width=\qlwidth,#6]\nqhedge;
	\fi
}

\newcommand{\redge}[8]{
	
	\ifx\relax#5\relax
	\def\qoffs{0pt}
	\else
	\def\qoffs{#5}
	\fi
	
	\def\rhedge{
		($#1+#4!\qoffs!-90:#2-#4$) -- 
		($#2+#1!\qoffs!-90:#3-#1$) -- 
		($#3+#2!\qoffs!-90:#4-#2$) -- 
		($#4+#3!\qoffs!-90:#1-#3$) -- cycle}

	\coordinate (12) at ($#1!\qoffs!90:#2$);
	\coordinate (14) at ($#1!\qoffs!-90:#4$);
	\coordinate (23) at ($#2!\qoffs!90:#3$);
	\coordinate (21) at ($#2!\qoffs!-90:#1$);
	\coordinate (34) at ($#3!\qoffs!90:#4$);
	\coordinate (32) at ($#3!\qoffs!-90:#2$);
	\coordinate (41) at ($#4!\qoffs!90:#1$);
	\coordinate (43) at ($#4!\qoffs!-90:#3$);
	
	\def\nrhedge{
		(14) let \p1=($(14)-#1$), \p2=($(12)-#1$) in 
		arc[start angle={atan2(\y1,\x1)}, delta angle={atan2(\y2,\x2)-atan2(\y1,\x1)-360*(atan2(\y2,\x2)-atan2(\y1,\x1)>0)}, x radius=\qoffs, y radius=\qoffs] --
		(21) let \p1=($(21)-#2$), \p2=($(23)-#2$) in 
		arc[start angle={atan2(\y1,\x1)}, delta angle={atan2(\y2,\x2)-atan2(\y1,\x1)-360*(atan2(\y2,\x2)-atan2(\y1,\x1)>0)}, x radius=\qoffs, y radius=\qoffs] --
		(32) let \p1=($(32)-#3$), \p2=($(34)-#3$) in 
		arc[start angle={atan2(\y1,\x1)}, delta angle={atan2(\y2,\x2)-atan2(\y1,\x1)-360*(atan2(\y2,\x2)-atan2(\y1,\x1)>0)}, x radius=\qoffs, y radius=\qoffs] --
		(43) let \p1=($(43)-#4$), \p2=($(41)-#4$) in 
		arc[start angle={atan2(\y1,\x1)}, delta angle={atan2(\y2,\x2)-atan2(\y1,\x1)-360*(atan2(\y2,\x2)-atan2(\y1,\x1)>0)}, x radius=\qoffs, y radius=\qoffs] --
		cycle}
	
	\ifx\relax#6\relax
	\def\rlwidth{1pt}
	\else
	\def\rlwidth{#6}
	\fi
	
	\ifx\relax#8\relax
	\fill \nrhedge;
	\else
	\fill[#8]\nrhedge;
	\fi
	
	\ifx\relax#7\relax
	\draw[line width=\rlwidth,rounded corners=\qoffs]\nrhedge;
	\else
	\draw[line width=\rlwidth,#7]\nrhedge;
	\fi
}

\let\phi=\varphi
\let\epsilon=\varepsilon
\let\eps=\epsilon
\let\rho=\varrho
\let\theta=\vartheta

\def\EE{{\mathds E}}
\def\NN{{\mathds N}}

\def\PP{{\mathds P}}

\newcommand{\Var}{\text{Var}}
\newcommand{\bbS}{\mathbb{S}}

\newcommand{\bbE}{\mathbb{E}}
\newcommand{\bx}{\mathbf{x}}
\newcommand{\bX}{\mathbf{X}}
\newcommand{\by}{\mathbf{y}}
\newcommand{\bz}{\mathbf{z}}

\newcommand{\cA}{\mathcal{A}}
\newcommand{\cB}{\mathcal{B}}

\newcommand{\cE}{\mathcal{E}}
\newcommand{\cF}{\mathcal{F}}
\newcommand{\cG}{\mathcal{G}}
\newcommand{\cH}{\mathcal{H}}

\newcommand{\cP}{\mathcal{P}}

\newcommand{\cR}{\mathcal{R}}
\newcommand{\cS}{\mathcal{S}}

\newcommand{\cV}{\mathcal{V}}

\newcommand{\fv}{\mathfrak v}

\newcommand{\ccA}{\mathscr{A}}

\DeclareMathOperator{\id}{id}
\DeclareMathOperator{\rev}{rev}
\DeclareMathOperator{\Ex}{ex}
\DeclareMathOperator{\pal}{pal}

\newcommand{\fin}{\text{fin}}

\newcommand{\ceil}[1]{\left\lceil #1 \right\rceil}

\newtheoremstyle{note}  {4pt}  {4pt}  {\sl}  {}  {\bfseries}  {.}  {.5em}          {}
\newtheoremstyle{introthms}  {3pt}  {3pt}  {\itshape}  {}  {\bfseries}  {.}  {.5em}          {\thmnote{#3}}
\newtheoremstyle{remark}  {2pt}  {2pt}  {\rm}  {}  {\bfseries}  {.}  {.3em}          {}

\theoremstyle{plain}
\newtheorem{theorem}{Theorem}[section]
\newtheorem{lemma}[theorem]{Lemma}
\newtheorem{prop}[theorem]{Proposition}

\newtheorem{cor}[theorem]{Corollary}
\newtheorem{fact}[theorem]{Fact}
\newtheorem{claim}[theorem]{Claim}
\newtheorem{subclaim}[theorem]{Subclaim}
\newtheorem{obs}[theorem]{Observation}

\theoremstyle{note}
\newtheorem{dfn}[theorem]{Definition}

\theoremstyle{remark}
\newtheorem{remark}[theorem]{Remark}

\usepackage{lineno}
\newcommand*\patchAmsMathEnvironmentForLineno[1]{
	\expandafter\let\csname old#1\expandafter\endcsname\csname #1\endcsname
	\expandafter\let\csname oldend#1\expandafter\endcsname\csname end#1\endcsname
	\renewenvironment{#1}
	{\linenomath\csname old#1\endcsname}
	{\csname oldend#1\endcsname\endlinenomath}}
\newcommand*\patchBothAmsMathEnvironmentsForLineno[1]{
	\patchAmsMathEnvironmentForLineno{#1}
	\patchAmsMathEnvironmentForLineno{#1*}}
\AtBeginDocument{
	\patchBothAmsMathEnvironmentsForLineno{equation}
	\patchBothAmsMathEnvironmentsForLineno{align}
	\patchBothAmsMathEnvironmentsForLineno{flalign}
	\patchBothAmsMathEnvironmentsForLineno{alignat}
	\patchBothAmsMathEnvironmentsForLineno{gather}
	\patchBothAmsMathEnvironmentsForLineno{multline}
}

\usepackage{scalerel}

\newsavebox\vdegbox
\savebox\vdegbox{\tikz{
		\draw[black,fill=black] (90:1) circle (.35);
		\draw[black,line width=0.10cm] (210:1) circle (.30);
		\draw[black,line width=0.10cm] (330:1) circle (.30);
		\draw[opacity=0] (0:1.2) circle (0.1);
}}

\newsavebox\pdegbox
\savebox\pdegbox{\tikz{
		\draw[black,line width=0.10cm] (90:1) circle (.30);
		\draw[black,fill=black] (210:1) circle (.35);
		\draw[black,fill=black] (330:1) circle (.35);
		\draw[black,line width=0.28cm ] (210:1) -- (330:1);
		\draw[opacity=0] (0:1.2) circle (0.1);
}}

\newsavebox\vvvbox
\savebox\vvvbox{\tikz{
		\draw[black,fill=black] (90:1) circle (.35);
		\draw[black,fill=black] (210:1) circle (.35);
		\draw[black,fill=black] (330:1) circle (.35);
		\draw[opacity=0] (0:1.2) circle (0.1);
}}
\newcommand{\vvv}{\mathord{\scaleobj{1.2}{\scalerel*{\usebox{\vvvbox}}{x}}}}
\newcommand{\pivvv}{\pi_{\vvv}}

\newsavebox\evbox
\savebox\evbox{\tikz{
		\draw[black,fill=black] (90:1) circle (.35);
		\draw[black,fill=black] (210:1) circle (.35);
		\draw[black,fill=black] (330:1) circle (.35);
		\draw[black,line width=0.28cm ] (210:1) -- (330:1);
		\draw[opacity=0] (0:1.2) circle (0.1);
}}

\newsavebox\eebox
\savebox\eebox{\tikz{
		\draw[black,fill=black] (90:1) circle (.35);
		\draw[black,fill=black] (210:1) circle (.35);
		\draw[black,fill=black] (330:1) circle (.35);
		\draw[black,line width=0.28cm ] (90:1) -- (330:1);
		\draw[black,line width=0.28cm ] (90:1) -- (210:1);
		\draw[opacity=0] (0:1.2) circle (0.1);
}}

\newsavebox\eeebox
\savebox\eeebox{\tikz{
		\draw[black,fill=black] (90:1) circle (.35);
		\draw[black,fill=black] (210:1) circle (.35);
		\draw[black,fill=black] (330:1) circle (.35);
		\draw[black,line width=0.28cm ] (90:1) -- (330:1);
		\draw[black,line width=0.28cm ] (90:1) -- (210:1);
		\draw[black,line width=0.28cm ] (210:1) -- (330:1);
		\draw[opacity=0] (0:1.2) circle (0.1);
}}

\newcommand{\pired}{\pi^{\text{rd}}_{\vvv}}
\newcommand{\expal}{\text{ex}_{\text{pal}}}
\newcommand{\EXpal}{\text{EX}_{\text{pal}}}

\makeatletter
\newcommand{\overrighharpoonup}[1]{\ThisStyle{%
		\vbox {\m@th\ialign{##\crcr
				\rightharpoonupfill \crcr
				\noalign{\kern-\p@\nointerlineskip}
				$\hfil\SavedStyle#1\hfil$\crcr}}}}

\def\rightharpoonupfill{%
	$\SavedStyle\m@th\mkern+0.8mu\cleaders\hbox{$\shortbar\mkern-4mu$}\hfill\rightharpoonuptip\mkern+0.8mu$}

\def\rightharpoonuptip{%
	\raisebox{\z@}[2pt][1pt]{\scalebox{0.55}{$\SavedStyle\rightharpoonup$}}}

\def\shortbar{%
	\smash{\scalebox{0.55}{$\SavedStyle\relbar$}}}
\makeatother

\makeatletter
\newcommand{\overlefharpoonup}[1]{\ThisStyle{%
		\vbox {\m@th\ialign{##\crcr
				\leftharpoonupfill \crcr
				\noalign{\kern-\p@\nointerlineskip}
				$\hfil\SavedStyle#1\hfil$\crcr}}}}

\def\leftharpoonupfill{%
	$\SavedStyle\m@th\mkern+0.8mu\cleaders\hbox{$\shortbar\mkern-4mu$}\hfill\leftharpoonuptip\mkern+0.8mu$}

\def\leftharpoonuptip{%
	\raisebox{\z@}[2pt][1pt]{\scalebox{0.55}{$\SavedStyle\leftharpoonup$}}}

\makeatletter
\newsavebox\myboxA
\newsavebox\myboxB
\newlength\mylenA

\newcommand*\xoverline[2][0.75]{%
	\sbox{\myboxA}{$\m@th#2$}%
	\setbox\myboxB\null
	\ht\myboxB=\ht\myboxA%
	\dp\myboxB=\dp\myboxA%
	\wd\myboxB=#1\wd\myboxA
	\sbox\myboxB{$\m@th\overline{\copy\myboxB}$}
	\setlength\mylenA{\the\wd\myboxA}
	\addtolength\mylenA{-\the\wd\myboxB}%
	\ifdim\wd\myboxB<\wd\myboxA%
	\rlap{\hskip 0.5\mylenA\usebox\myboxB}{\usebox\myboxA}%
	\else
	\hskip -0.5\mylenA\rlap{\usebox\myboxA}{\hskip 0.5\mylenA\usebox\myboxB}%
	\fi}
\makeatother

\DeclareSymbolFont{symbolsC}{U}{txsyc}{m}{n}
\SetSymbolFont{symbolsC}{bold}{U}{txsyc}{bx}{n}
\DeclareFontSubstitution{U}{txsyc}{m}{n}
\DeclareMathSymbol{\strictif}{\mathrel}{symbolsC}{74}
\DeclareMathSymbol{\strictfi}{\mathrel}{symbolsC}{75}

\title{On possible uniform Tur\'an densities}

 \author[D.~King]{Dylan King}
 \address{Mathematics Department, California Institute of Technology, Pasadena, USA}
 \email{dking@caltech.edu}

 \author[S.~Piga]{Sim\'{o}n Piga}
 \address{Fachbereich Mathematik, Universit\"{a}t Hamburg, Hamburg, Germany}
 \email{simon.piga@uni-hamburg.de}

 \author[M.~Sales]{Marcelo Sales}
 \address{Mathematics Department, University of California Irvine, Irvine, USA}
 \email{mtsales@uci.edu}
 \thanks{The third author is supported by US Air force grant FA9550-23-1-0298.}

 \author[B.~Sch\"ulke]{Bjarne Sch\"ulke}
 \address{Extremal Combinatorics and Probability Group, Institute for Basic Science, Daejeon, South Korea}
 \email{schuelke@ibs.re.kr}
 \thanks{The fourth author is supported by the Young Scientist Fellowship IBS-R029-Y7.}

\thanks{The work on this article was supported by the SQuaRE ``Variants of the hypergraph Tur\'an problem'' at the American Institute of Mathematics.}

\begin{document}
	\maketitle
	\begin{abstract}
  Given a family of $3$-graphs $\cF$, the \emph{uniform Tur\'{a}n density} $\pi_{\vvv}(\cF)$ is defined as the infimum $d\in[0,1]$ for which any sufficiently large uniformly $d$-dense $3$-graph — that is, a $3$-graph which has edge-density at least $d$ on all linearly sized subsets — contains a copy of some $F \in \cF$.
  Let $\Pi_{\vvv,\fin}$ denote the set of all possible uniform Tur\'{a}n densities of finite families.
  Erd\H{o}s, Hajnal, and R\"{o}dl introduced a family of constructions for lower bounds on uniform Turán densities called palette constructions.
  We show that $\Pi_{\vvv,\fin}$ contains every $d$ that is obtained as the uniform density of an optimized palette construction.
  A corollary of this is that $\Pi_{\vvv,\fin}$ contains the set of Lagrangians of $3$-graphs and includes irrational numbers.
  Our work complements a recent result of Lamaison, which states that every value in $\Pi_{\vvv,\fin}$ can be approximated by uniform densities of palette constructions.
	\end{abstract}

	\section{Introduction}\label{sec:intro}
	
	For~$n\in\mathds{N}$ and a family~$\cF$ of~$k$-uniform hypergraphs (or~$k$-graphs), let the \emph{extremal number}~$\Ex(n,\cF)$ be the maximum number of edges in a~$k$-graph~$G$ on~$n$ vertices that does not contain a copy of any~$F\in \cF$.
    Such a~$k$-graph~$G$ is called~\emph{$\cF$-free}.
    It is well known that the quantity~$\Ex(n,\cF)/\binom{n}{k}$ is decreasing \cite{KNS:64}, and therefore one may define the \emph{Tur\'{a}n density} of a family~$\cF$ as
	\begin{align*}
		\pi(\cF):=\lim_{n \to \infty} \frac{\Ex(n,\cF)}{\binom{n}{k}}.
	\end{align*}
	When the family~$\cF=\{F\}$ is a single~$k$-graph, we usually denote~$\pi(\cF)$ by~$\pi(F)$.
    Let~$\Pi^{(k)}_{\infty}$ be the set of all possible Tur\'{a}n densities of families of~$k$-graphs and let~$\Pi^{(k)}_{\fin}$ be the set of Tur\'{a}n densities of finite families~$\cF$.
	
	The study of Tur\'an densities was initiated by Tur\'{a}n \cite{T:41}, who determined $\Ex(n,F)$ when~$F$ is the complete ($2$-)graph.
    Erd\H{o}s, Stone, and Simonovits~\cites{ES:66, ES:46}, generalised this by establishing that
	\begin{align*}
		\pi(F) \geq \frac{\chi(F)-2}{\chi(F)-1}\,,
	\end{align*}
    where~$\chi(F)$ is the chromatic number of~$F$.
	Their proof also gives that
	\begin{align*}
		\Pi^{(2)}_{\infty} = \Pi^{(2)}_{\fin}=\left\{\frac{k}{k+1}:\: k\in \NN_{\geq0}\right\}\,.
	\end{align*}
	
	For higher uniformities, the problem becomes considerably harder and, despite much effort, remains wide open even for~$3$-graphs. Determining the Tur\'{a}n density even for seemingly ``simple''~$3$-graphs is notoriously difficult.
    Perhaps the most famous open problem is finding the Tur\'an density of the complete $3$-graph on four vertices,~$K_4^{(3)}$.
    In fact, even~$\pi(K_4^{(3)-})$, the Tur\'an density of the~$3$-graph with four vertices and three edges, is unknown.
	
	A different line of research investigates properties of the set of Tur\'an densities.
    Disproving a 1000\$ conjecture by Erd\H{o}s, Frankl and R\"odl \cite{FR:84} showed that for~$k \geq 3$ the set $\Pi^{(k)}_{\infty}$ is not well-ordered, i.e., there exists some~$\alpha \in [0,1)$ such that for every~$\epsilon > 0$, the set~$(\alpha, \alpha + \epsilon) \cap \Pi^{(k)}_{\infty}$ is not empty.
    This indicates how much more difficult the hypergraph Tur\'an problem is for~$k$-graphs with~$k\geq3$.
    Recently, Pikhurko \cite{P:12} proved a series of results concerning~$\Pi^{(k)}_{\infty}$ and~$\Pi^{(k)}_{\fin}$.
    In particular, he showed that~$\Pi^{(k)}_{\infty} \neq \Pi^{(k)}_{\fin}$,~$\Pi^{(k)}_{\infty}$ is uncountable, and, using a result by Brown and Simonovits~\cite{BS:84}, that~$\Pi^{(k)}_{\infty} = \overline{\Pi^{(k)}_{\fin}}$.
    However, the full description of the sets $\Pi^{(k)}_{\infty}$ and $\Pi^{(k)}_{\fin}$ remains open.
    For more on the hypergraph Tur\'an problem, we refer to the survey by Keevash \cite{K:11}.
	
	Here we consider a variant of the Tur\'{a}n density suggested by Erd\H{o}s and S\'{o}s \cites{ES:82, E:90}.
    Throughout the rest of the paper, we focus on~$3$-graphs.
    For~$d \in [0,1]$ and~$\eta > 0$, we say that a~$3$-graph~$H$ on~$n$ vertices is~\emph{$(d, \eta)$-dense} if for all~$X \subseteq V(H)$, we have
	\begin{align*}
		e(X) \geq d\binom{|X|}{3} - \eta n^3\,.
	\end{align*}
    The \emph{uniform Tur\'{a}n density~$\pi_{\vvv}$} of a family~$\cF$ of~$3$-graphs is defined as
	\begin{align*}
		\pi_{\vvv}(\cF) = \sup\{d \in [0,1] :\: &\text{for every $\eta > 0$ and $n \in \NN$, there exists}\\
		&\text{an $\cF$-free, $(d, \eta)$-dense $3$-graph $H$ with $|V(H)| \geq n$}\}\,.
	\end{align*}
	In other words,~$\pi_{\vvv}(\cF)$ is the smallest~$d\in[0,1]$ such that there is some~$\eta>0$ such that every sufficiently large~$3$-graph~$H$ on~$n$ vertices that is~$(d+o(1),\eta)$-dense contains a copy of some~$F\in \cF$.
	
	Erd\H{o}s and S\'os specifically asked to determine~$\pi_{\vvv}(K_4^{(3)})$ and~$\pi_{\vvv}(K_4^{(3)-})$.
    Similarly as with the original Tur\'an density, these problems turned out to be very difficult. 
    Only recently, Glebov, Kr\'al', and Volec~\cite{GKV:16} and Reiher, R\"odl, and Schacht~\cite{RRS:18} independently solved the latter, showing that~$\pi_{\vvv}(K_4^{(3)-})=1/4$, which confirmed a conjecture by Erd\H{o}s and S\'os.
    We refer to Reiher's survey~\cite{R:20} for a full description of the landscape of extremal problems in uniformly dense hypergraphs.
	
	Similarly as for the original Tur\'an density, let~$\Pi_{\vvv,\infty}$ be the set of all possible uniform Tur\'an densities of families and~$\Pi_{\vvv,\fin}$ the set of all possible uniform Tur\'an densities of finite families.
    In order to state and discuss the main result of this paper, we need to introduce the concept of Lagrange polynomials, which go back to the work of Motzkin and Straus~\cite{MS:65}.
    Given~$n\in\mathds{N}$ and a subset of ordered triples (which will later be called a ``palette'')~$P \subseteq [n]^{3}$, we define the \emph{Lagrange polynomial of $P$} as 
	\begin{align*}
		\lambda_P(x_1,\ldots,x_n)=\sum_{(i,j,k)\in P}x_ix_jx_k\,,
	\end{align*}
	and the \emph{Lagrangian of $P$}, denoted as~$\Lambda_P$, as the maximum of~$\lambda_P(x_1,\ldots,x_n)$ subject to~$x_1+\ldots+x_n=1$ and~$x_i \in [0,1]$ for all~$i\in[n]$.
    Let~$\Lambda_{\pal}$ be the set of all possible~$\Lambda_P$.
    In~\cite{KSS:24}, it was shown that~$\Lambda_{\pal} \subseteq \Pi_{\vvv,\infty}$ and, as a corollary, that the set~$\Pi_{\vvv,\infty}$ is not well-ordered and contains irrational numbers.
    However, the families constructed in the proof were all infinite, which does not shed light on the possible values of~$\Pi_{\vvv,\fin}$.
    Here we extend the result in~\cite{KSS:24} by showing that~$\Lambda_{\pal}\subseteq \Pi_{\vvv,\fin}$.
	
	\begin{theorem}\label{thm:main}
		For all $\lambda \in \Lambda_{\pal}$, there is a finite family $\cF$ of~$3$-graphs with $\pi_{\vvv}(\cF)=\lambda$.
	\end{theorem}
	
	Until recently, the only known members of $\Pi_{\vvv,\fin}$ were $0, 1/27, 4/27, 1/4$, and $8/27$ \cites{GKL:24, GKV:16, RRS:18, RRS:182, GIKKL:24}.
    Recently, two infinite families of uniform Tur\'an densities were obtained: one converging to~$1/2$ \cite{L:24} and another being the uniform Tur\'an densities of large stars~\cite{LW:24}.
    All of these densities are rational numbers.
    A corollary of Theorem~\ref{thm:main} is that there exist irrational uniform Tur\'an densities of finite families (e.g., see Observation 6.1, \cite{KSS:24}).
    
    Interestingly, one of the core steps (Lemma~\ref{lem:distinguish_palette}) in our proof is about structural Ramsey theory.
    The proof of this key lemma relies on the partite construction method of Ne\v{s}et\v{r}il and R\"{o}dl~\cite{NR:87}, and for this reason the bounds on the graphs in~$\cF$ are enormous.
    A generalization of the aforementioned lemma was obtained independently by Kr\'{a}l, Ku\v{c}er\'{a}k, Lamaison, and Tardos~\cite{KKLT:25+}.
	
	\section{Palettes}\label{sec:intropalettes}
	
	In this section, we present the main technical result of this paper.
	To do so, we first define the notion of a palette.
	
	\begin{dfn}\label{dfn:palLag}
		A \emph{palette} is a pair $P=(C,E)$ consisting of a set of colors~$C$ and a set of patterns~$E \subseteq C^3$.
	\end{dfn}
	
	Although this definition is very similar to that of ordered $3$-graphs, note that a palette may contain degenerate patterns, i.e., patterns that contain fewer than three colors.
	We denote the set of colors of a palette~$P$ by~$C(P)$, while~$P$ should be understood as the set of patterns~$E(P)$.
	Let~$c(P)$ and~$e(P)$ be the number of colors and the number of patterns of~$P$, respectively, and let~$d(P):=e(P)/c(P)^3$ be its density.
	Note that for~$p\in P$,~$\{p\}$ can be viewed as a palette itself and as usual we omit the parentheses when writing~$C(p)$ etc.
	We say that~$P$ is \emph{non-degenerate} if every pattern of~$P$ is non-degenerate, i.e., for every~$p \in P$, it holds that~$c(p)=3$.
	Given a subset~$U\subseteq C(P)$ of colors, let~$P[U]$ be the induced subpalette on~$U$.
	That is, the palette with~$C(P[U])=U$ and~$P[U]:=\{p\in P:\: C(p)\subseteq U\}$.
	
	Given two palettes~$P$ and~$Q$, a \emph{(palette) homomorphism} from~$P$ to~$Q$ is a map $\psi:C(P)\to C(Q)$ such that for every pattern~$p=(c_1,c_2,c_3)\in P$, we have~$\psi(p)=(\psi(c_1),\psi(c_2),\psi(c_3))\in Q$.
	As with hypergraphs, we usually do not distinguish between isomorphic palettes.
	If there is an injective homomorphism from~$P$ to~$Q$, we also say that~$P$ is contained in~$Q$ (or that $P$ is a subpalette of $Q$), denoted by~$P\subseteq Q$.
	We say that a palette~$Q$ is a \emph{blow-up} of a palette~$P$ if it can be obtained from~$P$ by replacing every color with some number of copies of itself.
	More formally, we say that~$Q$ is a blow-up of~$P$ with \emph{partition structure}~$C(Q)=\bigcup_{c\in C(P)}V_c$ for some pairwise disjoint sets~$V_c$,~$c\in C(P)$, if $$Q=\{(x_1,x_2,x_3):x_i\in V_{c_i}\text{ for all }i\in[3]\text{ and }(c_1,c_2,c_3)\in P\}\,.$$
	Note that~$P$ is contained in a blow-up of~$Q$ if and only if there is a homomorphism from~$P$ to~$Q$.
    In the case that there is not only a homomorphism from~$P$ to~$Q$ but an isomorphism, we denote this by~$P \cong Q$.
	
	Given a~$3$-graph~$F=(V,E)$ on~$n$ vertices, we say that~$P$ \emph{paints}~$F$ if there exists a total ordering~$\strictif$ of~$V$ and a coloring~$\chi:V^{(2)}\to C(P)$ such that for every edge~$xyz\in E$ with~$x\strictif y \strictif z$, we have
	\begin{align}\label{eq:paints}
		(\chi(xy),\chi(xz),\chi(yz))\in P.   
	\end{align} 
	Sometimes we refer to such a tuple~$(\strictif,\chi)$ as a painting of~$F$ (using~$P$).
	If there is no painting of~$F$ using~$P$, we say that $P$ does not paint $F$, or alternatively, that $P$ is $F$-deficient.
	We say that~$P$ does not paint a family~$\cF$, or is~$\cF$-deficient, if~$P$ does not paint~$F$ for every~$F\in \cF$. 
	
	Palettes were introduced in \cites{EH:72, Ro:86, R:20} in the context of describing a general lower bound construction for the uniform Tur\'{a}n density, called the \emph{Palette construction}.
	Given a family~$\cF$ of~$3$-graphs and a palette~$P\subseteq [t]^3$ on~$t$ colors such that~$P$ is~$\cF$-deficient, we construct an~$\cF$-free hypergraph~$H$ with vertex set~$[n]$ as follows.
	Let~$x_1,\ldots, x_t \in [0,1]$ with~$\sum_{i=1}^t x_i=1$, and let~$\chi:[n]^{(2)}\rightarrow [t]$ be an auxiliary coloring defined probabilistically by coloring each pair independently with
	\begin{align*}
		\PP(\chi(ab)=i)=x_i,\quad \forall ab \in [n]^{(2)},\, i\in [t].
	\end{align*}
	The edges of the hypergraph $H$ are defined using the auxiliary coloring $\chi$ as follows $$H:=\{abc\in [n]^{(3)}:\: a<b<c\text{ and }(\chi(ab),\chi(ac),\chi(bc))\in P\}\,.$$
	One can observe that by definition,~$P$ paints~$H$ and therefore paints any subgraph of~$H$.
	Hence, since~$P$ is~$\cF$-deficient,~$H$ is~$\cF$-free.
	Moreover, the probability that a triple~$abc\in[n]^{(3)}$ is an edge in~$H$ is given by
	\begin{align*}
		\PP(abc\in H)=\sum_{(i,j,k)\in P} x_ix_jx_k.
	\end{align*}
	It can be shown, by a standard application of concentration inequalities, that for each~$\eta$, with high probability the hypergraph~$H$ is~$(d,\eta)$-dense for~$d=\sum_{(i,j,k)\in P} x_ix_jx_k$ when~$n$ is taken sufficiently large. 
	
	The construction above naturally motivates the next definition.
	Denote the standard $(r-1)$-simplex by
	\begin{align*}
		\bbS_r:=\{(x_1,\ldots,x_r) \in [0,1]^r:\: x_1+\ldots+x_r=1\}.
	\end{align*}
	
	\begin{dfn}
		A \emph{weighting of a palette} $P$ is a vector~$\bx=(x_i)_{i \in C(P)}\in\bbS_{c(P)}$.
		Given a palette~$P$ with a weighting~$\bx$, set $$\lambda_P(\bx):=\sum_{(i,j,k) \in P}x_ix_jx_k\,.$$ 
		We define the \emph{palette Lagrangian} $\Lambda_P$ of $P$ as
		\begin{equation*}
			\Lambda_P:= \max_{\bx \in \bbS_{c(P)}}\lambda_P(\bx).
		\end{equation*}
	\end{dfn}
	
	As defined in Section \ref{sec:intro}, the set of values obtained as the Lagrangian of a palette is denoted by~$\Lambda_{\pal}:= \{\Lambda_P  : \: P \text{ is a palette} \}$.
	A consequence of the construction shown above is the following folklore result.
	
	\begin{fact}\label{lem:palette_lb}
		Let $P$ be a palette and let $\cF$ be a family of~$3$-graphs such that $P$ is $\cF$-deficient. Then~$\pi_{\vvv}(\cF) \geq \Lambda_P$.
	\end{fact}
	
	A folklore conjecture in the area, stated formally in \cite{R:20}, is that every lower bound for a uniform Tur\'{a}n density should be obtained by a palette construction.
	In a recent breakthrough, Lamaison \cite{L:24} proved an approximate version of the conjecture showing that for every family $\cF$, the Tur\'{a}n density $\pi_{\vvv}(\cF)$ can be approximated by a sequence of palette Lagrangians.
	In this paper, we prove in some sense the converse of the conjecture: Every palette Lagrangian is the uniform Tur\'{a}n density of some finite family $\cF$.
	
	The proof of our main result proceeds by transferring the original problem to a Tur\'{a}n-type problem for palettes.
	For this it is crucial that the property of being~$\cF$-deficient is invariant under homomorphisms.
	
	\begin{fact}\label{fact:hom_invariant}
		Let $P$ and $Q$ be two palettes such that there exists a homomorphism~$\psi:Q \rightarrow P$, and let $\cF$ be a family of $3$-graphs. If $P$ is $\cF$-deficient, then $Q$ is $\cF$-deficient.
	\end{fact}
	
	\begin{proof}
		Suppose, to the contrary, that~$Q$ paints some~$F \in \cF$.
		Then there exists an ordering~$\strictif$ of $V(F)$ and a coloring~$\chi: V(F)^{(2)} \rightarrow C(Q)$ satisfying (\ref{eq:paints}). Hence, by definition, the ordering~$\strictif$ and the coloring~$\psi \circ \chi: V(F)^{(2)}\rightarrow C(P)$ witness that~$P$ paints~$F$, contradicting the assumption that~$P$ is $\cF$-deficient.
	\end{proof}
	
	The first consequence of Fact \ref{fact:hom_invariant} is that the property of being $\cF$-deficient is closed under taking subpalettes.
	In particular, this allows us to define the palette Tur\'{a}n density of families of~$3$-graphs.
	Given a family~$\cF$ of~$3$-graphs, we define the \emph{palette extremal number of $\cF$} by
	\begin{align*}
		\expal(n,\cF):=\max\{e(P):\: P \text{ is an $\cF$-deficient palette with }c(P)=n\},
	\end{align*}
	i.e., the maximum number of patterns an $\cF$-deficient palette with $n$ colors can have.
	Similarly as for hypergraphs, one can show that the quantity $\expal(n,\cF)/n^3$ converges to a limit.
	
	\begin{prop}\label{prop:limit}
		The limit $\lim\limits_{n\rightarrow \infty} \frac{\expal(n,\cF)}{n^3}$ exists.
	\end{prop}
	
	\begin{proof}
		Recall that a palette $P$ is non-degenerate if every pattern of $P$ contains $3$ distinct colors. 
		For a family $\cF$ of $3$-graphs and an integer~$n\geq3$, we define the parameter $g(n,\cF)$ by
		\begin{align*}
			g(n,\cF):=\max\{e(P):\: P \text{ is a non-degenerate $\cF$-deficient palette with }c(P)=n\}.
		\end{align*}
		We claim that the quantity~$\frac{g(n,\cF)}{n(n-1)(n-2)}$ is non-increasing.
		Indeed, let~$P$ be a non-degenerate~$\cF$-deficient palette on~$n+1$ colors that has~$g(n+1,\cF)$ patterns.
		Take a random subset~$U\subseteq C(P)$ of size~$n$ and let~$P[U]$ be the induced subpalette on this set of colors.
		Then, by Fact~\ref{fact:hom_invariant}, we have that~$P[U]$ is~$\cF$-deficient and consequently~$e(P[U])\leq g(n,\cF)$.
		Hence,
		\begin{align*}
			g(n,\cF)\geq \EE(e(P[U]))=\frac{n-2}{n+1}g(n+1,\cF)\,,
		\end{align*}
		which implies that $\frac{g(n+1,\cF)}{(n+1)n(n-1)}\leq \frac{g(n,\cF)}{n(n-1)(n-2)}$.
		Since every non-negative non-increasing sequence has a limit, we obtain that~$\lim\limits_{n\rightarrow \infty} \frac{g(n,\cF)}{n(n-1)(n-2)}$ exists.
		The proposition now follows by the simple observation that~$g(n,\cF)\leq \expal(n,\cF)\leq g(n,\cF)+3n^2$.
	\end{proof}
	
	We define the limit obtained in Proposition~\ref{prop:limit} as the \emph{palette Turán density of a family}~$\cF$,
	\begin{align*}
		\pi_{\pal}(\cF):=\lim_{n\rightarrow \infty} \frac{\expal(n,\cF)}{n^3}\,.
	\end{align*}
	A consequence of the work in \cites{R:20, L:24} is that $\pi_{\pal}(\cF)=\pi_{\vvv}(\cF)$ for finite $\cF$ (see Section \ref{sec:main} for more details). Therefore, to show Theorem \ref{thm:main}, it suffices to show that every palette Lagrangian is attained as a palette Turán density of a finite family.

	By taking the uniform weighting~$\bx=(x_i)_{i\in C(P)}$ defined by~$x_i=1/c(P)$, we obtain that  
	\begin{align}\label{eq:lambda_density}
		d(P)\leq \Lambda_P\,.
	\end{align}
	Moreover, if $\by=(y_i)_{i\in C(P)}$ is an optimal weighting for the palette~$P$, then one can approximate~$\Lambda_P$ by taking a sequence of blow-ups~$\{S^{(\ell)}\}_{\ell\in\mathds{N}}$ of~$P$ with partition structure~$C(S^{(\ell)})=\bigcup_{i\in C(P)} V_{i}^{(\ell)}$ such that~$|V_i^{(\ell)}|/c(S^{(\ell)}) \rightarrow y_i$.
	Conversely, every blow-up of~$P$ yields a weighting for~$P$.
	Hence, it follows that
	\begin{align}\label{eq:lambda_blowup}
		\Lambda_P=\lim_{n\to\infty}\max\{d(S):\: S \text{ is a blow-up of } P\text{ with }c(S)=n\}.
	\end{align}
	
	The equality in~(\ref{eq:lambda_blowup}) and the fact that the property of being~$\cF$-deficient is blow-up invariant (Fact~\ref{fact:hom_invariant}) hint at a way to obtain a palette Lagrangian as a palette Turán density: finding a finite family~$\cF$ such that the extremal constructions for~$\expal(n,\cF)$ are exactly the family of blow-ups of~$P$.
	Unfortunately, such a family does not always exist (see Section~\ref{sec:ramsey}), but by adding an extra family of extremal constructions with the same palette Lagrangian, one can achieve such a goal.
	Given a palette~$P$, we define the \emph{reverse palette}~$\rev(P)$ of~$P$ as 
	\begin{align*}
		\rev(P)=\{(c,b,a):\: (a,b,c)\in P\}.
	\end{align*}
	That is,~$\rev(P)$ is the palette obtained by reversing the order of the patterns of~$P$. We say that a palette~$P$ is \emph{reduced} if for every proper subpalette~$Q\subsetneq P$, we have~$\Lambda_Q<\Lambda_P$. We follow the graph convention and let~$\EXpal(n,\cH) = \{Q : \: e(Q) = \expal(n,\cH)\text{ and }c(Q)=n \}$ be the set of extremal palettes.
	The following is the main technical result of this paper.
	
	\begin{theorem}\label{thm:expal_eq_blowup}
		Let~$P$ be a reduced palette.
		There exists a finite family~$\cH$ such that~$P$ is~$\cH$-deficient and for all~$n\in\mathds{N}$
            \begin{align}\label{eq:maintech}
			\EXpal(n,\cH) \subseteq \{Q :\: Q \text{ is a blow-up of $P$ or a blow-up of $\rev(P)$ and }c(Q)=n\}\,.
		\end{align}
            In particular, it immediately follows that
		\begin{align*}
			\expal(n,\cH) = \max \{e(Q) :\: Q \text{ is a blow-up of $P$ with }c(Q)=n\}\,.
		\end{align*}
	\end{theorem}
	
	We remark that $P$ is not necessarily isomorphic to $\rev(P)$.
	As a simple example, consider $P=\{(1,2,3), (1,3,2)\}$ and $\rev(P)=\{(2,3,1),(3,2,1)\}$.
	One can verify that in this case, $P\not\cong \rev(P)$.

	\subsection*{Organization}
	The paper is organized as follows.
	The proof of Theorem \ref{thm:expal_eq_blowup} relies on three main ingredients.
	The first is a palette variant of the removal lemma introduced in~\cite{AFKS:00}, so in Section~\ref{sec:regularity} we use a regularity lemma for palettes to prove counting and removal lemmata for palettes painting graphs.
	The second is a structural Ramsey result in Section~\ref{sec:ramsey}, dedicated to the problem of distinguishing palettes based on the graphs they can paint.
	The third component is a stability argument based on the work of~\cite{P:12} (Sections~\ref{sec:properties} and~\ref{sec:stability}).
	For a brief outline of the proof of Theorem~\ref{thm:expal_eq_blowup}, the reader may refer to the introduction of Section~\ref{sec:stability}.
	Finally, in Section~\ref{sec:main}, we present a proof of Theorem~\ref{thm:main}.
	
	\section{Regularity lemma for palettes}\label{sec:regularity}

        For graphs, the following infinite removal lemma was shown in \cite{AS:05} (and a hypergraph analogue in \cite{ARS:07}).

        \begin{lemma}\label{lem:graph_removal}
            Given a (possibly infinite) family of graphs~$\cF$ and~$\alpha>0$, there are~$M,n_0 \in \NN $ and~$\beta >0$ so that the following holds for every graph~$G$ on~$n \geq n_0$ vertices. If, for every~$F \in \cF$ with~$v(F) \leq M$,~$G$ contains fewer than~$\beta n^{v(F)}$ copies of~$F$, then~$G$ can be made~$\cF$-free by removing at most~$\alpha n^2$ edges.
        \end{lemma}

        The aim of this section is prove a version of this lemma for palettes.
        Instead of counting the number of copies of some~$F\in\cF$, we need to count the number of ways that a palette~$P$ paints~$F$.
        This is made precise in the following definition.
	
	\begin{dfn}\label{dfn:paint_count}
		Let~$F$ be a~$3$-graph and~$P$ a palette. 
        The number of ways that~$P$ paints~$F$ is defined as the number of maps~$\phi\colon \partial F \to C(P)$ for which there exists a total ordering~$<$ of~$V(F)$ such that~$(<,\phi)$ is a painting of~$F$.
	\end{dfn}
	
        We are now prepared to state the aforementioned removal lemma.
        
    	\begin{lemma}[Palette Removal Lemma]\label{lem:palette_removal}
		Given a (possibly infinite) family of~$3$-graphs~$\cF$ and~$\alpha>0$, there are~$M=M_{\ref{lem:palette_removal}},N=N_{\ref{lem:palette_removal}}\in \mathds N$ and~$\beta = \beta_{\ref{lem:palette_removal}}>0$ such that the following holds for every palette~$P$ on~$n\geq N$ colors.
		If, for every~$s\in [\binom{M}{2}]$,~$P$ paints the~$3$-graphs~$F\in \cF$ with~$|\partial F| = s$ and~$v(F) \leq M$ in less than~$\beta n^{s}$ ways, then there is an~$\cF$-deficient palette~$Q\subseteq P$ with~$|P\setminus Q| \leq \alpha n^3$.  
	\end{lemma}

        The proof of Lemma~\ref{lem:palette_removal} is given at the conclusion of this Section. Similar to the proof of Lemma~\ref{lem:graph_removal} for graphs, it will require a regularity theory for palettes. In some ways it is helpful to consider a palette merely as essentially an oriented~$3$-graph, since the number of degenerate patterns is~$O(n^2)$. 
        Let us define what it means for sets of colors in a palette to be~$\epsilon$-regular.
        Let~$P$ be a palette and suppose that~$W_1,W_2,W_3 \subseteq C(P)$ are non-empty.
        We set~$E(W_1,W_2,W_3)=(W_1\times W_2\times W_3)\cap P$ and~$e(W_1,W_2,W_3)=\vert E(W_1,W_2,W_3)\vert$\footnote{For ease of notation we suppress the dependency on~$P$ when it is clear from the context.}.
		Then the density of~$P$ induced on~$(W_1,W_2,W_3)$ is given by $$d(W_1,W_2,W_3) := \frac{e(W_1,W_2,W_3)}{|W_1||W_2||W_3|}\,.$$
        
	\begin{dfn}\label{dfn:palette_regularity}
		We call~$(V_1,V_2,V_3)$~$\epsilon$-regular if, for all~$W_1 \subseteq V_1,W_2 \subseteq V_2,W_3 \subseteq V_3$ with~$|W_i| \geq \epsilon |V_i|$,~$i\in[3]$, we have $$|d(W_1,W_2,W_3)-d(V_1,V_2,V_3)| \leq \epsilon\,.$$
	\end{dfn}
    
        As discussed above, the most important feature of the above definition is that it is sensitive to order, and the regularity of~$(V_1,V_2,V_3)$ has no bearing on the regularity of~$(V_2,V_1,V_3)$.
        However, we can derive many properties of these regular color sets by applying corresponding results in the unoriented hypergraph setting (see Lemma \ref{lem:palette_counting} below), so we give this definition as well.
        Let~$H=(V,E)$ be a~$3$-graph and suppose that~$X_1,X_2,X_3 \subseteq V$ are non-empty.
        We set $$E(X_1,X_2,X_3)=\{xyz\in E:\:x\in X_1,y\in X_2,z\in X_3\}$$ and~$e(X_1,X_2,X_3)=\vert E(X_1,X_2,X_3)\vert$.
		Then the density of~$H$ induced on~$X_1,X_2,X_3$ is given by $$d(X_1,X_2,X_3) := \frac{e(X_1,X_2,X_3)}{|X_1||X_2||X_3|}\,.$$

        \begin{dfn}\label{dfn:hypergaph_regularity}
		Suppose~$H=(V,E)$ is a~$3$-graph and that~$X_1,X_2,X_3 \subseteq V(H)$ are non-empty.
		We call~$X_1,X_2,X_3$ $\epsilon$-regular if, for all~$Y_1 \subseteq X_1,Y_2 \subseteq X_2,Y_3 \subseteq X_3$ with~$|Y_i| \geq \epsilon |X_i|$,~$i\in[3]$, we have $$|d(Y_1,Y_2,Y_3)-d(X_1,X_2,X_3)| \leq \epsilon\,.$$
	\end{dfn}

	As usual we will be interested in partitioning~$C(P)$ into a large (but bounded) number of parts so that most~$(V_i,V_j,V_k)$ are regular, so we also need the standard notions of equipartitions and refinements.
	\begin{dfn}
		Given a set~$C$, an equipartition~$\cA$ of~$C$ is a partition~$C=\bigdcup_{i\in[t]} V_i$, so that~$| |V_i| - |V_{i'}|| \leq 1$ for all $i,i' \in [t]$.
		A refinement of~$\cA$ is an equipartition~$\cB = \bigdcup_{i\in[t]}\bigdcup_{j\in[\ell]}V_{i,j}$ with~$V_{i,j} \subseteq V_i$ for all~$i \in [t]$ and~$j \in [\ell]$.
        We also identify~$\cA$ with the family of partition classes, i.e.,~$\cA=\{V_i:\:i\in[t]\}$.
	\end{dfn}

    Our palette analogue of the well-known Szemer\'edi regularity lemma is the following.

	\begin{theorem}\label{thm:palette_regularity}
		For all~$\epsilon>0$ and~$m \in \NN$ there exist~$M=M_{\ref{thm:palette_regularity}},N=N_{\ref{thm:palette_regularity}} \in \NN$ so that given any palette~$Q$ with~$c(Q) \geq N$ there is an equipartition~$\cA = \{V_i : i \in [t] \}$ of~$C(Q)$ so that
		\begin{enumerate}
			\item $m \leq t \leq M$ and
			\item the ordered triple~$(V_i,V_j,V_k)$ is~$\epsilon$-regular for all but~$\epsilon t^3$ of~$(i,j,k) \in [t]^3$.
		\end{enumerate}
        Moreover, given some equipartition~$\cA_0$ of~$C(Q)$ with at most~$m$ parts, there is an~$\cA$ as above which refines~$\cA_0$.
	\end{theorem}
	In our application we need a strengthening of this; in particular, the $\epsilon$-regularity obtained should be allowed to depend on the number of parts $t$ in the partition, as follows.
  
	\begin{cor}\label{cor:iterated_palette_regularity_app}
		For all non-increasing maps~$\mathcal{E}: \NN \to (0,1]$ and~$m\in\mathds{N}$ there are~$M=M_{\ref{cor:iterated_palette_regularity_app}},N=N_{\ref{cor:iterated_palette_regularity_app}}$ and $\delta= \delta_{\ref{cor:iterated_palette_regularity_app}} >0$ so that given any palette~$Q$ with $c(Q)=n\geq N$ there is are an equipartition~$\cA=\{V_i : i \in [t] \}$ of $C(Q)$ and an equipartition~$\cA' = \{U_i : i \in [t] \}$ of some subset of~$C(Q)$ so that:
		\begin{enumerate}[label = \rmlabel]
			\item\label{it:num_parts} $m\leq t \leq M$,
			\item\label{it:model_size} $U_i  \subseteq V_i$ with $|U_i  | \geq \delta n$,
			\item\label{it:model_regularity} the ordered triple~$(U_{i},U_j,U_k)$ is $\cE(t)$-regular for all~$(i,j,k) \in [t]^3$, and
			\item\label{it:model_accuracy} we have $|d(U_i,U_j,U_k)-d(V_i,V_j,V_k)|< \cE(0)$ for all but $\cE(0)t^3$ of $(i,j,k) \in [t]^3$.
		\end{enumerate}
	\end{cor}
	The proof of Theorem~\ref{thm:palette_regularity} follows along the standard technique of iterated refinement used for graph and weak hypergraph regularity, and the method to obtain Corollary~\ref{cor:iterated_palette_regularity_app} from Theorem~\ref{thm:palette_regularity} 
	was developed in \cite{AFKS:00} for graph-testing problems. For the interested reader we include the proofs of Theorem~\ref{thm:palette_regularity} and Corollary~\ref{cor:iterated_palette_regularity_app} in Appendix~\ref{app:regularity}.

	Before we prove Lemma~\ref{lem:palette_removal}, we show a counting lemma for the number of ways in which a palette paints a~$3$-graph.
    We use the following result, which counts the number of copies of a linear hypergraph inside a regularly partitioned hypergraph.
    We state only the~$3$-uniform case.

        \begin{lemma}[\protect{\cite[Lemma 10]{countinglemmalinear}}]\label{lem:linear_counting}
            Given~$\gamma, d_0>0$ and~$\ell\in \mathds N$ there are~$\eps=\eps_{\ref{lem:linear_counting}}(\gamma,d_0,\ell)$ and~$N=N_{\ref{lem:linear_counting}}=(\gamma,d_0,\ell)$ such that the following holds. 
		Let~$F$ be a linear~$3$-graph with~$V(F)=[\ell]$ and~$H$ an~$\ell$-partite~$3$-graph on parts~$V_1,\dots,V_\ell$ with~$|V_i| \geq N$ for each~$i \in [\ell]$.
        Suppose that~$\{V_i \}_{i \in f}$ is~$\epsilon$-regular with density~$d_f \geq d_0$ for every~$f \in E(F)$.
        Then the number of copies of~$F$ in~$H$ that map each vertex~$i\in[\ell]$ to~$V_i$ is at least~$(1-\gamma)d_0^{e(F)}\prod_{i \in [\ell]} |V_{i}|$.
        \end{lemma}
        By creating a hypergraph which captures the structure of~$P$,  we can obtain a similar statement estimating the number of ways in which~$P$ paints a given~$3$-graph~$F$. 
	\begin{lemma}\label{lem:palette_counting}
		Given~$\gamma, d_0>0$ and~$s\in \mathds N$ there are~$\eps=\eps_{\ref{lem:palette_counting}}(\gamma,d_0,s)$ and~$N=N_{\ref{lem:palette_counting}}=(\gamma,d_0,s)$ such that the following holds. 
		Let~$F$ be a~$3$-graph with~$|\partial F| = s$ and~$P$ be a palette whose colours are partitioned into~$C(P) = V_1\dcup \dots \dcup V_s$ where~$|V_i|\geq N$ for every~$i\in [s]$. 
		Suppose there is an ordering~$<$ of the vertices of~$F$ and a map~$\phi\colon \partial F \to [s]$ such that for each~$uvw\in E(F)$ with~$u<v<w$, the triple~$(V_{\phi(uv)}, V_{\phi(uw)}, V_{\phi(vw)})$ is~$\eps$-regular in~$P$ with density at least~$d_0$. 
		Then~$F$ is painted by~$P$ in at least~$(1-\gamma)d_0^{e(F)}\prod_{uv\in \partial F} |V_{\phi(uv)}|$ different ways.
	\end{lemma}
	
	\begin{proof}
		We write~$uvw$ for~$\{u,v,w\}$ whenever~$u < v < w$ for~$u,v,w \in V(F)$.
        Let~$F^\partial$ be a~$3$-graph with vertex-set~$\partial F$ and all edges of the form~$\{uv,uw,vw\}$ where~$uvw\in E(F)$.
		Note that~$F^{\partial}$ is linear and~$e(F) = e(F^\partial)$.
        For each~$uv \in \partial F$ let~$V_i^{uv}$ be a copy of~$V_i$, and label the copy of~$a \in V_i$ as~$a^{uv} \in V_i^{uv}$.
        Let~$H$ be a~$3$-graph with vertices~$V(H) = \bigdcup\limits_{\substack{uv \in \partial F \\ i\in [k]}}V_i^{uv}$ and edges
		$$E(H) = \{a^{uv}b^{uw}c^{vw} \in V(H)^{(3)}  \colon uv,uw,vw \in \partial F, \text{ and } (a,b,c)\in P\}\,.$$
        In words,~$H$ is obtained by taking~$s=|\partial F|$ copies of~$C(P)$, indexed by~$uv \in \partial F$, and then adding a~$3$-edge~$xyz$ only when~$x$,~$y$, and~$z$ are part of distinct copies and, when viewed as an ordered triple in~$C(P)$, they give a pattern in~$P$.
        Now if~$(V_i,V_j,V_k)$ is~$\epsilon$-regular in the palette~$P$ (Definition \ref{dfn:palette_regularity}), then the vertex sets~$(V_i^{uv},V_j^{uw},V_k^{vw})$ are~$\epsilon$-regular in~$H$ for each~$uv,uw,vw \in \partial F$ (Definition \ref{dfn:hypergaph_regularity}). 	
        Consider in particular the~$s$ sets given by~$V_{\phi(uv)}^{uv}$ for~$uv \in \partial F$.
        Since~$F^\partial$ is linear, Lemma~\ref{lem:linear_counting} yields at least~$(1-\gamma)d_0^{e(F^{\partial})}\prod_{uv \in\partial F}|V_{\phi(uv)}^{uv}|$ copies (or embeddings) of~$F^{\partial}$ in~$H$ that map each~$uv\in V(F^{\partial})$ to some vertex in~$V^{uv}_{\phi(uv)}$.
        Write~$\Psi$ for the collection of such embeddings.
        Let~$\chi:V(H) \to C(P)$ be the projection map which sends~$a^{uv}$ to~$\chi(a^{uv}) = a$.
        Note that for every~$\psi\in\Psi$, the tuple~$(<,\chi\circ\psi)$ is a painting of~$F$ using~$P$ and that any two distinct embeddings~$\psi,\psi'\in\Psi$ get projected to distinct paintings, whence the number of such paintings is at least~$(1-\gamma)d_0^{e(F)}\prod_{uv\in \partial F}|V_{\phi(uv)}|$.
	\end{proof}

	We are now ready to prove the removal lemma using the counting lemma.

	\begin{proof}[Proof of Lemma \ref{lem:palette_removal}]

		Let~$\cP_t(\cF)$ be the set of all palettes with color set~$[t]$ which paint at least one~$F \in \cF$.
		
		We define the map~$\fv_{\cF}\colon \mathds N \to \mathds N_{\geq0}$ by
		\begin{align*}
			\fv_{\cF}(t) 
			&= 
			\max_{R\in \cP_t(\cF)}
			\min \{|V(F)| \colon F\in \cF \text{ and } R \text{ paints }F\}       
		\end{align*}
        (and~$\fv_{\cF}(t)=0$ if~$\cP_t(\cF)=\emptyset$).
        Since~$\cP_t(\cF)$ is finite for every~$t$, the maximum exists.
		The idea is the following.
        Given a palette~$P$ that paints every small~$F\in\cF$ in few ways, we apply the palette regularity lemma.
        From this we obtain a ``reduced'' palette~$R$ with a constant number of colors~$t$ and a ``cleaned'' palette~$Q$ similar to~$P$.
        If~$Q$ would still paint some~$F\in\cF$, then~$R$ would paint~$F$ and hence - using the definition of~$\fv_{\cF}$ - some~$F'\in\cF$ with few vertices.
        Then the counting lemma entails that~$P$ must paint~$F'$ as well.
		Let us formalize this argument.
        
        Given~$\alpha>0$, define the function~$\cE\colon \mathds N_0 \to [0,1]$ by 
		$$\cE(t) = 
		\begin{cases}
			2\alpha/9  &\text{if~$t=0$ and} \\
			\epsilon_{\ref{lem:palette_counting}}(1/2,2\alpha/9, \binom{\fv_{\cF}(t)}{2}) \qquad &\text{if~$t\geq 1$}\,.
		\end{cases}$$
        Note that we may assume that the functions~$N_{\ref{lem:palette_counting}}(\gamma, d_0,s)$ and~$\epsilon_{\ref{lem:palette_counting}}(\gamma, d_0,s)$ are non-decreasing and non-increasing in~$s$, respectively.
		Let $M_{\ref{cor:iterated_palette_regularity_app}},N_{\ref{cor:iterated_palette_regularity_app}}, \delta=\delta_{\ref{cor:iterated_palette_regularity_app}}$ be the constants given by an application of Corollary~\ref{cor:iterated_palette_regularity_app} with~$\cE$ and
		\begin{align}\label{eq:m}
			m> 9/\alpha\,.   
		\end{align}
		Finally, take
		$$
		M = \max_{r\in [M_{\ref{cor:iterated_palette_regularity_app}}]}\fv_{\cF}(r)\,,
            \quad
            N = \max\big\{N_{\ref{cor:iterated_palette_regularity_app}}, \frac{1}{\delta} N_{\ref{lem:palette_counting}}\big(1/2,2\alpha/9, \binom{M}{2}\big)\big\}\,,
		\quad
		\beta \leq \frac{(\frac{2\alpha}{9})^{\binom{M}{3}}\delta^{\binom{M}{2}}}{2}\,.$$
		
		Now we argue that these choices of~$M$,~$N$, and~$\beta$ have the desired property.
        Given a palette~$P$ with~$n\geq N$ colors, apply (the conclusion of) Corollary~\ref{cor:iterated_palette_regularity_app} to obtain equipartitions~$\cA=\{V_i : i \in [t] \}$ and~$\cA' = \{U_i : i \in [t] \}$ satisfying \ref{it:num_parts}-\ref{it:model_accuracy}.
        Note that in particular,~$\vert U_i\vert\geq\delta N\geq N_{\ref{lem:palette_counting}}\big(1/2,2\alpha/9, \binom{M}{2}\big)$.
		We produce a `reduced' palette~$R$ with~$C(R)=[t]$ by including the pattern~$(i,j,k)$ in~$R$ if 
		\begin{enumerate}[label=\nlabel]
			\item\label{it:cross} $i,j,k$ are pairwise distinct,
			\item\label{it:similar} $|d(U_{i},U_{j},U_{k})-d(V_{i},V_{j},V_{k})| \leq \cE(0)=2\alpha/9$, and
			\item\label{it:density} $d(U_{i},U_{j},U_{k}) > \cE(0)=2\alpha/9$.
		\end{enumerate}
		Let~$Q\subseteq P$ be the `cleaned' version of~$P$, where we delete all edges $(a,b,c)\in P$ with $a \in V_{i},b \in V_{j},c \in V_{k}$ when~$(i,j,k)\notin R$.
		In this way, it is easy to see that~$Q$ is contained in a blow-up of~$R$.
		
		As in many other applications of the regularity lemma, it is not hard to check that
		\begin{align}
			|P\setminus Q| \leq \alpha n^3\,.
		\end{align}
		Indeed, there are at most~$3t^2$ triplets of indices not satisfying~\ref{it:cross}. 
		Thus, due to~\eqref{eq:m}, at most~$3t^2(n/t)^3 \leq 3n^3/t \leq \alpha n^3/3$ patterns are deleted in this way. 
		By Corollary~\ref{cor:iterated_palette_regularity_app} Part~\ref{it:model_accuracy} there are at most~$\cE(0)t^3 = 2\alpha t^3/9$ triplets~$(i,j,k)$ such that~$|d(U_i,U_j,U_k)-d(V_i,V_j,V_k)|> \cE(0)$, meaning that at most~$n^3/t^3\cdot 2\alpha t^3/9\leq 2\alpha n^3/9$ patterns need to be deleted to ensure~\ref{it:similar}. 
		Finally, using that for the remaining triplets of indices~\ref{it:similar} holds, in~\ref{it:density} we delete at most~$4\alpha n^3/9$ edges.
		
		Suppose for a contradiction that~$Q$ paints a hypergraph ~$F \in\cF$. 
		Since~$Q$ is contained in a blow-up of~$R$, it follows that~$R$ paints~$F$ as well.
        Therefore, keeping in mind that~$c(R)=t\leq M_{\ref{cor:iterated_palette_regularity_app}}$, there is some hypergraph~$F' \in \cF$ painted by~$R$ with the additional property that~$v(F') \leq \fv_{\cF}(t) \leq M$.
		Let~$s=|\partial F'|\leq \binom{M}{2}$ and let~$\phi :\partial F'\to C(R) = [t]$ be the coloring given by definition of painting (we leave the vertex ordering to be implicit).
		Due to~\ref{it:cross}-\ref{it:density} and~\ref{it:model_regularity} in Corollary~\ref{cor:iterated_palette_regularity_app}, the map~$\phi$ satisfies the conditions of Lemma~\ref{lem:palette_counting}, when restricted to the subpalette~$U_1 \dcup \dots \dcup U_s$, which implies that~$Q$ paints~$F'$ in more than~$\beta n^s$ ways, a contradiction.
	\end{proof}

	\section{A Ramsey result}\label{sec:ramsey}
	
	Given two distinct $3$-graphs $G$ and $H$ such that neither is a subgraph of the other, there always exists a $3$-graph $F$ with the property that $G$ contains a copy of $F$, but $H$ is $F$-free.
	Indeed, one can take the $3$-graph $F$ to be $G$ itself.
	It is somewhat natural to ask if the same is true for palettes.
	That is, for what pairs of palettes~$P$ and~$Q$, does there exist a~$3$-graph~$F$ such that~$P$ paints~$F$, but~$Q$ does not paint~$F$?
	The goal of this section is to answer this question.
	We remark that similar considerations were mentioned in \cite{L:24}. 
	
	Given a palette $P$ on $n$ colors, recall that the \emph{reverse palette} $\rev(P)$ of $P$ as 
	\begin{align*}
		\rev(P)=\{(c,b,a):\: (a,b,c)\in P\}\,,
	\end{align*}
	that is, $\rev(P)$ is the palette obtained by reversing the order of the patterns of $P$.
	Note that a palette $P$ paints a $3$-graph $F$ if and only if $\rev(P)$ paints $F$.
	Indeed, this can be seen by taking the ordering of the vertices in which $P$ paints $F$ and reversing it.
	A consequence of this observation is that no graph can distinguish a palette $P$ from $\rev(P)$.
	It turns out that up to taking blow-ups,~$\rev(P)$ is the only palette for which there is no~$3$-graph that distinguishes it from~$P$.
	We remind the reader that a palette $Q$ is contained in a blow-up of $P$ if there exists a homomorphism $\psi:Q\rightarrow P$. The next lemma is the main result in this section.
	
	\begin{lemma}\label{lem:distinguish_palette}
		Let $P$ and $Q$ be palettes such that $Q$ is not contained in a blow-up of $P$ nor in a blow-up of $\rev(P)$.
		Then there exists a $3$-graph $F$ such that $P$ is $F$-deficient and~$Q$ paints $F$.
	\end{lemma}
	
	The proof of Lemma \ref{lem:distinguish_palette} is completely Ramsey-theoretical and relies on a result of Ne\v{s}et\v{r}il and R\"{o}dl \cite{NR:87} about Ramsey classes for ordered Steiner systems.
	We start with some preparations.
	An ordered $k$-graph $(H,<)$ is a pair where $H$ is a $k$-graph and $<$ is a total ordering of $V(H)$.
	Given two ordered hypergraphs~$(F,<)$ and~$(H,<)$, we say that~$(F,<)$ is a subgraph of~$(H,<)$ if there exists an injective order-preserving map~$\psi: V(F)\rightarrow V(H)$ that is a homomorphism, i.e., a map such that~$\psi(x)< \psi(y)$ for~$x< y$ and such that~$\psi(f)\in E(H)$ for every edge~$f\in E(F)$.
	Let $\mbinom{(H,<)}{(F,<)}$ denote the family of copies of $(F,<)$ in $(H,<)$.
	The next theorem shows that the class of ordered linear~$k$-graphs is edge-Ramsey (see also~\cite{BNRR:18}*{Lemma~2.12} and~\cite{RR:23}*{Corollary 3.12}).
	\begin{theorem}[\cite{NR:87}]\label{thm:ramseysteiner}
		Let $(G,<)$ be an ordered linear $k$-graph with~$k\geq 2$, and let~$r\geq 1$ be an integer. Then there exists an ordered linear~$k$-graph~$(H,<)$ and a family~$\cG\subseteq \mbinom{(H,<)}{(G,<)}$ of copies of $(G, <)$ in~$(H, <)$ satisfying the following statements:
		\begin{enumerate}
			\item[$(i)$] For any $r$-coloring of the edges of $H$, there exists a monochromatic copy $(G',<)\in \cG$.
			\item[$(ii)$] For any two distinct copies $(G',<), (G'',<) \in \cG$, it holds that either $|V(G')\cap V(G'')|\leq 1$ or $V(G')\cap V(G'')= e$ for an edge $e \in G_3$ for some~$(G_3,<)\in\mathcal{G}$.
		\end{enumerate}
	\end{theorem}
	
	We remark that for $k=2$, a linear $k$-graph is just a graph, and condition (ii) translates to the fact that two distinct copies $(G',<)$ and $(G'',<) \in \cG$ intersect either in an emptys set, a single vertex or in exactly one edge.
	Theorem \ref{thm:ramseysteiner} can be used to prove the following Ramsey result about systems of graphs. 
	We note that similar statements were obtained previously in~\cites{NR:89, AH:78}.
	
	\begin{prop}\label{prop:ramseysystem}
		Let $n, t\geq 1$ be integers, and let~$(G, <)$ be an ordered graph where $G=\bigcup_{i=1}^n G_i$ is the union of~$n$ pairwise edge-disjoint ordered graphs with vertex set~$V(G)$.
		Then there exists an ordered graph $(H,<)$ where $H=\bigcup_{i=1}^n H_i$ is the union of~$n$ pairwise edge-disjoint ordered graphs with vertex set~$V(H)$ such that any $t$-coloring of~$H$ yields a set $X\subseteq V(H)$ with the following properties:
		\begin{enumerate}
			\item[$(i)$] For $1\leq i \leq n$, we have $(H_i[X],<)\cong(G_i, <)$.
			\item[$(ii)$] For $1\leq i \leq n$, the graph $H_i[X]$ is monochromatic.
		\end{enumerate}
	\end{prop}

	\begin{proof}
		For the sake of brevity, throughout the proof we will omit the total ordering $<$ from the notation and denote an ordered graph $(H,<)$ by $H$.
		We inductively construct ordered graphs $A^0, \ldots, A^n$ such that for each $0\leq j \leq n$, the ordered graph $A^j=\bigcup_{i=1}^n A_i^j$ is the union of $n$ pairwise edge-disjoint ordered graphs on $V(A^j)$ as follows: Let $A^0=G$ and $A^0_i=G_i$ for $1\leq i \leq n$.
		Suppose now that for $1\leq j \leq n$ we have already defined the ordered graph~$A^{j-1}$ and want to define~$A^j$.
		Apply Theorem \ref{thm:ramseysteiner} to the ordered graph $A^{j-1}_j$ and~$t$ colors to obtain the ordered graph $A^{j}_j$ and a system of copies $\cA_j\subseteq \mbinom{A^j_j}{A^{j-1}_j}$ satisfying properties (i) and (ii) of the statement.
		In particular, for any two copies $B, B' \in \cA_j$ of $A^{j-1}_j$ we have that
		\begin{align}\label{eq:clean}
			|V(B)\cap V(B')|=1 \quad \text{or} \quad V(B)\cap V(B')\subseteq e 
		\end{align}
		for some edge~$e$ of some copy of~$A_j^{j-1}$ in~$\mathcal{A}_j$.
		For $i\neq j$, let $A^{j}_i$ be the ordered graph on~$V(A_j^j)$ (with the same total ordering $<$ on $V(A^j_j)$) obtained by adding a copy of the ordered graph $A_i^{j-1}$ to each set of vertices $V(B)$ with $B\in \cA_j$, see Figure \ref{fig:fig1}.
		By~\eqref{eq:clean}, each pair of copies~$B, B'\in\cA_j$ either intersects in a single vertex or in an edge of some copy of~$A^{j-1}_{j}$.
		Hence, keeping in mind that the graphs~$A_i^{j-1}$ are pairwise edge-disjoint, all~$A_i^j$ are pairwise edge-disjoint.
		We set $A^j=\bigcup_{i=1}^jA_i^j$. The key point of the construction is that the ordered graphs $A^0, \ldots, A^n$ satisfy the following claim.
		
		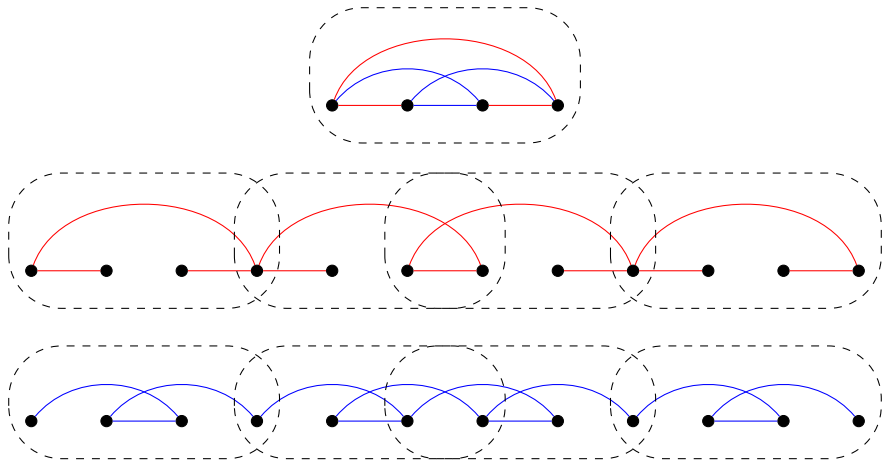
\begin{figure}[h]
			\centering
			{\hfil\begin{tikzpicture}
					\node[circle, draw, fill=black, inner sep=1.5pt] (x1) at (4,2.2) {};
					\node[circle, draw, fill=black, inner sep=1.5pt] (x2) at (5,2.2) {};
					\node[circle, draw, fill=black, inner sep=1.5pt] (x3) at (6,2.2) {};
					\node[circle, draw, fill=black, inner sep=1.5pt] (x4) at (7,2.2) {};
					
					\node[circle, draw, fill=black, inner sep=1.5pt] (v1) at (0,0) {};
					\node[circle, draw, fill=black, inner sep=1.5pt] (v2) at (1,0) {};
					\node[circle, draw, fill=black, inner sep=1.5pt] (v3) at (2,0) {};
					\node[circle, draw, fill=black, inner sep=1.5pt] (v4) at (3,0) {};
					\node[circle, draw, fill=black, inner sep=1.5pt] (v5) at (4,0) {};
					\node[circle, draw, fill=black, inner sep=1.5pt] (v6) at (5,0) {};
					\node[circle, draw, fill=black, inner sep=1.5pt] (v7) at (6,0) {};
					\node[circle, draw, fill=black, inner sep=1.5pt] (v8) at (7,0) {};
					\node[circle, draw, fill=black, inner sep=1.5pt] (v9) at (8,0) {};
					\node[circle, draw, fill=black, inner sep=1.5pt] (v10) at (9,0) {};
					\node[circle, draw, fill=black, inner sep=1.5pt] (v11) at (10,0) {};
					\node[circle, draw, fill=black, inner sep=1.5pt] (v12) at (11,0) {};
					
					\node[circle, draw, fill=black, inner sep=1.5pt] (u1) at (0,-2) {};
					\node[circle, draw, fill=black, inner sep=1.5pt] (u2) at (1,-2) {};
					\node[circle, draw, fill=black, inner sep=1.5pt] (u3) at (2,-2) {};
					\node[circle, draw, fill=black, inner sep=1.5pt] (u4) at (3,-2) {};
					\node[circle, draw, fill=black, inner sep=1.5pt] (u5) at (4,-2) {};
					\node[circle, draw, fill=black, inner sep=1.5pt] (u6) at (5,-2) {};
					\node[circle, draw, fill=black, inner sep=1.5pt] (u7) at (6,-2) {};
					\node[circle, draw, fill=black, inner sep=1.5pt] (u8) at (7,-2) {};
					\node[circle, draw, fill=black, inner sep=1.5pt] (u9) at (8,-2) {};
					\node[circle, draw, fill=black, inner sep=1.5pt] (u10) at (9,-2) {};
					\node[circle, draw, fill=black, inner sep=1.5pt] (u11) at (10,-2) {};
					\node[circle, draw, fill=black, inner sep=1.5pt] (u12) at (11,-2) {};
					
					\draw[red] (x1) -- (x2); 
					\draw[red][bend left=70] (x1) to (x4); 
					\draw[red](x3) -- (x4); 
					\draw[blue][bend left=50] (x1) to (x3);
					\draw[blue][bend left=50] (x2) to (x4); 
					\draw[blue](x2) -- (x3); 
					
					\draw[red] (v1) -- (v2); 
					\draw[red][bend left=70] (v1) to (v4); 
					\draw[red] (v3) -- (v4); 
					\draw[red] (v4) -- (v5); 
					\draw[red][bend left=70] (v4) to (v7); 
					\draw[red] (v6) -- (v7); 
					\draw[red][bend left=70] (v6) to (v9);
					\draw[red] (v8) -- (v9);
					\draw[red][bend left=70] (v9) to (v12); 
					\draw[red] (v9) -- (v10);
					\draw[red] (v11) -- (v12);
					
					\draw[blue] (u2) -- (u3); 
					\draw[blue][bend left=50] (u1) to (u3); 
					\draw[blue][bend left=50] (u2) to (u4); 
					\draw[blue] (u5) -- (u6); 
					\draw[blue][bend left=50] (u4) to (u6); 
					\draw[blue][bend left=50] (u5) to (u7); 
					\draw[blue] (u7) -- (u8);
					\draw[blue][bend left=50] (u6) to (u8);
					\draw[blue][bend left=50] (u7) to (u9); 
					\draw[blue] (u10) -- (u11);
					\draw[blue][bend left=50] (u9) to (u11);
					\draw[blue][bend left=50] (u10) to (u12);
					
					\draw[dashed, rounded corners=20pt] (3.7,1.7) rectangle (7.3,3.5);
					
					\draw[dashed, rounded corners=20pt] (-0.3,-0.5) rectangle (3.3,1.3);
					\draw[dashed, rounded corners=20pt] (2.7,-0.5) rectangle (6.3,1.3);
					\draw[dashed, rounded corners=20pt] (4.7,-0.5) rectangle (8.3,1.3);
					\draw[dashed, rounded corners=20pt] (7.7,-0.5) rectangle (11.3,1.3);
					
					\draw[dashed, rounded corners=20pt] (-0.3,-2.5) rectangle (3.3,-1);
					\draw[dashed, rounded corners=20pt] (2.7,-2.5) rectangle (6.3,-1);
					\draw[dashed, rounded corners=20pt] (4.7,-2.5) rectangle (8.3,-1);
					\draw[dashed, rounded corners=20pt] (7.7,-2.5) rectangle (11.3,-1);
					
				\end{tikzpicture}\hfil}
			\caption{An example with $n=2$. The ordered graph $A^0_1$ is in red and $A^0_2$ is in blue. On the second line we have the Ramsey graph $A^1_1$ with several copies of $A^0_1$ all intersecting in either an edge or a single vertex. Finally, in the last line we have the graph $A^1_2$ on the same set of vertices.}
			\label{fig:fig1}
		\end{figure}
		
		\begin{claim}\label{clm:reverse_induction}
			Let $0\leq j \leq n$. Then every $t$-coloring of $A^n$ contains a copy of $A^j$ such that the ordered graphs $A^j_{k}$ are monochromatic for $j+1\leq k\leq n$.
		\end{claim}
		\begin{proof}
			We prove the statement by reverse induction on $j$. If $j=n$, then the statement is vacuously true. Now assume that for a $t$-coloring of $A^n$ we obtain a copy $\tilde{A}^{j+1}=\bigcup_{i=1}^n \tilde{A}_i^{j+1}$ of $A^{j+1}$ such that the ordered graphs $\tilde{A}^{j+1}_k$ are monochromatic for $j+2\leq k\leq n$.
			Consider the restriction of the~$t$-coloring to the ordered graph~$\tilde{A}_{j+1}^{j+1}$.
			By construction and Theorem~\ref{thm:ramseysteiner}, there exists a monochromatic copy $\tilde{A}_{j+1}^{j}$ of $A_{j+1}^{j}$.
			For $i\neq j$, let $\tilde{A}_i^j=\tilde{A}_i^{j+1}[V(\tilde{A}_{j+1}^{j})]$. It is easy to see that $\tilde{A}^j=\bigcup_{i=1}^n \tilde{A}_i^j$ is a copy of $A^j$. 
			Moreover, since $\tilde{A}_k^j\subseteq \tilde{A}_k^{j+1}$ for $j+2\leq k\leq n$, we have that $\tilde{A}_k^j$ is monochromatic for $j+1\leq k \leq n$.
			This concludes the proof of the claim.  
		\end{proof}
		
		Let $H:=A^n$. Then by Claim \ref{clm:reverse_induction}, every $t$-coloring of $H$ contains a copy of $A^0$ such that every $A^{0}_i$ is monochromatic. Since $A^0=G$, properties (i) and (ii) follow. 
	\end{proof}
	
	The second auxiliary result establishes the existence of an ordered linear $k$-graph with the property that, regardless of how one orders its vertices, there will always be an edge which according to the new ordering is arranged in either a strictly increasing or decreasing order (with respect to the original order).
	We remark that the problem becomes somewhat simpler if we drop the condition that the $k$-graph is linear.
	Indeed, in this case, one can construct a graph by simply taking the complete $k$-graph on $(k-1)^2+1$ vertices and applying the Erd\H{o}s--Szekeres theorem \cite{ES:35}.
	
	\begin{prop}\label{prop:ordering}
		For every integer $k\geq 2$, there exists an ordered linear $k$-graph $(H,<)$ with the following property. For any total ordering $\strictif$ of $V(H)$, there exists an edge $\{x_1,\ldots,x_k\}$ with $x_1<\ldots< x_k$ such that either
		\begin{enumerate}
			\item[$(a)$] the edge $\{x_1,\ldots,x_k\}$ is increasing in $\strictif$, i.e., $x_1\strictif \ldots \strictif x_k$; or
			\item[$(b)$] the edge $\{x_1,\ldots,x_k\}$ is decreasing in $\strictif$, i.e., $x_1\strictfi \ldots \strictfi x_k$.
		\end{enumerate}
	\end{prop}
	
	\begin{proof}
		Let $S_k$ be the group of all permutations on $k$ elements, and let $\id, \rev \in S_k$ be the permutations given by $\id(i)=i$ and $\rev(i)=k+1-i$. In other words, $\id$ and $\rev$ are the permutations that arrange the elements in increasing and decreasing order, respectively. Suppose that $\sigma \in S_k\setminus\{\id, \rev\}$. Then there exist integers $a_{\sigma}, b_\sigma, c_\sigma, d_\sigma \in [k]$, not necessarily distinct, satisfying $a_\sigma<b_\sigma$, $c_\sigma<d_\sigma$, and
		\begin{align}\label{eq:abcd}
			\sigma(a_\sigma)<\sigma(b_\sigma)\quad \text{ and }\quad \sigma(c_{\sigma})>\sigma(d_{\sigma}).
		\end{align}
		
		For every permutation $\sigma \in S_k\setminus\{\id, \rev\}$, we construct an ordered linear $k$-graph $(G_{\sigma},<)$ as follows. Let $G_{\sigma}$ be a $k$-graph on $3k-3$ vertices consisting of three edges $e_1$, $e_2$, and $e_3$, with $|e_i \cap e_j|=1$ for all $i\neq j$. Let $<$ be an ordering of $V(G_{\sigma})$ with $e_1=\{x_1,\ldots,x_k\}$, $e_2=\{y_1,\ldots,y_k\}$, and $e_3=\{z_1,\ldots, z_k\}$, where the vertices are labeled in increasing order in $<$, such that
		\begin{align}\label{eq:G_sigma}
			x_{a_\sigma}=z_{c_{\sigma}},\quad x_{b_\sigma}=y_{a_\sigma}, \quad \text{and} \quad y_{b_\sigma}=z_{d_\sigma}. 
		\end{align}
		Observe that such an ordering is always possible (e.g., see Figure \ref{fig:fig2}). Let $(G,<)$ be the ordered $k$-graph obtained by taking the vertex-disjoint union of $(G_{\sigma},<)$ for all permutations $\sigma \in S_k\setminus\{\id, \rev\}$. Our ordered linear $k$-graph $(H,<)$ is the $k$-graph obtained by applying Theorem \ref{thm:ramseysteiner} to $(G,<)$ with $t=k!$ colors.
		
		\begin{figure}[h]
			\centering
			{\hfil \begin{tikzpicture}[scale=1]
					
					\coordinate (x1) at (0,0);
					\coordinate (x2) at (1,0);
					\coordinate (x3) at (2,0);
					\coordinate (x4) at (3,0);
					\coordinate (x5) at (4,0);
					\coordinate (x6) at (5,0);
					\coordinate (x7) at (6,0);
					\coordinate (x8) at (7,0);
					\coordinate (x9) at (8,0);
					
					\draw[fill][green,opacity=0.2] (x6) arc [
					start angle=0,
					end angle=180,
					x radius=1.001,
					y radius =0.5
					] -- (x4) arc [
					start angle=0,
					end angle=180,
					x radius=1.001,
					y radius =0.5
					] -- (x2) arc
					[
					start angle=0,
					end angle=180,
					x radius=0.5,
					y radius =0.25
					] -- (x1) arc [
					start angle=180,
					end angle=0,
					x radius=2.5,
					y radius =1.25
					] -- (x6);
					
					\draw[fill][blue,opacity=0.2] (x7) arc
					[
					start angle=0,
					end angle=180,
					x radius=1.001,
					y radius =0.5
					] -- (x5) arc
					[
					start angle=0,
					end angle=180,
					x radius=0.5,
					y radius =0.25
					] -- (x4) arc
					[
					start angle=0,
					end angle=180,
					x radius=0.5,
					y radius =0.25
					] -- (x3) arc 
					[
					start angle=180,
					end angle=0,
					x radius=2,
					y radius =1
					];
					
					\draw[fill][red,opacity=0.2] (x9) arc
					[
					start angle=0,
					end angle=180,
					x radius=0.5,
					y radius =0.25
					] -- (x8) arc
					[
					start angle=0,
					end angle=180,
					x radius=1.5,
					y radius =0.75
					] -- (x5) arc
					[
					start angle=0,
					end angle=180,
					x radius=1.5,
					y radius =0.75
					] -- (x2) arc
					[
					start angle=180,
					end angle=0,
					x radius=3.5,
					y radius =1.75
					] -- (x9);

					\draw[green!60!black] (x2) arc
					[
					start angle=0,
					end angle=180,
					x radius=0.5,
					y radius =0.25
					] ;
					\draw[green!60!black] (x4) arc
					[
					start angle=0,
					end angle=180,
					x radius=1,
					y radius =0.5
					] ;
					
					\draw[blue!60!black] (x4) arc
					[
					start angle=0,
					end angle=180,
					x radius=0.5,
					y radius =0.25
					] ;
					\draw[blue!60!black] (x5) arc
					[
					start angle=0,
					end angle=180,
					x radius=0.5,
					y radius =0.25
					] ;
					\draw[red!60!black] (x5) arc
					[
					start angle=0,
					end angle=180,
					x radius=1.5,
					y radius =0.75
					] ;
					\draw[green!60!black] (x6) arc
					[
					start angle=0,
					end angle=180,
					x radius=2.5,
					y radius =1.25
					] ;
					\draw[green!60!black] (x6) arc
					[
					start angle=0,
					end angle=180,
					x radius=1,
					y radius =0.5
					] ;
					\draw[blue!60!black] (x7) arc
					[
					start angle=0,
					end angle=180,
					x radius=1,
					y radius =0.5
					] ;
					\draw[blue!60!black] (x7) arc
					[
					start angle=0,
					end angle=180,
					x radius=2,
					y radius =1
					] ;
					\draw[red!60!black] (x8) arc
					[
					start angle=0,
					end angle=180,
					x radius=1.5,
					y radius =0.75
					] ;
					\draw[red!60!black] (x9) arc
					[
					start angle=0,
					end angle=180,
					x radius=0.5,
					y radius =0.25
					] ;
					\draw[red!60!black] (x9) arc
					[
					start angle=0,
					end angle=180,
					x radius=3.5,
					y radius =1.75
					] ;
					
					\draw[fill] (x1) circle [radius=0.05];
					\draw[fill] (x2) circle [radius=0.05];
					\draw[fill] (x3) circle [radius=0.05];
					\draw[fill] (x4) circle [radius=0.05];
					\draw[fill] (x5) circle [radius=0.05];
					\draw[fill] (x6) circle [radius=0.05];
					\draw[fill] (x7) circle [radius=0.05];
					\draw[fill] (x8) circle [radius=0.05];
					\draw[fill] (x9) circle [radius=0.05];
					
					\node (r1) at (x1) [below] {$1$};
					\node (r2) at (x2) [below] {$2$};
					\node (r3) at (x3) [below] {$3$};
					\node (r4) at (x4) [below] {$4$};
					\node (r5) at (x5) [below] {$5$};
					\node (r6) at (x6) [below] {$6$};
					\node (r7) at (x7) [below] {$7$};
					\node (r8) at (x8) [below] {$8$};
					\node (r9) at (x9) [below] {$9$};
					
				\end{tikzpicture}\hfil}
			\caption{An example of $G_{\sigma}$ for the permutation $\sigma\in S_4$ given by $\sigma(1)=3$, $\sigma(2)=1$, $\sigma(3)=4$ and $\sigma(4)=2$ and $a_\sigma=2$, $b_{\sigma}=3$, $c_\sigma=1$ and $d_\sigma=2$. The edges are given by $e_1=\{1,2,4,6\}$ (green), $e_2=\{3,4,5,7\}$ (blue) and $e_3=\{2,5,8,9\}$ (red).}\label{fig:fig2}
		\end{figure}
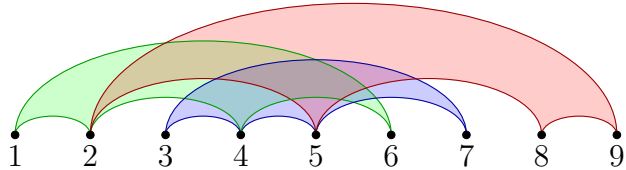
		
		The importance of the $k$-graphs $G_{\sigma}$ is illustrated in the next claim. Given an ordered edge $e=\{x_1,\ldots,x_k\}$ with $x_1< \ldots < x_k$, a permutation $\sigma\in S_k$, and a total ordering $\strictif$, we say that the edge $e$ is \emph{$\sigma$-compatible} with respect to the total ordering $\strictif$ if
		\begin{align}\label{eq:compatible}
			x_{i}\strictif x_{j} \quad \text{if and only if} \quad \sigma(i)<\sigma(j).
		\end{align}
		In particular, the edge $e$ is $\id$-compatible with respect to $<$.
		
		\begin{claim}\label{clm:gsigma}
			Let $\sigma\in S_k\setminus\{\id, \rev\}$ and let $\strictif$ be a total ordering of $V(G_\sigma)$.
			Then not all the edges of $(G_{\sigma},<)$ are $\sigma$-compatible with respect to~$\strictif$.
		\end{claim}
		
		\begin{proof}
			Suppose, for the sake of contradiction, that $\strictif$ is a total ordering of $V(G_{\sigma})$ such that all the edges are $\sigma$-compatible with respect to~$\strictif$.
			Let $e_1=\{x_1,\ldots,x_k\}$, $e_2=\{y_1,\ldots,y_k\}$, and $e_3=\{z_1,\ldots,z_k\}$ be the edges of $G_{\sigma}$, where the vertices are labeled in increasing order with respect to $<$ and satisfy (\ref{eq:G_sigma}).
			Since all the edges are $\sigma$-compatible with respect to $\strictif$, we have by (\ref{eq:abcd}), (\ref{eq:G_sigma}), and (\ref{eq:compatible}),
			\begin{align*}
				x_{a_{\sigma}}\strictif x_{b_{\sigma}}=y_{a_{\sigma}}\strictif y_{b_\sigma}=z_{d_\sigma}\strictif z_{c_\sigma}=x_{a_\sigma},
			\end{align*}
			which is a contradiction. This concludes the proof of the claim.
		\end{proof}
		
		We are now ready to prove that~$(H,<)$ satisfies the statement.
		Let~$\strictif$ be a total ordering of~$V(H)$, and let~$S_k=\{\sigma_1,\ldots, \sigma_{k!}\}$ be a labeling of the~$k!$ permutations.
		We define an auxiliary coloring~$\chi:H\to [k!]$ of the edges of~$H$ as follows.
		For every edge~$e \in H$, let~$\chi(e)=i$ if the edge~$e$ is~$\sigma_i$-compatible with respect to~$\strictif$.
		Since for every edge~$e$, there exists a unique permutation that is compatible with respect to~$\strictif$, the auxiliary coloring~$\chi$ is well defined.
		
		By the construction of $(H,<)$ and Theorem \ref{thm:ramseysteiner}, there exists a monochromatic copy of~$(G,<)$ with respect to~$\chi$.
		In particular, this implies that there exists~$\tau\in S_k$ such that every edge of~$G$ is~$\tau$-compatible with respect to~$\strictif$.
		Since~$G$ is the disjoint union of~$G_{\sigma}$ for~$\sigma\in S_k\setminus\{\id, \rev\}$, we obtain by Claim~\ref{clm:gsigma} that~$\tau\in \{\id,\rev\}$. 
		If~$\tau=\id$, then every edge of~$G$ satisfies (a). Otherwise, if~$\tau=\rev$, then every edge of~$G$ satisfies (b). This concludes the proof of the proposition.
	\end{proof}
	
	We are now ready to prove Lemma \ref{lem:distinguish_palette}.
	
	\begin{proof}[Proof of Lemma \ref{lem:distinguish_palette}]
		Let~$t:=c(P)$ and~$n:=c(Q)$ be the number of colors of~$P$ and~$Q$.
		Let~$m:=e(Q)$ and enumerate the patterns of~$Q$ by~$Q=\{q_1,\ldots,q_{m}\}$.
		We construct an ordered graph~$(G,<)$ on the vertex set~$[3m]$ with the natural order~$<$ by taking~$G=\bigcup_{j=1}^m T_j$ as the vertex-disjoint union of triangles~$T_j$, with vertex set
		\begin{align*}
			V(T_j)=\{3j-2, 3j-1, 3j\}, 
		\end{align*}
		for $1\leq j \leq m$.
		We partition the edges of $G$ into $n$ edge-disjoint ordered graphs $G=\bigcup_{i=1}^n G_i$ as follows.
		For each $1\leq j \leq m$, let~$q_j=(a_j,b_j,c_j)\in Q$ be the $j$-th pattern of $Q$.
		We define the subgraphs $G_i$ on the vertex set~$[3m]$ by setting
		\begin{align}\label{eq:constructG}
			E(G_i)=\bigcup_{j\in[m]:a_j=i}\{\{3j-2,3j-1\}\}\cup\bigcup_{j\in[m]:b_j=i}\{\{3j-2,3j\}\}\cup\bigcup_{j\in[m]:c_j=i}\{\{3j-1,3j\}\}\,.
		\end{align}
		Informally speaking, each triangle~$T_j$ corresponds to the~$j$-th pattern of~$Q$ and the graph~$G_i$ consists of all those pairs across all patterns which have the color~$i$.
		It is easy to check that the $G_{i}$'s are pairwise edge-disjoint and that $G=\bigcup_{i=1}^n G_i$ (e.g., see Figure \ref{fig:fig3}). 
		
		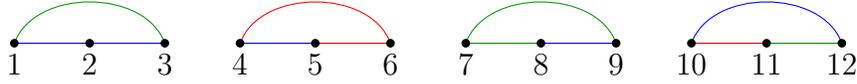
\begin{figure}[h]
			\centering
			{\hfil \begin{tikzpicture}[scale=1]
					
					\coordinate (x1) at (0,0);
					\coordinate (x2) at (1,0);
					\coordinate (x3) at (2,0);
					\coordinate (x4) at (3,0);
					\coordinate (x5) at (4,0);
					\coordinate (x6) at (5,0);
					\coordinate (x7) at (6,0);
					\coordinate (x8) at (7,0);
					\coordinate (x9) at (8,0);
					\coordinate (x10) at (9,0);
					\coordinate (x11) at (10,0);
					\coordinate (x12) at (11,0);
					
					\draw[blue] (x1) -- (x2); 
					\draw[green!60!black][bend left=70] (x1) to (x3); 
					\draw[blue](x2) -- (x3); 
					
					\draw[blue] (x4) -- (x5); 
					\draw[red][bend left=70] (x4) to (x6); 
					\draw[red](x5) -- (x6); 
					
					\draw[green!60!black] (x7) -- (x8); 
					\draw[green!60!black][bend left=70] (x7) to (x9); 
					\draw[blue](x8) -- (x9); 
					
					\draw[red] (x10) -- (x11); 
					\draw[blue][bend left=70] (x10) to (x12); 
					\draw[green!60!black](x11) -- (x12); 

					\draw[fill] (x1) circle [radius=0.05];
					\draw[fill] (x2) circle [radius=0.05];
					\draw[fill] (x3) circle [radius=0.05];
					\draw[fill] (x4) circle [radius=0.05];
					\draw[fill] (x5) circle [radius=0.05];
					\draw[fill] (x6) circle [radius=0.05];
					\draw[fill] (x7) circle [radius=0.05];
					\draw[fill] (x8) circle [radius=0.05];
					\draw[fill] (x9) circle [radius=0.05];
					\draw[fill] (x10) circle [radius=0.05];
					\draw[fill] (x11) circle [radius=0.05];
					\draw[fill] (x12) circle [radius=0.05];
					
					\node (r1) at (x1) [below] {$1$};
					\node (r2) at (x2) [below] {$2$};
					\node (r3) at (x3) [below] {$3$};
					\node (r4) at (x4) [below] {$4$};
					\node (r5) at (x5) [below] {$5$};
					\node (r6) at (x6) [below] {$6$};
					\node (r7) at (x7) [below] {$7$};
					\node (r8) at (x8) [below] {$8$};
					\node (r9) at (x9) [below] {$9$};
					\node (r10) at (x10) [below] {$10$};
					\node (r11) at (x11) [below] {$11$};
					\node (r12) at (x12) [below] {$12$};
					
				\end{tikzpicture}\hfil}
			\caption{An example of $G$ for the palette $Q=\{q_1,q_2,q_3,q_4\}$ given by $q_1=(\text{blue}, \text{green}, \text{blue})$, $q_2=(\text{blue},\text{red},\text{red})$, $q_3=(\text{green}, \text{green}, \text{blue})$ and $q_4=(\text{red},\text{blue},\text{green})$. The graph $G$ consists of $m=4$ triangles and it can be partitioned into $G_{\text{blue}}\cup G_{\text{green}}\cup G_{\text{red}}$ as shown in the picture.}\label{fig:fig3}
		\end{figure}
		
		We construct our desired $3$-graph $F^{(3)}$ by applying Propositions \ref{prop:ramseysystem} and \ref{prop:ordering}.
		Let $(H,<)$ with $H=\bigcup_{i=1}^n H_i$ be the ordered graph obtained by applying Proposition \ref{prop:ramseysystem} to the ordered graph $(G,<)$ with $G=\bigcup_{i=1}^n G_i$ and $t$-colors.
		Set~$k:=v(H)$ to be the number of vertices of~$H$ and let~$(\cH,<)$ be the linear~$k$-graph obtained by Proposition~\ref{prop:ordering}.
		We construct the ordered graph~$(A,<)$ with vertex set~$V(A)=V(\cH)$ by replacing each edge~$e\in \cH$ with a copy~$(H^e,<)$ of~$(H,<)$.
		Since the~$k$-graph~$\cH$ is linear, every two copies~$H^{e}$ and~$H^{e'}$ in~$A$ intersect in at most one vertex.
		This in particular implies that~$A=\bigcup_{i=1}^n A_i$ is the union of~$n$ edge-disjoint ordered graphs, where~$A_i=\bigcup_{e\in\cH}H^e_i$.
		Finally, the~$3$-graph~$F:=F^{(3)}$ is the hypergraph with vertex set~$V(F)=V(A)$ and the edge set as follows.
		Let~$\phi:A\rightarrow [n]$ be the map defined by setting~$\phi(e)=i$ if and only if~$e\in A_i$.
		With this notation in mind, the edge set of~$F$ is given by
		\begin{align}\label{eq:paintF}
			F=\Big\{\{x,y,z\}\in A^{(3)}:\: x<y<z,\,\{x,y\},\, \{x,z\},\, \{y,z\}\in A\nonumber\\ \text{ and } (\phi(x,y),\phi(x,z),\phi(y,z))\in Q\Big\}\,.
		\end{align} 
		In other words,~$F$ is the~$3$-graph where the edges correspond to those triangles in~$A$ whose color pattern is given by the palette~$Q$.
		Note that the construction given by~\eqref{eq:paintF} immediately gives us that~$Q$ paints~$F$.
		Indeed, just take the natural order~$<$ of~$V(F)$ and consider the coloring~$\chi:V(F)^{(2)}\rightarrow [n]$ given by 
		\begin{align*}
			\chi(x,y)=\begin{cases}
				\phi(x,y),& \quad \text{if $\{x,y\}\in A$,}\\
				1, &\quad \text{otherwise}.
			\end{cases}
		\end{align*}
		
		We claim that~$P$ does not paint~$F$. Suppose to the contrary that it does.
		Then there exists a total ordering~$\strictif$ of $V(F)$ and a coloring~$\chi_P:V(F)^{(2)}\rightarrow [t]$ such that
		\begin{align*}
			(\chi_P(x,y), \chi_P(x,z),\chi_P(x,z))\in P
		\end{align*}
		for every~$\{x,y,z\}\in F$ with~$x\strictif y \strictif z$.
		Consider the restriction of~$\chi_P$ to the edges of~$A\subseteq V(F)^{(2)}$.
		By construction of~$A$ and Propositions~\ref{prop:ramseysystem} and~\ref{prop:ordering}, there exists a copy of~$(G,<)$ with~$G=G_1\cup\ldots \cup G_n$ and vertex set~$X=\{x_1,\ldots,x_{3m}\}$ with~$x_1<\ldots<x_{3m}$ such that
		\begin{enumerate}
			\item[(i)] For $1\leq i \leq n$, the graph $G_i$ is monochromatic with respect to $\chi_P$
			\item[(ii)] Either $x_1\strictif \ldots \strictif x_{3m}$ or $x_1\strictfi \ldots \strictfi x_{3m}$.
		\end{enumerate}
		Note that by (\ref{eq:paintF}) and the definition of $G$, the induced graph $F[X]$ is just a matching of size $m$ with edges $\{x_{3j-2},x_{3j-1}, x_{3j}\}$ for $1\leq j\leq m$ (see Figure \ref{fig:fig3}). Let $\psi:C(Q)\rightarrow C(P)$ be the map defined by
		\begin{align*}
			\psi(i)=\chi_P(G_i),
		\end{align*}
		i.e., $\psi(i)$ is the color of the monochromatic graph $G_i$.
		Claiming that this map gives a homomorphism from~$Q$ to~$P$ or a homomorphism from~$Q$ to~$\rev(P)$, we split the proof into two cases.
		
		\vspace{0.2cm}
		
		\noindent \underline{Case 1:} $x_1\strictif\ldots \strictif x_{3m}$.
		
		\vspace{0.2cm}
		
		For $1\leq j \leq m$, let $q_j=(a_j,b_j,c_j) \in Q$ be the $j$-th pattern of $Q$. By (\ref{eq:constructG}) we have that $\phi(x_{3j-2},x_{3j-1})=a_j$, $\phi(x_{3j-2},x_{3j})=b_j$ and $\phi(x_{3j-1},x_{3j})=c_j$.
		Since~$\chi_P$ is a witness that~$P$ paints~$F$ and we further have~$x_{3j-2}\strictif x_{3j-1}\strictif x_{3j}$ and~$\{x_{3j-2},x_{3j-1},x_{3j}\}\in F$, we infer that
		\begin{align*}
			(\psi(a_j),\psi(b_j),\psi(c_j))=(\chi_P(x_{3j-2},x_{3j-1}),\chi_P(x_{3j-2},x_{3j}), \chi_P(x_{3j-1},x_{3j}))\in P,
		\end{align*}
		for $1\leq j \leq m$. This in particular implies that $\psi$ is a homomorphism from $Q$ to $P$, which contradicts the assumption of the lemma.
		
		\vspace{0.2cm}
		
		\noindent \underline{Case 2:} $x_1\strictfi\ldots \strictfi x_{3m}$.
		
		\vspace{0.2cm}
		
		Similarly as in Case 1, since $P$ paints $F$ and $x_{3j}\strictif x_{3j-1} \strictif x_{3j-2}$, we have that
		\begin{align*}
			(\psi(c_j),\psi(b_j),\psi(a_j))=(\chi_P(x_{3j-1},x_{3j}),\chi_P(x_{3j-2},x_{3j}), \chi_P(x_{3j-2},x_{3j-1}))\in P,
		\end{align*}
		for $1\leq j \leq m$. This implies that $\psi(q_j)=(\psi(a_j),\psi(b_j),\psi(c_j)) \in \rev(P)$ and hence $\psi$ is a homomorphism from $Q$ to $\rev(P)$, which is again a contradiction.
	\end{proof}

	\begin{remark}\label{rmk:familyP}
		We observe that the same proof of Lemma \ref{lem:distinguish_palette} can be used to prove the statement for a finite family of palettes $\{P_1,\ldots,P_k\}$. That is, given a palette $Q$ that is not contained in a blow-up of $P_i$ and $\rev(P_i)$ for $1\leq i \leq k$, then there exists a $3$-graph $F$ such that $Q$ paints $F$ and none of the $P_i$'s paints $F$.
	\end{remark}
	
	\section{Properties of blow-ups of palettes}\label{sec:properties}

	In this section, we discuss properties of palettes contained in a blow-up of a given palette~$P$ which will be necessary for the stability argument in Section~\ref{sec:stability}.
	We begin by examining the interplay between~$P$ and~$
	\rev(P)$.
	Let $C(P)=C(
	\rev(P))=[t]$.
	The first observation is that since~$(a,b,c)\in P$ if and only if~$(c,b,a)\in \rev(P)$, it follows that~$P$ and~$
	\rev(P)$ have the same Lagrange polynomial and, consequently, the same Lagrangian.
	Clearly, this still holds when inducing to any subset of colors.
	
	\begin{fact}\label{fact:lambdaPrevP}
		$\Lambda_{P[U]}=\Lambda_{\rev(P)[U]}$ for every $U\subseteq [t]$.
	\end{fact}
	
	It immediately follows that~$P$ is reduced if and only if~$\rev(P)$ is reduced.
    Another simple observation is that if~$S_P$ and~$S_{\rev(P)}$ are blow-ups of~$P$ and~$\rev(P)$ on the same set of colors~$C$ and with the same partition structure~$C=\bigcup_{i=1}^t V_i$, then~$e(S_P)=e(S_{\rev(P)})$.
	In particular, this implies that the maximum blow-up of~$P$ on~$n$ colors has the same number of patterns as the maximum blow-up of~$\rev(P)$ on~$n$ colors (something we have already used in the ``in particular'' part of Theorem~\ref{thm:expal_eq_blowup}).
	
	The following observation shows that the Lagrangian of palettes is monotone with respect to homomorphisms. 
	
	\begin{obs}\label{obs:QcontainP}
		If~$Q$ and~$P$ are palettes and there is a homomorphism $\psi:Q \to P$, then $$\Lambda_Q\leq \Lambda_P\,.$$
	\end{obs}
	\begin{proof}
		Let $\bx \in \bbS_{c(Q)}$ be a weighting of $Q$ with~$\lambda_Q(\bx) = \Lambda_Q$.
		Then, for~$d \in C(P)$, define 
		\begin{align*}
			y_d = \sum_{a \in \psi^{-1}(d)}x_a.
		\end{align*}
		Let $\by=(y_d)_{d\in C(P)}$ and note that~$\by \in \bbS_{c(P)}$ because every~$a\in C(Q)$ is in the preimage of exactly one~$d\in C(P)$.
		Since~$\psi$ is a homomorphism and~$\bx \in \bbS_{c(Q)}$, it follows that
		\begin{align*}
			\Lambda_P &\geq \lambda_P(\by) = \sum_{(d,e,f) \in P}y_dy_ey_f \\
			&= \sum_{(d,e,f) \in P}\left(\sum_{a \in \psi^{-1}(d)}x_a \right)\left(\sum_{b \in \psi^{-1}(e)}x_b \right)\left(\sum_{c \in \psi^{-1}(f)}x_c \right)\\
			&\geq \sum_{(a,b,c) \in Q}x_ax_bx_c = \Lambda_Q.
		\end{align*}
		This concludes the proof.
	\end{proof}
	
	Recall that a palette~$P$ is \emph{reduced} if for every proper subpalette~$Q\subsetneq P$ we have~$\Lambda_Q<\Lambda_P$.
	We conclude our discussion on the interplay between~$P$ and~$\rev(P)$ by showing that if~$P$ is reduced and~$P\not\cong \rev(P)$, then~$P$ is not contained in a blow-up of~$\rev(P)$ and~$\rev(P)$ is not contained in a blow-up of~$P$.

	\begin{prop}\label{prop:PandrevP}
		Let $P$ be reduced. If there is a homomorphism $\psi: P \to \rev(P)$, then $P\cong \rev(P)$.
	\end{prop}
	
	\begin{proof}
		First, we check that $\psi$ must be surjective.
		Indeed, if $\text{Im}(\psi)\subsetneq C(\rev(P))$, then
		\begin{align*}
			\Lambda_P \leq \Lambda_{\rev(P)[\text{Im}(\psi)]} = \Lambda_{P[\text{Im}(\psi)]} < \Lambda_{P},
		\end{align*}
		where the first inequality follows from Observation~\ref{obs:QcontainP}, the equality follows from Fact~\ref{fact:lambdaPrevP}, and the final inequality follows from $P$ being reduced.
		This is a contradiction, and therefore~$\psi$ is surjective.
		Since~$P$ and~$\rev(P)$ have the same number of colors and patterns, the surjective homomorphism~$\psi$ must be an isomorphism, concluding the proof.
	\end{proof}
	
	The next two results deal with properties of reduced palettes~$P$.
	Roughly speaking, the first one states that any palette~$Q$ which is contained in a blow-up of~$P$ and has density~$d(Q)=e(Q)/c(Q)^3$ very close to~$\Lambda_P$ must have a positive proportion of colors in each class of the partition structure.
	
	\begin{prop}\label{prop:closetomax}
		Given a reduced palette~$P$ with~$t$ colors, there are~$\beta=\beta_{\ref{prop:closetomax}}>0$ and~$\epsilon=\epsilon_{\ref{prop:closetomax}}>0$ such that the following holds.
		Suppose that~$Q$ is a palette that is contained in a blow-up~$P'$ of~$P$ with partition structure $C(P')=\bigcup_{i=1}^t V_i$ and~$c(P')=c(Q):=n$.
		If in addition we have~$d(Q) \geq \Lambda_P - \epsilon$, then $|V_i|\geq \beta n$ for every $i\in[t]$.  
	\end{prop}
	\begin{proof}
		Suppose for the sake of contradiction, that the statement does not hold.
		Then, for each integer $m \in \NN$, there exist a blow-up~$P'^{(m)}$ of~$P$ with partition structure $C^{(m)}:=C(P'^{(m)})=\bigcup_{i=1}^tV_i^{(m)}$ and a palette~$Q^{(m)}\subseteq P'^{(m)}$ such that $d(Q^{(m)}) \geq \Lambda_P - \frac{1}{m}$ and
		\begin{align}\label{eq:negation}
			|V_{i^{(m)}}^{(m)}|< \frac{1}{m} c(Q^{(m)})
		\end{align}
		for some sequence of indices $i^{(m)} \in [t]$. By applying the pigeonhole principle, we obtain a subsequence (which we reindex using $m$ again) such that~$i^{(m)}$ is constant, say~$i^{(m)}=t$.
		Then, for each~$m \in \NN$, we have
		\begin{align}\label{eq:lambdainducedont-1}
			\Lambda_{P[[t-1]]} \geq \Lambda_{Q^{(m)}[C^{(m)}\setminus V_t^{(m)}]} \geq d(Q^{(m)}[C^{(m)}\setminus V^{(m)}_t]) \geq d(Q^{(m)}) - \frac{1}{m} \geq \Lambda_{P}-\frac{2}{m},
		\end{align}
		where the first inequality follows from Observation~\ref{obs:QcontainP} applied to $P[[t-1]]$ and $Q^{(m)}[C^{(m)}\setminus V^{(m)}_t]$, the second inequality comes from~(\ref{eq:lambda_density}), the third from~(\ref{eq:negation}), and the fourth from the choice of~$Q^{(m)}$.
		Since~$P$ is reduced, we must have~$\Lambda_{P[[t-1]]}< \Lambda_{P}$, which contradicts~\eqref{eq:lambdainducedont-1} for~$m$ large enough.
	\end{proof}
	
	Given a palette~$P$ and two (distinct) colors~$a, b \in C(P)$, we say that~$b$ \emph{dominates}~$a$ if, for every pattern~$p\in P$ containing~$a$, any substitution of the color~$a$ with the color~$b$ results in a pattern~$p' \in P$.
    As an example, suppose that~$1,2 \in C(P)$ and~$2$ dominates~$1$.
    Then~$(1,1,x) \in P$ implies that~$(1,2,x)$,~$(2,1,x)$, and~$(2,2,x)$ are all in~$P$.
    Although it is straightforward to verify the following lemma, we include a proof for the convenience of the reader.

	\begin{lemma}\label{lem:domination}
		For a reduced palette~$P$ with~$C(P)=[t]$ there are no~$a,b \in [t]$ such that $b$ dominates $a$.
	\end{lemma}
    
        \begin{proof}
        First note that we may assume that there is no~$c\in[t]$ with~$(c,c,c)\in P$.
        Otherwise~$d(P[\{c\}])=1$, whence~$P$ being induced would imply~$C(P)=\{c\}$, and we would be done.
        
        Now suppose, for the sake of contradiction, that there are~$a,b\in[t]$ such that~$b$ dominates~$a$. 
        For~$\bz\in\bbS_{t}$ we can write $$\lambda_P(\bz) = z_af_a(\bz')+z_bf_b(\bz')+z_a^2f_{a,a}(\bz')+z_b^2f_{b,b}(\bz')+z_az_bf_{a,b}(\bz')+g(\bz')$$
        for some polynomials~$f_a$,~$f_b$,~$f_{a,a}$,~$f_{b,b}$,~$f_{a,b}$, and~$g$ in~$\bz'=(z_c)_{c\in[t]\setminus\{a,b\}}$.
        Let~$\bx$ be an optimal weighting of~$P$ witnessing~$\lambda_P(\bx)=\Lambda_P$.
        The hypothesis that~$b$ dominates~$a$ implies that~$f_b \geq f_a$, as well as~$f_{b,b} \geq f_{a,a}$, and~$2f_{b,b} \geq f_{a,b}$, where the~$2$ appears since~$(a,b,x)$ and~$(b,a,x)$ are both `covered' by~$(b,b,x)$. 
        We claim that the weighting~$\by \in \bbS_t$ given by~$y_a=0$,~$y_b=x_a+x_b$, and~$y_k=x_k$ for $k\in[t]\setminus\{a,b\}$ satisfies
		\begin{align}\label{eq:domination}
			\lambda_P(\by)\geq \lambda_P(\bx)=\Lambda_P.
		\end{align}
		Indeed,
            \begin{align*}
                \lambda_P(\by) &= (x_a+x_b)f_b(\by')+(x_a+x_b)^2f_{b,b}(\by')+g(\by')\\
                &\geq x_af_b(\bx')+x_bf_b(\bx')+x_b^2f_{b,b}(\bx')+x_a^2f_{b,b}(\bx')+2x_ax_bf_{b,b}(\bx')+g(\bx')\\
                & \geq \lambda_P(\bx)\,,
            \end{align*}
		where~$\by'=(y_c)_{c\in[t]\setminus\{a,b\}}=(x_c)_{c\in[t]\setminus\{a,b\}}=\bx'$.
		On the other hand, note that $y_a=0$ and therefore $\lambda_P(\by)\leq \Lambda_{P[[t]\setminus\{a\}]}<\Lambda_P$, where we use the fact that $P$ is reduced in the last inequality.
        This contradicts (\ref{eq:domination}), which concludes the proof.
        \end{proof}
		
	
	We finish the section by introducing a concept used in~\cite{P:12} that will be important in the stability argument.
    Let~$P$ be a palette with~$C(P)=[t]$.
    A palette~$R$ contained in a blow-up of~$P$ is \emph{rigid (with respect to~$P$)} if there exists a partition of its colors~$C(R)=\bigcup_{i=1}^t U_i$ satisfying the following: If~$R$ is contained in a blow-up~$S$ of~$P$ with partition structure~$C(S)=\bigcup_{i=1}^t V_i$, then there exists an automorphism~$h:[t] \rightarrow [t]$ of~$P$ such that~$U_i\subseteq V_{h(i)}$ for $i \in [t]$.
    In other words, a palette~$R$ is rigid if there is essentially a unique way to embed it in a blow-up of~$P$.
    The next result shows that if~$P$ is reduced, then rigid palettes always exist for sufficiently many colors.
    Recall that a palette is non-degenerate if every pattern has exactly~$3$ colors.
	
	\begin{lemma}\label{lem:rigidity}
		Let $P$ be a reduced palette with color set $[t]$. Then there exists an integer $M:=M_{\ref{lem:rigidity}}(P)$ and a rigid palette $R\subseteq [M]^3$ with partition structure~$[M]=\bigcup_{i=1}^t U_i$ such that
		\begin{enumerate}
			\item[$(i)$] $R$ is non-degenerate.
			\item[$(ii)$] For $i \in [t]$, we have $|U_i|\geq 3t$.
			\item[$(iii)$] Any blow-up $R'$ of $R$ on $M+1$ colors is rigid.
		\end{enumerate}
		Moreover, if $P\not \cong \rev(P)$, then the palette $R$ is not contained in a blow-up of $\rev(P)$.
	\end{lemma}
	
	\begin{proof}
		Since $P$ is reduced, there exists a real number $\delta>0$ such that $\Lambda_Q<\Lambda_P-\delta$ for every proper subset $Q\subsetneq P$. Let $\epsilon, \beta$ be the constants given by Proposition \ref{prop:closetomax}. We will choose a sufficiently large $M$ satisfying the following conditions:
		\begin{enumerate}
			\item[(a)] $M\gg \frac{1}{\varepsilon},\frac{1}{\delta},t,\frac{1}{\beta}$.
			\item[(b)] There exists a blow-up $\tilde{R}$ of $P$ on $M$ colors with partition structure $C(\tilde{R})=\bigcup_{i=1}^t U_i$ such that $d(\tilde{R})>\Lambda_P-\min\{\delta/2,\epsilon/2\}$.
		\end{enumerate}
		Note that condition (b) can always be satisfied because of (\ref{eq:lambda_blowup}).
        Let $R$ be the palette obtained by removing every degenerate edge from $\tilde{R}$.
        It is not difficult to see from conditions (a) and (b) that
		\begin{align}\label{eq:Rdense}
			d(R)\geq \frac{e(\tilde{R})-3M^2}{M^3}\geq d(\tilde{R})-3/M\geq \Lambda_P-\min\{\delta,\epsilon\}.
		\end{align}
		We claim that the palette $R$ is a rigid palette satisfying properties (i), (ii), and (iii).
		
		We first check properties (i) and (ii). Property (i) follows immediately from the construction since we deleted all degenerate edges from $\tilde{R}$.
        To see that property (ii) holds, note that by (\ref{eq:Rdense}), we have~$d(R)\geq \Lambda_P-\epsilon$.
        Hence, Proposition~\ref{prop:closetomax} gives us that~$|U_i|\geq \beta M\geq 3t$, where the last inequality holds due to our choice of~$M$ (condition (a)).
		
		We now proceed to prove that $R$ is rigid.
        Clearly,~$R$ is contained in a blow-up of~$P$ (namely~$\tilde{R}$).
        Let $S$ be a blow-up of $P$ with partition structure $C(S)=\bigcup_{j=1}^t V_j$ and let $\psi:C(R) \to C(S)$ be an embedding (i.e., an injective homomorphism) of $R$ into $S$.
		We define a mapping $h:[t]\rightarrow [t]$ by letting $h(i)$ be an arbitrary index in $[t]$ such that
		\begin{align*}
			|\psi(U_i)\cap V_{h(i)}|\geq 3
		\end{align*}
		for $i \in [t]$. Such a choice of $h$ always exists because $|U_i|\geq 3t$ for $i \in [t]$.
        Let $Y_i\subseteq U_i$ be the preimage of~$\psi(U_i)\cap V_{h(i)}$ under~$\psi$, i.e., the subset of~$U_i$ with~$\psi(Y_i)\subseteq V_{h(i)}$.
        Our goal is to prove that~$h:[t]\rightarrow [t]$ is an automorphism of~$P$ and that~$\psi(U_i)\subseteq V_{h(i)}$ for every~$i \in [t]$ (that is,~$Y_i=U_i$).
        We will do that in several claims.
		
		\begin{claim}\label{clm:Hhomo}
			The map $h$ is a homomorphism from $P$ to $P$.
		\end{claim}
		
		\begin{proof}
			Note that if~$(i_1,i_2,i_3)\in P$, then there is a pattern~$(a,b,c)\in Y_{i_1}\times Y_{i_2} \times Y_{i_3}$ in~$R$.
            Indeed, by definition,~$R$ contains all the non-degenerate patterns in~$U_{i_1}\times U_{i_2}\times U_{i_3}$.
            And since~$\vert Y_i\vert=|\psi(U_i)\cap V_{h(i)}|\geq 3$ for all~$i \in [t]$, even if $i_1=i_2=i_3$, there is some non-degenerate pattern in~$Y_{i_1}\times Y_{i_2} \times Y_{i_3}$.
            However, since~$\psi$ is a homomorphism, the pattern~$(\psi(a),\psi(b),\psi(c)) \in V_{h(i_1)}\times V_{h(i_2)}\times V_{h(i_3)}$ is a pattern of $S$.
            This implies that~$(h(i_1),h(i_2),h(i_3)) \in P$ and consequently,~$h:P\rightarrow P$ is a homomorphism.   
		\end{proof}
		
		\begin{claim}\label{clm:Hbijective}
			The homomorphism $h:[t]\rightarrow [t]$ of~$P$ is bijective.
		\end{claim}
		\begin{proof}
			We construct an auxiliary blow-up $S^h$ of $P$ with partition structure $C(S^h)=\bigcup_{j=1}^tV_j^h$ defined by
			\begin{align*}
				V_j^h=\bigcup_{i\in h^{-1}(j)} \psi(U_i).
			\end{align*}
			That is, the sets $V_j^h$ are the union of the images of $U_i$ such that $h(i)=j$. It is clear that $C(S^h)\subseteq C(S)$. Moreover, note that some of the sets $V_i^h$ might be empty.
			
			We claim that the map~$\psi: C(R)\rightarrow C(S^h)$ is an embedding of~$R$ into~$S^h$ such that~$\psi(U_i)\subseteq V_{h(i)}^h$.
            The latter part holds by definition.
            To see that it is an embedding, let~$(a,b,c)\in U_{i_1}\times U_{i_2}\times U_{i_3}$ be a pattern of~$R$.
            Then~$(i_1,i_2,i_3)\in P$ and hence, since~$h$ is a homomorphism, we have that $(h(i_1),h(i_2),h(i_3))\in P$.
            This implies that~$(\psi(a),\psi(b),\psi(c)) \in V_{h(i_1)}^h\times V_{h(i_2)}^h\times V_{h(i_3)}^h$ is a pattern of~$S^h$ and~$\psi:R\rightarrow S^h$ is an embedding.
			
			Suppose, for the sake of contradiction, that~$h$ is not bijective.
            Then there exists an index~$j \in [t]$ such that~$V_{j}^h=\emptyset$.
            This implies that there exists a homomorphism~$\phi: R\rightarrow P[[t]\setminus\{j\}]$.
            Hence, by (\ref{eq:lambda_density}), Observation~\ref{obs:QcontainP}, and the definition of~$\delta$, we have
			\begin{align*}
				d(R)\leq \Lambda_R\leq \Lambda_{P[[t]\setminus\{j\}]}<\Lambda_P-\delta,
			\end{align*}
			which contradicts (\ref{eq:Rdense}).
		\end{proof}
		
		The last two claims show that~$h:[t]\rightarrow [t]$ is an automorphism of~$P$.
        Suppose, without loss of generality, that~$h(i)=i$ (and henceforth we index both~$U_i$ and~$V_i$ by~$i \in [t]$).
		
		\begin{claim}\label{clm:fully_contained}
			For all~$i \in [t]$ we have~$\psi(U_i)\subseteq V_i$. 
		\end{claim}
		\begin{proof}
			Suppose to the contrary that there exist indices~$i, j$ and a color~$x\in U_i$ such that~$\psi(x)\in V_j$.
            We claim that~$j$ dominates~$i$ in~$P$. Let~$p \in P$ be a pattern containing the color~$i$.
            Suppose, without loss of generality, that~$p$ is of the form~$p=(i,k,\ell)$, where~$k,\ell \in [t]$ ($k$ and $\ell$ might be equal to~$i$ or~$j$).
            Then we can choose distinct~$b\in Y_k$ and~$c\in Y_{\ell}$ (which are also both distinct from~$x$) and so the triple~$(x,b,c)\in U_i\times Y_k\times Y_\ell$ is a pattern of~$R$.
            This implies that~$(\psi(x),\psi(b),\psi(c))\in V_j\times V_k \times V_\ell$ is a pattern of~$S$.
            Hence, by construction, we have~$(j,k,\ell) \in P$. 
            That is, by substituting the color~$i$ with the color~$j$, we still have a pattern of~$P$.
            However, since~$P$ is reduced, by Lemma~\ref{lem:domination} there are no pairs~$\{i,j\}$ with~$j$ dominating~$i$, which yields a contradiction.
		\end{proof}
		
		Claims \ref{clm:Hhomo}--\ref{clm:fully_contained} show that~$R$ is a rigid configuration.
        Next we prove that~$R$ satisfies property (iii).
        First observe that the above argument verifying the rigidity of~$R$ used only properties (i) and (ii) of Lemma~\ref{lem:rigidity} and (\ref{eq:Rdense}).
        Thus, it is sufficient to check these for any palette~$R'$ on $M+1$ colors obtained by taking a blow-up of~$R$.
        Viewing~$R'$ as a blow-up of~$P$, it has partition structure~$C(R')=\bigcup_{i=1}^t V_i'$, where~$|V_i|\leq |V_i'|\leq |V_i|+1$ for all~$i \in [t]$ and so~$R'$ inherits properties (i) and (ii) from~$R$.
        Moreover, by our choice of~$M$ in condition (a), we have
		\begin{align*}
			d(R')\geq \frac{e(\tilde{R})-3M^2}{(M+1)^3}\geq (d(\tilde{R})-3/M)\frac{M^3}{(M+1)^3}\geq \Lambda_P-\min\{\delta,\epsilon\}. 
		\end{align*}
		Thus, by the same proof, the palette $R'$ is rigid.
		
		Finally, to prove the moreover part, suppose that~$P\not \cong \rev(P)$ and that there exists a homomorphism~$\phi: R \rightarrow \rev(P)$.
        Let~$C(R)=\bigcup_{i=1}^t W_i$ be the partition structure of~$R$ as a subpalette of a blow-up of~$\rev(P)$, i.e.,~$W_i=\phi^{-1}(i)$.
        We construct a map~$\xi:[t]\rightarrow [t]$ by letting~$\xi(i)$ be an index such that
		\begin{align*}
			|U_i \cap W_{\xi(i)}|\geq 3. 
		\end{align*}
		Such an index always exists since~$|U_i|\geq 3t$.
        For~$i \in [t]$ let~$Z_i=U_i\cap W_{\xi(i)}$.
        We claim that~$\xi$ is a homomorphism from~$P$ to~$\rev(P)$.
        Let~$(i_1,i_2,i_3)\in P$ be a pattern.
        Since~$R$ is obtained by deleting just the non-degenerate edges of a blow-up of~$P$ and~$|Z_{i_1}|, |Z_{i_2}|, |Z_{i_3}|\geq 3$, there exists a pattern~$(a,b,c)\in Z_{i_1}\times Z_{i_2}\times Z_{i_3}$.
        This implies that~$(a,b,c)\in W_{\xi(i_1)}\times W_{\xi(i_2)}\times W_{\xi(i_3)}$ and consequently~$(\xi(i_1),\xi(i_2),\xi(i_3))\in \rev(P)$.
        Hence,~$\xi:P\rightarrow \rev(P)$ is a homomorphism.
        A contradiction now follows from Proposition~\ref{prop:PandrevP} and the fact that~$P$ is reduced.
        This concludes the proof of Lemma~\ref{lem:rigidity}.
	\end{proof}

	\section{Stability argument}\label{sec:stability}
	
	The main goal of this section is to provide a proof of Theorem~\ref{thm:expal_eq_blowup}.
	We remark that our approach is very similar to the approach in~\cite{P:12}, but adapted to our needs.
	We begin with an outline of the proof.
	For a reduced palette~$P \subseteq [t]^3$ with $t$ colors, define the family~$\cF(P)$ of unpaintable~$3$-graphs as follows:
	\begin{align}\label{eq:F_P}
		\cF(P) := \{F : \text{$F$ is a $3$-graph and }P \text{ is $F$-deficient}\}\,.
	\end{align}
	The first observation is that if~$Q$ is an~$\cF(P)$-deficient palette, then~$Q$ is contained in a blow-up of~$P$ or in a blow-up of~$\text{rev}(P)$ (see Lemma \ref{lem:distinguish_palette}). 
	Unfortunately, the family~$\cF(P)$ might be infinite in size and hence cannot be used directly as a witness for Theorem~\ref{thm:expal_eq_blowup}.
	To circumvent this issue, we choose an appropriate integer~$M$ and truncate the family~$\cF(P)$ to the subfamily~$\cF(P)_M$ of~$3$-graphs with at most~$M$ vertices, i.e.,
	\begin{align}\label{eq:FP_M}
		\cF(P)_M := \{F \in \cF(P) : v(F) \leq M\}\,.
	\end{align}
	Let~$Q$ be an~$\cF(P)_M$-deficient palette that maximizes the number of patterns among all~$\cF(P)_M$-deficient palettes; in other words, suppose that~$Q \in \EXpal(n,\cF(P)_M)$. 
	An application of the removal lemma for palettes (see Lemma~\ref{lem:palette_removal}) implies that~$Q$ is very close to a blow-up of~$P$ or a blow-up of~$\text{rev}(P)$ (see Corollary~\ref{cor:palette_removal} below).
	We then complete the proof using a stability argument, showing that any extremal palette sufficiently close to a blow-up of~$P$ must actually be a blow-up of~$P$ (see Lemma~\ref{lem:stability}). 
	Now let us proceed with the details.
	
	We remind the reader that given two palettes~$P$ and~$Q$ on~$n$ colors, the edit distance~$|P \triangle Q|=\vert P\setminus Q\vert+\vert Q\setminus P\vert$ is the minimum number of patterns that must be deleted or added to transform~$P$ into~$Q$.
    Moreover, we say that~$P$ is~$\alpha$-close to~$Q$ if~$|P \triangle Q| \leq \alpha n^3$. 
	
	The following is a corollary of Lemmata~\ref{lem:palette_removal} and~\ref{lem:distinguish_palette}. 
	\begin{cor}\label{cor:palette_removal}
		Given a palette~$P$ and~$\alpha>0$, there exist~$M=M_{\ref{cor:palette_removal}},N=N_{\ref{cor:palette_removal}} \in \mathds N$ such that the following holds for every palette~$Q$ on~$n\geq N$ colors.
		If~$Q$ is~$\cF(P)_M$-deficient, then it is~$\alpha$-close to being contained in a blow-up of~$P$ or a blow-up of~$\rev(P)$. 
	\end{cor}
	
	\begin{proof}
		Given~$\alpha>0$, apply Lemma~\ref{lem:palette_removal} to the family~$\cF(P)$ to obtain~$M, N\in \mathds N$ and~$\beta>0$.
		The conclusion of Lemma~\ref{lem:palette_removal} entails that for every~$\cF(P)_M$-deficient palette~$Q$ with~$c(Q)=n\geq N$, there exists an~$\alpha$-close palette~$Q'$ that does not paint~$\cF(P)$.
		Suppose for the sake of contradiction that~$Q'$ is not contained in a blow-up of~$P$ or~$\rev(P)$. 
		Then Lemma~\ref{lem:distinguish_palette} yields a~$3$-graph~$F$ painted by~$Q'$ but not by~$P$. 
		In particular,~$F\in \cF(P)$, a contradiction. 
	\end{proof}
	
	The next lemma is the key technical result of this section.
	We prove it using a stability argument similar to the one in~\cite{P:12}.
	
	\begin{lemma}\label{lem:stability}
		Let $P$ be a reduced palette with~$C(P)=[t]$. 
        There are~$N:=N_{\ref{lem:stability}}(P)\in\mathds{N}$, $M := M_{\ref{lem:stability}}(P)\in\mathds{N}$ and~$\alpha := \alpha_{\ref{lem:stability}}(P)>0$ such that for all integers~$n\geq N$ and~$m\geq M$, the following holds. 
        If~$Q \in \EXpal(n,\cF(P)_m)$ and~$Q$ is~$\frac{\alpha}{2}$-close to a blow-up of~$P$, then~$Q$ is a blow-up of~$P$.
	\end{lemma}
	
	\begin{proof}
		Let $M_1=M_{\ref{lem:rigidity}}$ be the integer obtained from Lemma \ref{lem:rigidity}, and $\epsilon=\epsilon_{\ref{prop:closetomax}}, \beta = \beta_{\ref{prop:closetomax}}>0$ be the constants given by Proposition \ref{prop:closetomax}.
        Set~$\alpha$ and~$\gamma$ as real numbers such that
		\begin{align}\label{eq:alphagamma}
			\alpha=\min\{(\beta/4)^{2M_1^2}, \epsilon/3\} \quad \text{and} \quad \gamma=(\beta/4)^{2M_1}\,.
		\end{align}
        Consider the family of palettes $\cR$ given by
		\begin{align*}
			\cR := \{R :\: &c(R) \leq M_1 + 3 \text{ and }\\ &R \text{ is neither contained in a blow-up of } P \text{ nor in a blow-up of } \rev(P) \}\,.
		\end{align*}
		For every~$R \in \cR$, by Lemma \ref{lem:distinguish_palette}, there exists a~$3$-graph~$F_R$ such that~$R$ paints~$F_R$ and~$P$ does not paint~$F_R$ (i.e.,~$F_R \in \cF(P)$).
        Let~$M_2 := \max_{R \in \cR} \{v(F_R)\}$.
        Such an~$M_2$ always exists since~$\cR$ is a finite family.
        We will show that the statement of the lemma holds for~$M_2$ as~$M$ and~$N\in\mathds{N}$ large enough.
        Fix integers~$n\geq N$ and~$m\geq M_2$.
        An immediate consequence of our choice is the following:
		\begin{align}\label{eq:transfer}
			\text{If } Q \text{ does not paint } \cF_{m}(P), \text{ then } Q \text{ is } \cR\text{-free.}
		\end{align}

		Given~$n\in\mathds{N}$, let~$S$ be a blow-up of~$P$ with partition structure~$[n] = \bigcup_{i=1}^t V_i$, i.e.,
		\begin{align*}
			S = \{(x,y,z) \in V_i \times V_j \times V_k : (i,j,k) \in P\}\,.
		\end{align*}
		Given any such~$S$ we can define the set of \emph{missing patterns}~$A$ and the set of \emph{bad patterns}~$B$ of~$Q$ (with respect to~$S$) by
		\begin{align}\label{eq:missingbad}
			A := S \setminus Q \quad \text{and} \quad B := Q \setminus S\,,
		\end{align}
		i.e., the patterns from~$S$ missing in~$Q$ and the patterns in~$Q$ that are not in our target blow-up~$S$.
        Let~$Q \in \EXpal(n,\cF_{M_2}(P))$ be a maximum palette that does not paint~$\cF_{M_2}(P)$ and is~$\frac{\alpha}{2}$-close to a blow-up~$S \subseteq [n]^3$ of~$P$.
        Since $|Q \triangle S| \leq \frac{\alpha}{2} n^3$, we have~$|A|+|B| \leq \frac{\alpha}{2} n^3$.
        Moreover, because~$S$ is a blow-up of~$P$, it does not paint~$\cF_M(P)$.
        Therefore, by the maximality of~$Q$, we have that~$|A| \leq |B| \leq \frac{\alpha}{2} n^3$.
        
        It will be useful later to compare~$P$ not to~$S$ but instead to some blow-up~$S'$ of~$P$ with~$n$ colors that minimizes~$|B|$ among all blow-ups of~$P$ with~$n$ colors.
        Fortunately, we can still check that such~$S'$ is not too far from~$Q$.
        Indeed, if there exists a blow-up~$S'$ with partition structure~$[n] = \bigcup_{i=1}^t V_i'$ such that the missing patterns~$B' = Q \setminus S'$ of~$Q$ with respect to~$S'$ satisfy $|B'| < |B|$, then by the maximality of~$Q$, we have~$|Q \triangle S'| = |A'| + |B'| < 2|B| \leq \alpha n^3$. 
        Hence, at the marginal cost of a slightly larger edit distance, we can replace $S$ by $S'$ and assume that the partition $[n] = \bigcup_{i=1}^t V_i$ minimizes the number of bad patterns.
        To ease the notation, we write~$S$ and~$A$ and~$B$ instead of~$S'$,~$A'$, and~$B'$.
        In particular, we now have~$|Q \triangle S| = |A| + |B| \leq \alpha n^3$ and
        \begin{align}\label{eq:missingbad2}
			|A| \leq |B| \leq \alpha n^3\,.
		\end{align}
		
		Finally, note that by the maximality of $Q$, there exists some~$N$ so that we have $d(Q)\geq \Lambda_P-\alpha$ whenever~$c(Q)=n \geq N$ (since such density can be achieved by a maximal blow-up of $P$ on $n$ colors). 
        This implies by \eqref{eq:alphagamma} that $d(S)\geq \Lambda_P-3\alpha \geq \Lambda_P-\epsilon$.
        Hence, by Proposition \ref{prop:closetomax}, we have
		\begin{align}\label{eq:sizeofVi}
			|V_i|\geq \beta n
		\end{align}
		for $i \in [t]$.
		
		We claim that~$|A| = |B| = 0$.
        Suppose, to the contrary, that~$|B|>0$.
        Given a color~$x \in [n]$, let~$\deg_B(x)$ be the number of patterns in~$B$ containing~$x$ (note that we do not count multiplicities, i.e., the pattern~$(x,x,y)$ counts only once).
        We define the \emph{maximum degree~$\Delta(B)$} as the maximum of~$\deg_B(x)$ over~$x \in [n]$.
        We will first show that the palette of bad patterns~$B$ does not have a large maximum degree.
		
		\begin{claim}\label{clm:maxB}
			$\Delta(B)\leq \gamma n^2$.
		\end{claim}
		\begin{proof}
			Suppose that $\Delta(B)>\gamma n^2$ and let~$x\in [n]$ such that $\deg_B(x)>\gamma n^2$.
            Suppose, without loss of generality, that~$x\in V_1$.
            For every~$k \in [t]$, we define a set~$B_{x}^{(k)}$ as follows.
            Let~$B^{(1)}:=B$.
            For~$k\neq 1$, consider the partition~$[n]=\bigcup_{i=1}^t V_i^{(k)}$ with~$V_i^{(k)}=V_i$ for~$i\notin\{1,k\}$, and~$V_1^{(k)}=V_1\setminus \{x\}$ as well as $V_k^{(k)}=V_k\cup \{x\}$.
            That is,~$[n]=\bigcup_{i=1}^t V_i^{(k)}$ is the partition obtained by moving the color~$x$ from~$V_1$ to~$V_k$.
            Let~$S^{(k)}$ be the blow-up of~$P$ with this partition structure, and let~$B^{(k)}=Q\setminus S^{(k)}$ be the new set of bad patterns of~$Q$ with respect to~$S^{(k)}$.
            Now we define 
			\begin{align*}
				B_{x}^{(k)}:=\{q\in B^{(k)}:\:x\in q\}    
			\end{align*}
			for~$k \in [t]$.
            Recall that we chose~$S$ to minimize the number~$|B|=|Q \setminus S|$, so we have that $|B|\leq |B^{(k)}|$ for $k \in [t]$. 
            Since $|B^{(k)}|-|B|=|B_{x}^{(k)}|-|B_{x}^{(1)}|$, we obtain that 
			\begin{align}\label{eq:Bix}
				|B_{x}^{(k)}|\geq |B_{x}^{(1)}|=\deg_B(x)>\gamma n^2\,.
			\end{align}
			for $k \in [t]$. 
			
			Let~$R_{\star}$ be the non-degenerate rigid palette on~$M_1$ colors obtained by Lemma~\ref{lem:rigidity} with partition~$[M_1]=\bigcup_{i=1}^t U_i$.
            For each~$t$-tuple $\bX=(p_1,\ldots,p_t)\in \prod_{k=1}^t B_x^{(k)}$ of patterns, we define a palette $R_{\star}^{\bX}$ on~$M_1+1$ colors as follows.
            Let~$X=\{p_1,\ldots,p_t\}$ be the palette containing the~$t$ elements of~$\bX$.
            Setting~$W_i=V_i\cap (C(X)\setminus\{x\})$, we have the partitioning~$C(X)=\left(\bigcup_{i=1}^t W_i\right)\cup \{x\}$.
            Consider an injective map~$\iota: C(X)\rightarrow [M_1+1]$ such that~$\iota(W_i)\subseteq U_i$ and~$\iota(x)=M_1+1$.
            Such an embedding is always possible since~$|U_i|\geq 3t$ for~$i \in [t]$ (Property (ii) of Lemma \ref{lem:rigidity}) and~$|W_i| \leq 3t$. 
            Let~$\iota(X)$ be the palette with patterns given by~$\iota(p)$ for every~$p \in X$.
            The palette~$R_{\star}^{\bX}$ is defined as the union~$R_{\star}^{\bX}=R_{\star}\cup \iota(X)$.
			
			\begin{subclaim}\label{sub:notincR}
				$R_{\star}^{\bX} \in \cR$.
			\end{subclaim}
			\begin{proof}
				Clearly,~$c(R_\star^\bX) \leq M_1+3$. 
                We claim that~$R_\star^\bX$ is neither contained in a blow-up of~$P$ nor in a blow-up of~$\rev(P)$.
                First, note that it suffices to check the claim only for~$P$.
                Indeed, by Lemma~\ref{lem:rigidity}, if~$P \not\cong \rev(P)$, then the palette~$R_{\star}$ is not contained in a blow-up of~$\rev(P)$.
                Since~$R_\star \subseteq R_\star^\bX$, this in particular implies that~$R_{\star}^\bX$ is not contained in a blow-up of~$\rev(P)$.
				
				Now suppose for the sake of contradiction that $R_{\star}^\bX$ is contained in a blow-up $S'$ of $P$ with partition structure $\bigcup_{i=1}^t V_i'$.
                In particular,~$R_\star \subseteq S'$, and by the definition of rigidity, there exists some automorphism~$h: [t] \to [t]$ of~$P$ with~$U_i \subseteq V_{h(i)}'$ for all~$i \in [t]$.
                We may assume without loss of generality that $h(i) = i$, i.e., that $U_i \subseteq V_i'$ for $i \in [t]$.
                Suppose that the last color~$M_1+1$ of~$R_\star^\bX$ is contained in~$V_k'$ for some index~$k \in [t]$. 
                Let~$p_k$ be the~$k$-th pattern of~$X$.
                From the fact that~$\iota(x) = M_1+1 \in V_k'$ and~$U_i \subseteq V_i'$, we obtain that the patterns~$p_k$ and~$\iota(p_k)$ respect the same underlying structure in the blow-ups~$S^{(k)}$ and~$S'$ of~$P$, respectively.
                That is,~$p_k \in V_a^{(k)} \times V_b^{(k)} \times V_c^{(k)}$ if and only if~$\iota(p_k) \in V_a' \times V_b' \times V_c'$. 
                Since we assumed that~$R_{\star}^\bX\subseteq S'$, we have that in particular~$\iota(p_k)\in S'$.
                This implies that~$p_k \in S^{(k)}$, which is a contradiction since~$p_k \in B_x^{(k)} = Q \setminus S^{(k)}$. 
                Therefore,~$R_{\star}^\bX$ is not contained in a blow-up of~$P$.
			\end{proof}
			
			Our goal now is to lower bound the cardinality of the set~$A = S \setminus Q$ to obtain a contradiction.
            We say that an (injective) embedding~$f: R_{\star}^\bX \rightarrow Q\cup S$ is \emph{good} if~$f(M_1+1) = x$ and~$f(\iota(p_k)) = p_k$ for~$k \in [t]$.
            That is, if~$f$ embeds~$\iota(X)$ into~$X$.
            Fix a pair~$(f,\bX)$ where~$f$ is a good embedding.
            By Subclaim~\ref{sub:notincR} and (\ref{eq:transfer}), we obtain that~$f(R_\star^\bX) \not\subseteq Q$.
            Therefore,~$f(R_\star^\bX) \cap A \neq \emptyset$, i.e., there exists a pattern in~$f(R_\star^\bX)$ that is a missing pattern.
            Let~$\xi(f,\bX)$ be such a pattern. 
            Since~$X \subseteq Q$,~$R_{\star}^{\bX} = R_{\star}\cup \iota(X)$, and~$f$ embeds~$\iota(X)$ into~$X$ we must have that~$\xi(f,\bX) \in f(R_{\star})$ and so in particular~$x\notin\xi(f,\bX)$.
            By property (i) of Lemma \ref{lem:rigidity}, this implies that $\xi(f,\bX)$ is a non-degenerate pattern (i.e., it has three distinct colors).
            For this reason, we will only estimate the number of non-degenerate missing patterns not containing~$x$.
			
			For $i \in [t]$, let $c_i(\bX)$ be the number of distinct colors of $U_i$ present in the patterns of $\iota(X)$, and let $c(\bX) = \sum_{i=1}^{t} c_i(\bX)$.
            Note that~$c(\bX)\leq2t$ since each of the~$t$ patterns in~$X$ contains~$x$.
            For a fixed~$\bX\in \prod_{k=1}^t B_x^{(k)}$, the number of good embeddings~$f$ is at least the number of ways to select, for each~$i \in [t]$,~$\left(|U_i|-c_i(\bX)\right)$ vertices from~$V_i \setminus \{X \cap V_i \}$.
            Therefore, for~$N$ large enough, the number of distinct pairs $(f,\bX)$ can be lower bounded by
			\begin{align}\label{eq:numberoffx}
				\sum_{\bX\in \prod_{k=1}^t B_x^{(k)}} \prod_{i=1}^t \left(\frac{|V_i|}{2}\right)^{|U_i|-c_i(\bX)}
				&\geq \sum_{\bX\in \prod_{k=1}^t B_x^{(k)}}\left(\frac{\beta n}{2}\right)^{M_1-c(\bX)} \nonumber\\
				&\geq  \left(\frac{\beta n}{2}\right)^{M_1-2t}\prod_{k=1}^t|B_x^{(k)}|\geq \gamma^t n^{M_1} \left(\frac{\beta}{2}\right)^{M_1-2t}\,, 
			\end{align}
			where we use (\ref{eq:sizeofVi}) and (\ref{eq:Bix}).
            Moreover, for a fixed non-degenerate missing pattern~$p$ not containing~$x$, the number of pairs~$(f,\bX)$ such that~$p \in f(R_\star^\bX)$ is at most~$n^{M_1-3}$ (since the pattern~$p$ and the vertex~$x$ are fixed).
            Therefore, the number of missing patterns can be estimated by
			\begin{align*}
				|A| \geq \gamma^t n^3\left(\frac{\beta}{2}\right)^{M_1-2t} > \alpha n^3
			\end{align*}
			by our choice of $\alpha$ and $\gamma$ in (\ref{eq:alphagamma}). However, this contradicts (\ref{eq:missingbad2}), which concludes the proof.
		\end{proof}
		
		By applying a similar argument as in the last claim, one can obtain the following.
		
		\begin{claim}\label{clm:missingcount}
			For every bad pattern $q_B \in B$, there exist at least $(\beta/4)^{M_1} n^2$ non-degenerate missing patterns $q_A \in A$ such that $|C(q_A) \cap C(q_B)| = 1$.
		\end{claim}
		
		\begin{proof}
			Fix~$q_B = (a,b,c) \in B$.
            Let~$\chi: C(q_B) \to [t]$ be the~$t$-coloring such that~$q_B \in V_{\chi(a)} \times V_{\chi(b)} \times V_{\chi(c)}$.
            Let~$R_{\star}$ be the non-degenerate rigid palette on~$M_1$ colors obtained by Lemma~\ref{lem:rigidity} with partition~$[M_1] = \bigcup_{i=1}^t U_i$. 
            We define the palette~$R_\star^{q_B} \subseteq [M_1 + c(q_B)]^3$ as follows.
            Let~$C(q_B) = \{x_1, \ldots, x_{c(q_B)}\}$ (note that~$c(q_B)$ can be either~$2$ or~$3$, and that~$\{a, b, c\} = \{x_1, \ldots, x_{c(q_B)}\}$ as sets). 
            For~$1 \leq k \leq c(q_B)$, let~$R_k$ be a palette with color set~$[M_1] \cup \{M_1+k\}$ obtained from~$R_{\star}$ by blowing up a color in the set $U_{\chi(x_k)}$ to form the new vertex~$M_1+k$.
            Let~$p_B \in \{M_1+1, \ldots, M_1+c(q_B)\}^3$ be the pattern obtained by sending each color~$x_k$ in the pattern~$q_B$ to the color~$M_1+k$.
            Then we define~$R_\star^{q_B} := \left( \bigcup_{k=1}^{c(q_B)} R_k \right) \cup \{p_B\}$ and remark that in this construction the only pattern in~$R_\star^{q_B}$ which meets~$\{M+1,\dots,M+c(q_B)\}$ in more than one color is~$p_B$.
			
			\begin{subclaim}\label{sub:notincR2}
				$R_\star^{q_B} \in \cR$.
			\end{subclaim}
			\begin{proof}
				As in the proof of Subclaim \ref{sub:notincR}, it suffices to check that~$R_\star^{q_B}$ is not contained in a blow-up of~$P$. 
                Suppose, for the sake of contradiction, that~$R_\star^{q_B}$ is contained in a blow-up~$S'$ of~$P$ with partition structure~$\bigcup_{i=1}^t V_i'$.
                This in particular implies that~$R_\star \subseteq S'$ and~$R_k \subseteq S'$ for each~$1 \leq k \leq c(q_B)$.
                By the rigidity of~$R_\star$, we may assume without loss of generality that~$U_i \subseteq V_i'$ for $i \in [t]$.
                Moreover, by property (iii) of Lemma~\ref{lem:rigidity}, we also obtain that~$M_1+k \in V'_{\chi(x_k)}$.
                Thus, and since~$q_B \in V_{\chi(a)} \times V_{\chi(b)} \times V_{\chi(c)}$, we obtain that~$p_B \in V'_{\chi(a)} \times V'_{\chi(b)} \times V'_{\chi(c)}$.
                This implies that~$(\chi(a), \chi(b), \chi(c)) \in P$, which contradicts the fact that~$q_B \in B$ is a bad pattern. 
			\end{proof}
			
			Our goal now is to lower bound the cardinality of the missing patterns $A = S \setminus Q$. We define the subsets $A_0, A_1 \subseteq A$ as
			\begin{align*}
				A_0 & := \{q_A \in A : c(q_A) = 3,\, C(q_A) \cap C(q_B) = \emptyset\}\,, \\
				A_1 & := \{q_A \in A : c(q_A) = 3,\, |C(q_A) \cap C(q_B)| = 1\}\,.
			\end{align*}
			That is,~$A_0$ consists of the non-degenerate missing patterns that are disjoint from~$q_B$, and~$A_1$ consists of those that intersect~$q_B$ in exactly one color.
            We say that an embedding $f : R_\star^{q_B} \to Q \cup S$ is \emph{good} if~$f$ sends~$p_B$ to~$q_B$. 
            Let~$\Psi$ be the set of all good embeddings~$f$.
            Fix~$f \in \Psi$.
            By Subclaim~\ref{sub:notincR2} and (\ref{eq:transfer}), we obtain that~$f(R_\star^{q_B}) \not\subseteq Q$.
            Therefore,~$f(R_\star^{q_B})$ must contain a missing pattern~$\xi(f) \in A$. 
            Since~$q_B \in B$, we obtain that~$\xi(f) \in f(R_\star^{q_B} \setminus \{p_B\})$. 
            By recalling that the only pattern in~$R_\star^{q_B}$ which meets~$C(p_B)$ in more than one color is~$p_B$, this implies that $\xi(f) \in A_0 \cup A_1$. 
			Using (i) and (iii) of Lemma \ref{lem:rigidity}, the number of good embeddings $f$ can be lower bounded by
			\begin{align*}
				|\Psi| & \geq \prod_{i=1}^t \left(\frac{|V_i|}{2}\right)^{|U_i|} \geq \left(\frac{\beta n}{2}\right)^{M_1}\,,
			\end{align*}
			where we use (\ref{eq:sizeofVi}). We can split the argument into two cases:
			
			\vspace{0.2cm}
			\noindent \underline{Case 1:} For at least $|\Psi| / 2$ embeddings $f$, the pattern $\xi(f)$ is in $A_0$.
			
			\vspace{0.2cm}
			
			For a pattern $q_A \in A_0$, the number of $f \in \Psi$ such that $q_A \in f(R_\star^{q_B})$ is at most $n^{M_1-3}$. Therefore,
			\begin{align*}
				|A| & \geq |A_0| \geq \frac{|\Psi|}{2 n^{M_1-3}} \geq \frac{n^3}{2} \left(\frac{\beta}{2}\right)^{M_1} > \alpha n^3,
			\end{align*}
			by our choice of $\alpha$ in (\ref{eq:alphagamma}). However, this contradicts (\ref{eq:missingbad2}).
			
			\vspace{0.2cm}
			
			\noindent \underline{Case 2:} For at least $|\Psi| / 2$ embeddings $f$, the pattern $\xi(f)$ is in $A_1$.
			
			\vspace{0.2cm}
			
			For a pattern $q_A \in A_1$, the number of $f \in \Psi$ such that $q_A \in f(R_\star^{q_B})$ is at most $n^{M_1-2}$ (since $|C(q_A) \cap C(q_B)| = 1$ and $c(q_A) = 3$).
            Therefore,
			\begin{align*}
				|A_1| & \geq \frac{|\Psi|}{2 n^{M_1-2}} \geq \frac{n^2}{2} \left(\frac{\beta}{2}\right)^{M_1} > \left(\frac{\beta}{4}\right)^{M_1} n^2\,.
			\end{align*}
			This concludes the proof of the claim.
		\end{proof}
		
		To finish the proof, we double count the number of pairs~$(q_A, q_B)$ where~$q_B \in B$ is a bad pattern and~$q_A \in A$ is a non-degenerate missing pattern such that~$|C(q_A) \cap C(q_B)| = 1$.
        Fix~$q_B\in B$.
        By Claim~\ref{clm:missingcount}, there exist at least~$(\beta / 4)^{M_1} n^2$ patterns~$q_A$. 
        Therefore, there are at least~$(\beta / 4)^{M_1} n^2 |B|$ such pairs~$(q_A, q_B)$.
        On the other hand, for a fixed~$q_A\in A$, there are at most~$3\Delta(B)$ patterns~$q_B$ such that~$|C(q_A) \cap C(q_B)| = 1$. 
        Consequently, by Claim~\ref{clm:maxB}, the number of pairs~$(q_A, q_B)$ is at most~$3\gamma n^2 |A|$.
        Combining the two bounds yields, by (\ref{eq:alphagamma}), that
		\begin{align*}
			|A| \geq \frac{(\beta / 4)^{M_1}}{3\gamma} |B| > |B|,
		\end{align*}
		which contradicts (\ref{eq:missingbad2}). Therefore $|A|=|B|=0$ and consequently $Q=S$, concluding the proof of the lemma.
	\end{proof}
	
	We are now able to prove Theorem \ref{thm:expal_eq_blowup} as described in the outline at the beginning of the section.
	
	\begin{proof}[Proof of Theorem \ref{thm:expal_eq_blowup}]
	       We start by setting up the constants. Let $M^{P}_{\ref{lem:stability}}$, $N^{P}_{\ref{lem:stability}}$, $\alpha^{P}_{\ref{lem:stability}}$ and $M^{\rev(P)}_{\ref{lem:stability}}$, $N^{\rev(P)}_{\ref{lem:stability}}$, $\alpha^{\rev(P)}_{\ref{lem:stability}}$ be the constants obtained by applying Lemma \ref{lem:stability} to the palettes $P$ and $\rev(P)$, respectively (Lemma \ref{lem:stability} applies to~$\rev(P)$ because, as remarked in Section \ref{sec:properties}, it is reduced if and only if~$P$ is).
		Let $M_{\ref{cor:palette_removal}}$ and $N_{\ref{cor:palette_removal}}$ be the constants obtained by applying Corollary~\ref{cor:palette_removal} to the palette~$P$ with $\alpha=\frac{1}{2}\min\{\alpha^{P}_{\ref{lem:stability}},\alpha^{\rev(P)}_{\ref{lem:stability}}\}$.
        Set $N_0:=\max\{N^{P}_{\ref{lem:stability}},N^{\rev(P)}_{\ref{lem:stability}},N_{\ref{cor:palette_removal}}\}$.
        Consider the family of palettes~$\cR$ given by
		\begin{align*}
			\cR:=\{R:\: R &\text{ is a palette with $c(R)\leq N_0$}\\ &\text{ and }R\text{ is not contained in a blow-up of } P \text{ nor in a blow-up of } \rev(P)\}\,.\
		\end{align*}
		For every $R\in \cR$, by Lemma \ref{lem:distinguish_palette}, there exists a~$3$-graph~$F_R$ such that~$R$ paints~$F_R$ and~$P$ does not paint~$F_R$.
		Let~$M_4:=\max\limits_{R\in \cR}\{v(F_R)\}$ and note that this is well-defined since~$\cR$ is finite.
		Since~$\{F_R\}_{R\in\cR}\subseteq\cF(P)$ and by our choice of~$M_4$, every palette in~$\cR$ paints a~$3$-graph in~$\cF(P)_{M_4}$.
		In other words, we have that
		\begin{align}\label{eq:transfer2}
			\text{if } Q \text{ is } \cF_{M_4}(P)\text{-deficient, then } Q \text{ is } \cR\text{-free.}
		\end{align}
		Set $M:=\max\{M^{P}_{\ref{lem:stability}},M^{\rev(P)}_{\ref{lem:stability}},M_{\ref{cor:palette_removal}},M_4\}$. We claim that the family $\cH:=\cF(P)_M$ satisfies the hypothesis of the theorem. 
		
		Let~$n\in\mathds{N}$ and let~$Q\in \EXpal(n,\cH)$ be a palette with~$c(Q)=n$ that maximizes the number of patterns among all~$\cH$-deficient palettes.
		Suppose first that~$n\leq N_0$. By~\eqref{eq:transfer2}, the palette~$Q$ is~$\cR$-free, and then the definition of~$\cR$ and~$c(Q)\leq N_0$ entail that~$Q$ is contained in a blow-up of~$P$ or in a blow-up of~$\rev(P)$ and so~$Q$ is among the palettes which are considered in the right-hand side of~\eqref{eq:maintech}.
		Hence, we may assume that~$n\geq N_0$. 
		Now, Corollary~\ref{cor:palette_removal} implies that~$Q$ is~$\alpha$-close to being contained in a blow-up of~$P$ of $\rev(P)$. 
		Therefore, by Lemma~\ref{lem:stability}, $Q$ is a blow-up of $P$ or $\rev(P)$. 
	\end{proof}

	\section{Proof of main theorem}\label{sec:main}
	
	In this section we prove Theorem~\ref{thm:main}.
	As mentioned before, it essentially follows from Theorem~\ref{thm:expal_eq_blowup} and the fact that~$\pi_{\pal}(\cF)=\pi_{\vvv}(\cF)$ holds for finite families.
	This equality in turn follows from the work in~\cite{R:20} (which is implicit in~\cite{RRS:18}) and~\cite{L:24}.
	To expand on this argument, we recall the notion of reduced hypergraphs from~\cite{RRS:18}.
	Essentially, reduced hypergraphs capture the setting that one arrives at after applying hypergraph regularity.
	
	\begin{dfn}\label{dfn:reduced_hypergraph}
		A reduced~$3$-graph is a triple~$\ccA=(I,\{\cP^{ij}\}_{ij\in I^{(2)}},\{\ccA^{ijk}\}_{ijk\in I^{(3)}})$ consisting of a finite index set~$I$, a collection of pairwise disjoint sets of vertices~$\{\cP^{ij}\}_{ij\in I^{(2)}}$, and a collection of~$3$-partite~$3$-graphs~$\{\ccA^{ijk}\}_{ijk\in I^{(ijk)}}$ such that for every~$ijk\in I^{(3)}$ the vertex classes of~$\ccA^{ijk}$ are~$\cP^{ij}$,~$\cP^{ik}$, and~$\cP^{jk}$.
		If $d\in [0, 1]$ and $e(\ccA^{ijk})\ge d|\cP^{ij}||\cP^{ik}||\cP^{jk}|$ holds for all~$ijk\in I^{(3)}$, we say that $\ccA$ is {\it $(d, \vvv)$-dense}. 
	\end{dfn}
	
	To ease the notation, we often simply write a reduced~$3$-graph as~$\ccA=(I,\cP^{ij},\mathcal{A}^{ijk})$.
	By the hypergraph embedding lemma, in order to find a copy of a~$3$-graph~$F$ in the original host~$3$-graph~$H$, it is sufficient to find a ``reduced map'' of~$F$ to a suitable reduced~$3$-graph.
	This is made formal with the following definition.
	
	\begin{dfn}\label{dfn:reduced_map}
		A {\it reduced map} from a~$3$-graph $F$ to a reduced~$3$-graph~$\ccA=(I,\cP^{ij},\mathcal{A}^{ijk})$ is a pair 
		$(\lambda, \phi)$ such that 
		\begin{enumerate}[label=\rmlabel]
			\item\label{it:rm1} $\lambda\colon V(F)\longrightarrow I$  and  
			$\phi\colon \partial F\longrightarrow \bigcup_{ij\in I^{(2)}}\cP^{ij}$, where $\partial F$ denotes the 
			set of all pairs of vertices covered by an edge of $F$;
			\item\label{it:rm2} if $uv\in\partial F$, then $\lambda(u)\ne \lambda(v)$ and 
			$\phi(uv)\in\cP^{\lambda(u)\lambda(v)}$;
			\item\label{it:rm3} if $uvw\in E(F)$, then 
			$\phi(uv)\phi(uw)\phi(vw)
			\in E(\ccA^{\lambda(u)\lambda(v)\lambda(w)})$.
		\end{enumerate}
		
		If some such reduced map exists, we say that $\ccA$ {\it contains a reduced image 
			of $F$}, and otherwise $\ccA$ is called {\it $F$-free}. 
	\end{dfn}
	
	Given a family~$\cF$ of~$3$-graphs we say that a reduced~$3$-graph~$\ccA$ is~$\cF$-free if it is~$F$-free for all~$F\in\cF$.
	One can now define the Tur\'an density of a family of~$3$-graphs with respect to reduced~$3$-graphs.
	
	\begin{dfn}\label{dfn:pired}
		If~$\cF$ is a family of~$3$-graphs, then 
		\begin{align*} 
			\pired(\cF)  = \sup\bigl\{d\in [0, 1]\colon & \text{For every $m\in\NN$ } \text{there 
				is a $(d, \vvv)$-dense, } \notag \\
			& \text{$\cF$-free, reduced~$3$-graph with an index set of size $m$}\bigr\}\,. 
		\end{align*}
	\end{dfn}
	
	The key behind almost all the progress on the uniform Tur\'an problem in the past decade is that an argument based on the hypergraph regularity method yields the following result.
	
	\begin{theorem}[Theorem 3.3 in \cite{R:20}, implicit in~\cite{RRS:18}]\label{thm:reduced_to_uniform}
		If~$\cF$ is a finite family of~$3$-graphs, then $$\pired(\cF) = \pivvv(\cF).$$
	\end{theorem}
	
	For the next lemma, we need to set up the following notation.
	Let~$\ccA=(U,\cP^{ij},\ccA^{ijk})$ be a reduced~$3$-graph and let~$\cS^{ij}\subseteq\cP^{ij}$ be multisets, for all~$ij\in U^{(2)}$.
	For all~$ij\in U^{(2)}$, set~$(\cS^{ij})'=\{(x,r):x\in\cS^{ij},r\in[\ell(x)]\}$, where~$\ell(x)$ is the multiplicity of~$x$ in~$\cS^{ij}$.
	Next, for all~$ijk\in U^{(3)}$, let~$(\ccA^{ijk})'$ be the~$3$-partite~$3$-graph with vertex classes~$(\cS^{ij})'$,~$(\cS^{ik})'$, and~$(\cS^{jk})'$ and edge set $$\{(x,a)(y,b)(z,c):xyz\in\ccA^{ijk},a\in[\ell(x)],b\in[\ell(y)],c\in[\ell(z)]\}\,.$$

    The following was a crucial technical lemma used in~\cite{L:24}, to obtain a palette from a reduced~$3$-graph with the appropriate dependence on~$\epsilon$ and~$m$.
	
	\begin{lemma}\label{lem:reduced_to_palette}
		For all~$\epsilon > 0$ there is some~$s\in\mathds{N}$ such that for all~$m\in\mathds{N}$ there is some~$N\in\mathds{N}$ with the following property.
		Every reduced~$3$-graph~$\ccA=([N],\cP^{ij},\ccA^{ijk})$ with density~$d\in[0,1]$ contains an index subset~$U \subseteq [N]$ with~$|U| \geq m$ and multisets~$\cS^{ij} \subseteq \cP^{ij}$, for all~$ij \in U^{(2)}$, such that each~$|\cS^{ij}|=s$ and the reduced~$3$-graph~$(U,(\cS^{ij})',(\ccA^{ijk})')$ is~$(d-\epsilon,\vvv)$-dense.
	\end{lemma}

        By applying Theorem~\ref{thm:reduced_to_uniform}, the following theorem suffices to complete the proof of Theorem~\ref{thm:main}.
	
	\begin{theorem}\label{thm:pired_eq_lagrangian}
		Let~$P$ be a palette.
		Then there exists a finite family~$\cH$ of~$3$-graphs so that~$\pired(\cH) = \Lambda_P$.
	\end{theorem}
	\begin{proof}
		If~$P$ is not reduced, there exists a reduced palette~$P'\subsetneq P$ with~$\Lambda_{P'}=\Lambda_P$ that we could consider instead, so we may assume that~$P$ is reduced.
		Then Theorem~\ref{thm:expal_eq_blowup} yields a finite family~$\cH$ such that~$P$ is~$\cH$-deficient and for all~$n\in\mathds{N}$
		\begin{align}\label{eq:expalgood}
			\expal(n,\cH)=\max\{e(Q):Q\text{ is a blow-up of $P$ and $c(Q)=n$}\}\,.
		\end{align}
		We will show that for every~$\varepsilon>0$ we have~$\Lambda_P-\varepsilon\leq\pired(\cH)\leq\Lambda_P+2\varepsilon$.
		So let~$\epsilon>0$ be given.
		First we show that~$\Lambda_P-\epsilon \leq \pired(\cH)$.
		By~\eqref{eq:lambda_blowup}, there is some~$n_0$ such that for every~$n\in\mathds{N}$ with~$n\geq n_0$ there is a palette~$Q$ with~$n$ colors which is a blow-up of~$P$, attains the maximum on the right-hand side in~\eqref{eq:expalgood}, and satisfies~$\frac{e(Q)}{n^3}\geq \Lambda_P-\epsilon$.
        Now it follows from (\ref{eq:lambda_density}), Fact~\ref{lem:palette_lb}, and Theorem~\ref{thm:reduced_to_uniform} that $$\Lambda_P-\varepsilon\leq\frac{e(Q)}{n^3}\leq\Lambda_{Q}\leq\pivvv(\cH)=\pired(\cH)\,.$$
		
		Next we show that~$\pired(\cH)-2\epsilon \leq \Lambda_P$.
        First we observe that for all~$s\in\NN$ we have
        \begin{align}\label{eq:approxLambda}
            \frac{\expal(s,\cH)}{s^3}\leq\Lambda_P\,.
        \end{align}
        Indeed, by~\eqref{eq:expalgood} and~\eqref{eq:lambda_density}, there is a blow-up of~$Q$ with~$c(Q)=s$ such that~$\expal(s,\cH)/s^3=d(Q)\leq\Lambda_Q$.
        Since~$Q$ is a blow-up of~$P$, Observation~\ref{obs:QcontainP} entails that~$\Lambda_Q\leq\Lambda_P$, and~\eqref{eq:approxLambda} follows.
        
        Now let~$m$ be the maximum number of vertices of any~$H\in\cH$ and let~$s\in\mathds{N}$ be given by applying Lemma~\ref{lem:reduced_to_palette} with~$\varepsilon$.
        Next, let~$R$ be the Ramsey number~$R_3(m;2^{s^3})$ (i.e.,~$R$ is the smallest integer such that any coloring of the~$3$-edges of~$K_R^{(3)}$ with~$2^{s^3}$ colors contains a monochromatic~$K_m^{(3)}$).
        Finally, let~$N\in\mathds{N}$ be as guaranteed by (the conclusion of) Lemma~\ref{lem:reduced_to_palette} applied to~$R$ here instead of~$m$ there.
		Let~$\ccA=([N],\cP^{ij},\ccA^{ijk})$ be a~$(\pired(\cH) - \epsilon,\vvv)$-dense $\cH$-free reduced~$3$-graph.
		The conclusion of Lemma~\ref{lem:reduced_to_palette} provides an index subset~$U \subseteq [N]$ with~$|U| \geq R$ and multisets~$\cS^{ij} \subseteq \cP^{ij}$ with~$\vert \cS^{ij}\vert=s$, for all~$ij \in U^{(2)}$, such that the reduced~$3$-graph~$(U,(\cS^{ij})',(\ccA^{ijk})')$ is~$(\pired(\cH)-2\epsilon,\vvv)$-dense.
		For all~$ij\in U^{(2)}$ we identify~$(\cS^{ij})'$ (arbitrarily) with~$[s]$.
		Then each~$(\ccA^{ijk})'$ can be viewed as one of the~$s^3$ possible subsets of~$[s]^3$.
		This yields a~$2^{s^3}$-coloring of~$U^{(3)}$, whence our choice of~$R$ provides an index set $U' \subseteq U$ with $\vert U'\vert =m$ and a subset~$G\subseteq[s]^3$ so that for each~$ijk \in (U')^{(3)}$ with~$i<j<k$,~$(\ccA^{ijk})'$ corresponds to~$G$ (under the fixed identifications of~$S^{ij}$,~$S^{ik}$, and~$S^{jk}$ with~$[s]$).
		Now~$G$ can naturally be interpreted as a palette~$G'$ with~$C(G')=[s]$ and~$E(G')=G$.
		Since~$e((\ccA^{ijk})')\geq(\pired(\cH)-2\epsilon)s^3$, it follows that~$d(G') \geq \pired(\cH)-2\epsilon$.
		Further, it can easily be checked that if~$G'$ paints any~$H\in\cH$, this would entail a reduced image of~$H$ in~$(U,(\cS^{ij})',(\ccA^{ijk})')$ and thus in~$\ccA$.
		Hence,~$\ccA$ being~$\cH$-free, we know that~$G'$ is~$\cH$-deficient.
		Therefore, using~(\ref{eq:approxLambda}), we get
		$$\pired(\cH)-2\epsilon \leq \frac{e(G')}{s^3} \leq \frac{\expal(s,\cH)}{s^3}\leq\Lambda_P\,.$$
        This concludes the proof of the theorem.
	\end{proof}
	
	\section{Concluding remarks}
	
	In this work, we obtain that the Lagrangian of any finite palette is attained as the uniform Tur\'{a}n density of a finite family of~$3$-graphs.
    A consequence of this result combined with~\cite{L:24} is that
    \begin{align}\label{eq:summary}
        \Lambda_{\pal} \subseteq \Pi_{\vvv, \fin} \subseteq \Pi_{\vvv,\infty}\subseteq\widebar{\Lambda}_{\pal}\,.
    \end{align}
    It would be interesting to determine which of these inclusions are strict. 
    In addition, if~$\Pi_{\vvv}$ denotes the set of uniform Tur\'{a}n densities of single~$3$-graphs, is it true that~$\Pi_{\vvv,} \subsetneq \Pi_{\vvv, \fin} \subsetneq  \Pi_{\vvv,\infty}$?
    In \cite{P:12}, it was proved that the set of Tur\'{a}n densities of possibly infinite families of~$k$-graphs,~$\Pi^{(k)}_{\infty}$ is uncountable and closed for~$k\geq 3$.
    One could ask whether similar statements hold for~$\Pi_{\vvv,\infty}$.
    We remark that a direct application of the methods used in~\cite{P:12} does not seem to work here. 

    We say that~$d \in [0,1)$ is a \emph{jump} in a set~$X \subseteq [0,1]$ if there exists some~$\varepsilon>0$ such that $(d,d+\varepsilon)\cap X=\emptyset$.
    Erd\H{o}s~\cite{E:64} showed that for every~$k\geq2$,~$0$ is a jump in~$\Pi_{\infty}^{(k)}$.
    On the other hand, Frankl and R\"{o}dl~\cite{FR:84} proved that for every~$k\geq3$ there is some~$d \in [0,1)$ that is not a jump in $\Pi^{(k)}_{\infty}$, disproving the famous Erd\H{o}s jumping conjecture.
    For the uniform Tur\'{a}n density, Reiher, R\"{o}dl and Schacht~\cite{RRS:18} showed that $0$ is a jump for~$\Pi_{\vvv,\infty}$.
    A consequence of our work is that every non-jump in~$\Pi^{(3)}_{\infty}$ yields a non-jump in~$\Pi_{\vvv,\fin}$ (see~\cite{KSS:24}).

    Note that the palettes considered here, as well as in~\cite{KSS:24} and~\cite{L:24}, have finitely many colors.
    One might ask what the situation looks like for a (countably) infinite palette, which is a pair~$P=(C,E)$ consisting of an infinite set of colors~$C$ and an infinite set of patterns~$E\subseteq C^3$.

    To define the Lagrangian of an infinite palette, note that the Lagrangian of a (finite) hypergraph~$F$ is (up to scaling) simply the maximum edge density a blow-up of~$F$ can have.
    Following this spirit, a reasonable way to define the ``Lagrangian'' of an infinite palette~$P$ is
    \begin{align*}
        \Lambda_P=\sup\{d\in[0,1]: \: &\text{ for every $\eta>0$ and $n\in\mathds{N}$, there is} \\ &\text{a $(d,\eta)$-dense $3$-graph $H$ with $v(H)\geq n$ such that $P$ paints $H$}\}\,.
    \end{align*}
    From~\cite{KSS:24} it follows that for finite palettes this definition is equivalent to our previous definition.
    Now one can ask whether Theorem~\ref{thm:main} still holds for every infinite palette~$P$.
    Lamaison and Wu~\cite{LW:25+} announced that there exists a~$3$-graph~$F$ such that there is no finite palette~$P$ that is~$F$-deficient and satisfies~$\Lambda_P=\pivvv(F)$.
    Note that this means that~$\Pi_{\vvv,\text{fin}}\subsetneq \Lambda_{\pal}$. 
    Setting~$\Lambda_{\pal,\infty}=\{\Lambda_P:\: P\text{ is a finite or infinite palette}\}$, it would be curious if in fact any of the sets~$\Lambda_{\pal,\infty}$,~$\Pi_{\vvv,\text{fin}}$, and~$\Pi_{\vvv,\infty}$ are equal.
    
	\bibliographystyle{plain}
	\bibliography{refs} 
	
	\appendix
	\section{Weak Palette Regularity}\label{app:regularity}
	Before we begin with details we provide a brief overview. Our proof of Theorem~\ref{thm:palette_regularity} will be as expected. We will define an energy function~$q$ which takes as argument a partition of~$C(Q)$ and returns an element of~$[0,1]$, and show both that refining a partition cannot decrease the value of~$q$, and that if~$\cA$ fails to be~$\epsilon$-regular then there exists a refinement which substantially increases~$q$.
	Once we have proven Theorem~\ref{thm:palette_regularity}, Corollary \ref{cor:iterated_palette_regularity_app} will be obtained through the following intermediary result. 
	
	\begin{theorem}\label{thm:iterated_palette_regularity}
		For all non-increasing~$\mathcal{E}(r): \NN \to (0,1]$ and~$m$ there exists~$M=M_{\ref{thm:iterated_palette_regularity}},N=N_{\ref{thm:iterated_palette_regularity}}$ so that given any palette~$Q$ with $c(Q)\geq N$ there is an equipartition $\cA=\{V_i : i \in [t] \}$ of $C(Q)$ with refinement $\cB = \{V_{i,j}:i \in [t], j \in [\ell]\}$ so that:
		\begin{enumerate}[label = \nlabel]
			\item \label{it:part_count}$m \leq t$ and $t\ell \leq M$,
			\item \label{it:eps0} all but $\cE(0) t^3$ of $(i_1,i_2,i_3) \in [t]^3$ have the ordered triple $(V_{i_1},V_{i_2},V_{i_3})$~$\cE(0)$-regular,
			\item \label{it:epst} for all~$(i_1,i_2,i_3) \in [t]^3$, all but~$\cE(t)\ell^3$ of $(j_1,j_2,j_3) \in [\ell]^3$ have the ordered triple  $(V_{i_1,j_1},V_{i_2,j_2},V_{i_3,j_3})$~$\cE(t)$-regular, and
			\item \label{it:palette-reg-density} for all but $\cE(0)t^3$ of~$(i_1,i_2,i_3) \in [t]^3$, all but~$\cE(0)\ell^3$ of~$(j_1,j_2,j_3) \in [\ell]^3$ have\newline~$|d(V_{i_1,j_1},V_{i_2,j_2},V_{i_3,j_3})-d(V_{i_1},V_{i_2},V_{i_3})|< \cE(0)$.
		\end{enumerate}
	\end{theorem}
	
	The relationship between our~Theorem~\ref{thm:palette_regularity}, Corollary~\ref{cor:iterated_palette_regularity_app}, and Theorem~\ref{thm:iterated_palette_regularity} is entirely analogous to those found in \cite{AFKS:00} between Lemma 3.3 (the traditional Szemer\'edi graph regularity lemma), Lemma 4.1, and Corollary 4.2 there. Our proofs use the same strategies.

    Let us begin now with the details of Theorem~\ref{thm:palette_regularity}.
    We follow closely the probabilistic techniques presented in~\cite{AS:15}, working here with palettes in place of graphs. For ease of notation we will still refer to colors as vertices, and patterns as edges on these vertices.
    We are interested in partitions of~$C(Q)$, and we let~$n=c(Q)$ throughout. We will, in intermediate steps, allow our equipartitions to include a small set~$V_0$ of exceptional vertices.
    As part of a partition of~$C(Q)$ we consider~$V_0$ as composed of singleton vertices, so that~$\cB = U_0 \dcup U_1 \dcup \dots \dcup U_{t'}$ refines~$\cA = V_0 \dcup V_1 \dcup \dots \dcup V_t$ (denoted~$\cB \prec \cA$) so long as, for each~$i \in [t]$,~$V_i$ is obtained exactly as the union of some $U_j$ together with some vertices from~$U_0$, and~$V_0 \subseteq U_0$. 
    Recalling Definition~\ref{dfn:palette_regularity}, we say that a partition~$\cA= V_0 \dcup V_1 \dcup \dots \dcup V_t$ is~$\epsilon$-regular if for all but~$\epsilon t^3$ of $(i_1,i_2,i_3) \in [t]^3$, the ordered triple~$(V_{i_1},V_{i_2},V_{i_3})$ is~$\epsilon$-regular. When we index over~$V \in \cA$ for a partition~$\cA$, we mean to take one term for each part of~$\cA$, denoted~$V$.
    With these notions in mind we can now define the energy~$q$.
	
	\begin{dfn}
		Suppose that~$V_1,V_2,V_3 \subseteq C(Q)$ with~$c(Q)=n$. Then$$q(V_1,V_2,V_3) := \frac{|V_1||V_2||V_3|}{n^3}d^2(V_1,V_2,V_3)\,.$$ If~$\cA_1,\cA_2,\cA_3$ are partitions of~$C(Q)$, we set$$q(\cA_1,\cA_2,\cA_3) := \sum_{V_1 \in \cA_1,V_2  \in \cA_2, V_3 \in \cA_3}q(V_1,V_2,V_3)\,.$$ We will use~$q(\cA)$ to refer to~$q(\cA,\cA,\cA)$ along a single partition. 
	\end{dfn}
	
	Since~$q(V_1,V_2,V_3) \leq \frac{|V_1||V_2||V_3|}{n^3}$ it is immediate that~$q(\cA_1,\cA_2,\cA_3) \in [0,1]$.
    The following technical lemma captures two more important properties we will require of the function~$q$.
	
	\begin{lemma}\label{lem:q_props}
        
        If~$\cB_1 \prec \cA_1,\cB_2 \prec \cA_2$, and~$\cB_3 \prec \cA_3$ then~$q(\cB_1,\cB_2,\cB_3) \geq q(\cA_1,\cA_2,\cA_3)$. Furthermore, if~$V_1,V_2,V_3 \subseteq C(Q)$ are so that~$(V_1,V_2,V_3)$ is not~$\epsilon$-regular, then there are partitions~$V_1=V_1^{1}\dcup V_1^{2},V_2=V_2^{1}\dcup V_2^{2},V_3=V_3^{1}\dcup V_3^{2}$ so that $$q(V_1^{1}\dcup V_1^{2},V_2^{1}\dcup V_2^{2},V_3^{1}\dcup V_3^{2}) \geq q(V_1,V_2,V_3)+\epsilon^5 \frac{|V_1||V_2||V_3|}{n^3}\,.$$

	\end{lemma}
	\begin{proof}

		For the first part, it suffices to check the case when~$\cA_1=V_1,\cA_2=V_2,\cA_3=V_3$ consist of a single set each, since any~$\cB_i$ can be obtained by successive refinement in this way. 
			In this case we define a random variable~$Z$ as follows. Select vertices~$x_1 \in V_1,x_2 \in V_2,x_3 \in V_3$ uniformly at random and let~$U_1 \in \cB_1,U_2 \in \cB_2, U_3 \in \cB_3$ be the unique parts of their respective partitions so that~$x_1 \in U_1, x_2 \in U_2, x_3 \in U_3$ before setting~$Z = d(U_1,U_2,U_3)$.
			We can directly compute both the expectation
			\begin{align*}\bbE(Z) &= \sum_{U_1 \in \cB_1, U_2 \in \cB_2, U_3 \in \cB_3}\frac{|U_1||U_2||U_3|}{|V_1||V_2||V_3|}d(U_1,U_2,U_3)\\ &= \sum_{U_1 \in \cB_1, U_2 \in \cB_2, U_3 \in \cB_3}\frac{e(U_1,U_2,U_3)}{|V_1||V_2||V_3|}\\ &= d(V_1,V_2,V_3)
            \end{align*}
			and the second moment
			\begin{align*}\bbE(Z^2) &= \sum_{U_1 \in \cB_1, U_2 \in \cB_2, U_3 \in \cB_3}\frac{|U_1||U_2||U_3|}{|V_1||V_2||V_3|}d^2(U_1,U_2,U_3)\\ &= \frac{n^3}{|V_1||V_2||V_3|}\sum_{U_1 \in \cB_1, U_2 \in \cB_2, U_3 \in \cB_3}\frac{|U_1||U_2||U_3|}{n^3}d^2(U_1,U_2,U_3)\\
				& = \frac{n^3}{|V_1||V_2||V_3|}q(\cB_1,\cB_2,\cB_3)\,.\end{align*}
			Combining these yields
			~$$0 \leq \Var(Z) = \bbE(Z^2)-\bbE(Z)^2 = \frac{n^3}{|V_1||V_2||V_3|}\bigg{(}q(\cB_1,\cB_2,\cB_3)-q(V_1,V_2,V_3)\bigg{)}$$
			as needed.
            
			Now we proceed to the `furthermore' part. Let~$V_i^1$ be witness sets to the failure of~$\epsilon$-regularity and~$V_i^2$ their complements, so that~$|V_i^1| \geq \epsilon |V_i|$ but~$|d(V_1^1,V_2^1,V_3^1) - d(V_1,V_2,V_3)| \geq \epsilon$.
			Let~$Z$ be the random variable defined above on the new partition obtained of~$V_1,V_2,V_3$, where~$\cB_i = V_i^1 \dcup V_i^2$. Then Chebyshev's inequality shows
			$$\epsilon^3 \leq \frac{|V_1^1||V_2^1||V_3^1|}{|V_1||V_2||V_3|}\leq \Pr( |Z-\bbE(Z)| \geq \epsilon ) \leq \frac{\Var(Z)}{\epsilon^2}$$
			and we are done, having computed~$\Var(Z)$ above.
	\end{proof}

        The furthermore part of the previous lemma shows that irregular triples can be refined to increase the value of~$q$ by a small amount. Next we will show that if~$\cA$ has many irregular triples, by refining each we can increment~$q$ by a small constant depending only on~$\epsilon$, while still controlling the order of the partitions we create.
	
	\begin{lemma}\label{lem:irregular_incr}
		Suppose~$\cA$ is an equipartition of~$C(Q)$ into~$V_0 \dcup V_1 \dots \dcup V_t$ and~$t \geq 6$. If there are~$\epsilon t^3$ many $(i_1,i_2,i_3) \in [t]^3$ with the ordered triple~$(V_{i_1},V_{i_2},V_{i_3})$ failing to be~$\epsilon$-regular and~$|V_0| \leq \epsilon n$,then there exists a refinement~$\cB=U_0 \dcup U_1 \dots \dcup U_\ell$ of~$\cA$ with~$q(\cB) \geq q(\cA) + \frac{\epsilon^6}{16}$,~$|U_0| \leq |V_0| + \frac{n}{2^t}$ and~$\ell \leq 2t 2^{3t^2}2^t $.
	\end{lemma}
	\begin{proof}
        Any time that~$V_i$ appears as part of an irregular triple (in first, second, or third position) with distinct indices, we will apply Lemma~\ref{lem:q_props} to partition~$V_i$ into~$V_i^1$ and~$V_i^2$. Formally, for every~$i \in [t]$ and $(i_1,i_2,i_3) \in [t]^3$ with $i \in \{i_1,i_2,i_3\}$ and~$i_1,i_2,i_3$ all distinct, we define a partition~$\cV^{(i)}_{(i_1,i_2,i_3)}$ of~$V_{i}$. If~$(V_{i_1},V_{i_2},V_{i_3})$ is~$\epsilon$-regular we let~$\cV^{(i)}_{(i_1,i_2,i_3)}$ be the trivial partition consisting of the single set~$V_i$, and if~$(V_{i_1},V_{i_2},V_{i_3})$ is not~$\epsilon$-regular we let~$\cV^{(i)}_{(i_1,i_2,i_3)}$ be the partition furnished by Lemma~\ref{lem:q_props}. Each~$\cV^{(i)}_{(i_1,i_2,i_3)}$ consists of either~$1$ or~$2$ sets, so for fixed~$i$ the mutual refinement consists of at most~$2^{3t^2}$ parts. Let~$\tilde{\cB}$ be the partition obtained by mutually refining each~$V_i$ in this manner, so that (since there at least~$\epsilon t^3$ irregular triples and at most~$\epsilon 3 t^2 \leq \frac{\epsilon}{2}t^3$ of them have a repeated index)
		$$q(\tilde{\cB}) \geq q(\cA) + \frac{\epsilon}{2} t^3 \epsilon^5\frac{1}{8t^3}\,.$$
		
		At this point~$\tilde{\cB}$ has incremented~$q$ as desired; all that remains is to balance the sizes of the member sets.~$\tilde{\cB}$ has at most~$t 2^{3t^2}$ parts, say~$\tilde{U}_j$, together with the exceptional set~$V_0$ remaining from~$\cA$. Let~$\cB$ be obtained from~$\tilde{\cB}$ by taking, from each~$\tilde{U}_j$, a maximal disjoint family of sets of size~$\frac{n}{t2^{3t^2}2^t}$, say~$U_{j,k}$ for~$0 \leq k \leq K_j$, and adding all of the leftover vertices to~$V_0$ to form~$U_0$. Formally,$$U_0 = V_0 \dcup \left(\bigcup_j \tilde{U}_j \setminus \left( \bigcup_{1 \leq k \leq K_j} U_{j,k}\right) \right)\,.$$ The first part of Lemma~\ref{lem:q_props} gives~$q(\cB) \geq q(\tilde{\cB}) \geq q(\cA)+ \frac{\epsilon^6}{16}$. Overcounting the number of vertices discarded gives $|U_0| \leq |V_0| + \frac{n}{2^t}$, and flattening pairs~$(j,k)$ to a single index~$\ell$ gives an equipartition~$\cB$ with at most~$2t 2^{3t^2}2^t $members, as needed.
	\end{proof}

        Finally, the following lemma will be used to redistribute the exceptional set~$V_0$ amongst the other classes of the equipartition without destroying regularity.

        \begin{lemma}\label{lem:regular_superset}
		For all~$\epsilon >0$ there exists~$\gamma >0$ so that the following holds. Suppose that~$(V_1,V_2,V_3)$ is~$\epsilon$-regular in~$C(Q)$ and that~$X_1,X_2,X_3 \subseteq C(Q)$ with~$|X_i| \leq \gamma |V_i|$. Then~$(V_1 \cup X_1,V_2 \cup X_2, V_3 \cup X_3)$ is~$2\epsilon$-regular.
	\end{lemma}
	\begin{proof}
		We may assume that~$\gamma < \epsilon$; we will verify directly that~$(V_1 \cup X_1,V_2 \cup X_2, V_3 \cup X_3)$ is~$2\epsilon$-regular. Suppose that~$W_i \subseteq (V_i \cup X_i)$ with~$|W_i| \geq 2\epsilon |V_i \cup X_i|$, and set~$W_i^V := W_i \cap V_i$ with leftovers~$W_i^X := W_i \setminus W_i^V$. Then ~$|W_i^V| \geq |W_i| - \gamma |V_i| \geq \epsilon |V_i|$, so by the~$\epsilon$-regularity of~$(V_1,V_2,V_3)$ it follows that$$|d(W_1^V,W_2^V,W_3^V) -d(V_1,V_2,V_3)| \leq \epsilon\,.$$ 
		Next observe the simple subset
		\begin{align*}
			e(V_1,V_2,V_3) \leq e(V_1 \cup X_1,V_2 \cup X_2, V_3 \cup X_3)
		\end{align*}
		and union bounds
		\begin{align*}
			e(V_1 \cup X_1,V_2 \cup X_2, V_3 \cup X_3)\leq e(V_1,V_2,V_3) &+ |X_1||V_2 \cup X_2||V_3\cup X_3|\\&+|V_1 \cup X_1||X_2||V_3\cup X_3|\\&+ |V_1 \cup X_1||V_2 \cup X_2||X_3|\,.
		\end{align*}
		Dividing through by~$|V_1 \cup X_1||V_2 \cup X_2||V_3 \cup X_3|$ yields $$\frac{1}{(1+\gamma)^3}d(V_1,V_2,V_3) \leq d(V_1 \cup X_1, V_2 \cup X_2, V_3 \cup X_3) \leq d(V_1,V_2,V_3)+ 3 \gamma$$
		and repeating the same argument with~$W_i^V$ in place of~$V_i$ and~$W_i^X$ in place of $X_i$ gives$$\frac{1}{(1+\frac{\gamma}{\epsilon})^3}d(W_1^V,W_2^V,W_3^V)\leq d(W_1,W_2,W_3) \leq d(W_1^V,W_2^V,W_3^V)+3\frac{\gamma}{\epsilon}\,.$$
		Then for~$\gamma$ taken small enough as a function of~$\epsilon$, it follows that
		$$|d(V_1,V_2,V_3)-d(V_1 \cup X_1, V_2 \cup X_2, V_3 \cup X_3) | \leq \epsilon/2 $$
		and
		$$|d(W_1,W_2,W_3) - d(W_1^V,W_2^V,W_3^V)  | \leq \epsilon/2 \,.$$
		Finally the triangle inequality shows$$|d(W_1,W_2,W_3) - d(V_1 \cup X_1, V_2 \cup X_2, V_3 \cup X_3) | \leq 2 \epsilon$$ as needed.
	\end{proof}

	\begin{proof}[Proof of Theorem \ref{thm:palette_regularity}]

		We directly prove the `more generally' part of the Theorem, which implies the first part by taking an arbitrary equipartition of~$m$ parts. Let~$\gamma$ be the result of Lemma~\ref{lem:regular_superset} applied to~$\epsilon$. By increasing the value of~$m$ we may assume that~$\ceil{\frac{16}{\epsilon^6}}\frac{1}{2^m} \leq \min{(\epsilon,\gamma/4)}$ and~$m \geq 6$ - this suffices for the general case. Inductively define a sequence~$t_0 = 2m$ and~$t_{i+1} = 2t_i 2^{3t_i^2}2^{t_i}$. We will show that taking~$M =N= t_{\ceil{\frac{16}{\epsilon^6}}}$ suffices.
		
		Indeed, let~$\cA^0$ be the initial equipartition with~$s_0=m$ parts and~$V_0^0=\emptyset$. Iterate the following process, beginning with~$i=0$. If~$\cA^i$ is not~$\epsilon$-regular, apply Lemma \ref{lem:irregular_incr} to obtain a refinement~$\cA^{i+1}$ with~$q(\cA^{i+1}) \geq q(\cA^i) + \frac{\epsilon^6}{16}$, so that~$|V^{i+1}_0| \leq |V^i_0| + \frac{n}{2^{s_i}}$ and~$\cA^{i+1}$ has~$s_{i+1} \leq t_{i+1}$ parts (since the inductive definition of~$t_i$ matches the output of Lemma~\ref{lem:irregular_incr}). Since~$q(\cdot{}) \in [0,1]$, it follows that in~$\ceil{\frac{16}{\epsilon^6}}$ steps we must find some~$\cA^j$ which is~$\epsilon$-regular, with~$s\leq M$ parts, at which point our iterative process halts.
        
        We can estimate the size of the exceptional set as
        $$|V_0^j| \leq n\left(\ceil{\frac{16}{\epsilon^6}}\frac{1}{2^m}\right)\leq \epsilon n \,.$$
        All that remains is to redistribute~$V_0^j$ amongst the other parts of the equipartition. By evenly distributing the vertices of~$V_0^j$ no part will receive more than~$2\frac{|V_0^j|}{s}\leq \gamma $ many vertices with~$2\frac{|V_0^j|/s}{|V_i|}\leq \gamma$ so that Lemma~\ref{lem:regular_superset} guarantees the~$2 \epsilon$-regularity of our equipartition. Applying the above proof to~$\epsilon'=\frac{\epsilon}{2}$ would provide~$\epsilon$-regularity; we leave the proof as written for readability.
	\end{proof}

        Next we show, as done in~\cite{AFKS:00}, how to iterate Theorem~\ref{thm:palette_regularity} and then randomize to obtain Corollary~\ref{cor:iterated_palette_regularity_app}.
        To do so we will require another property of~$q$. If~$\cB \prec \cA$ it follows from Lemma~\ref{lem:q_props} that~$q(\cB) \geq q(\cA)$; we will show that if~$q(\cB)$ remains very close to~$q(\cA)$, then most of the densities~$d$ found across parts of~$\cB$ are very close to those found in~$\cA$.
	
	\begin{lemma}\label{lem:q_close}
		Suppose equipartitions~$\cA = \{ V_i : i \in [t]\}$ and refinement~$\cB = \{ V_{i,j}: i \in [k],j \in [\ell]\}$ have~$q(\cB) - q(\cA) \leq \epsilon^4/64$. Then for all but~$\epsilon t^3$~$(i_1,i_2,i_3)\in [t]^3$, all but~$\epsilon \ell^3$~$(j_1,j_2,j_3) \in [\ell]^3$ have~$|d(V_{i_1},V_{i_2},V_{i_3})-d(V_{i_1,j_1},V_{i_2,j_2},V_{i_3,j_3})| \leq \epsilon$.
	\end{lemma}
	\begin{proof}
		First fix~$(i_1,i_2,i_3) \in [t]^3$ and let~$Z$ be the random variable defined in Lemma~\ref{lem:q_props}, over the refinement given by~$\cB$ of the sets~$V_{i_1},V_{i_2},V_{i_3}$. If~$L^{(i_1,i_2,i_3)} \subseteq [\ell]^3$ are those~$(j_1,j_2,j_3) \in [\ell]^3$ with~$|d(V_{i_1},V_{i_2},V_{i_3})-d(V_{i_1,j_1},V_{i_2,j_2},V_{i_3,j_3})| \geq \epsilon$, and~$|L^{(i_1,i_2,i_3)}| \geq \epsilon \ell^3$, then
		$$\epsilon \ell^3 \left(\frac{1}{2\ell}\right)^3 \leq \sum_{(j_1,j_2,j_3) \in L^{(i_1,i_2,i_3)}}\frac{|V_{i_1,j_1}||V_{i_2,j_3}||V_{i_3,j_3}|}{|V_{i_1}||V_{i_2}||V_{i_3}|}\leq \Pr( |Z-\bbE(Z)| \geq \epsilon ) \leq \frac{\Var(Z)}{\epsilon^2}$$
		and therefore, recalling we calculated~$\Var(Z)$ in Lemma~\ref{lem:q_props},
		$$q(\bigdcup_{j_1 \in [\ell]} V_{i_1,j_1},\bigdcup_{j_2 \in [\ell]} V_{i_2,j_2},\bigdcup_{j_3 \in [\ell]}  V_{i_3,j_3})-q(V_{i_1},V_{i_2},V_{i_3}) \geq \frac{\epsilon^3}{8}\frac{|V_{i_1}||V_{i_2}||V_{i_3}|}{n^3}\,.$$
		Next define~$I \subseteq [t]^3$ as those~$(i_1,i_2,i_3) \in t^3$ for which~$|L^{(i_1,i_2,i_3)}| \geq \epsilon \ell^3$ and suppose for contradiction that~$|I| > \epsilon t^3$. Then
		\begin{align*}q(\cB) - q(\cA) &\geq  \sum_{(i_1,i_2,i_3) \in I}q(\bigdcup_{j_1 \in [\ell]} V_{i_1,j_1},\bigdcup_{j_2 \in [\ell]} V_{i_2,j_2},\bigdcup_{j_3 \in [\ell]}  V_{i_3,j_3})-q(V_{i_1},V_{i_2},V_{i_3})\\ 
        &> \epsilon t^3 \frac{\epsilon^3}{8}\frac{|V_{i_1}||V_{i_2}||V_{i_3}|}{n^3} \geq \frac{\epsilon^4}{64} \end{align*}
        a contradiction.
	\end{proof}

        We can now iterate~\ref{thm:palette_regularity} to prove Theorem~\ref{thm:iterated_palette_regularity}.
        
	\begin{proof}[Proof of Theorem \ref{thm:iterated_palette_regularity}]
		Given any~$m \in \NN$ and~$\epsilon >0$, let~$M_{\ref{thm:palette_regularity}}(m,\epsilon)$ and~$N_{\ref{thm:palette_regularity}}(m,\epsilon)$ denote the output of Theorem~$\ref{thm:palette_regularity}$. Let~$\cE$ and~$m$ be as in the Theorem statement, and let~$\epsilon = \cE(0)$. Set~$M_0=M_{\ref{thm:palette_regularity}}(m,\epsilon)$ and~$N_0=N_{\ref{thm:palette_regularity}}(m,\epsilon)$ before inductively defining
		\begin{align*}
			M_i &= M_{\ref{thm:palette_regularity}}\left(M_{i-1},\frac{\cE(M_{i-1})}{M_{i-1}^3} \right)\\
			N_i &= N_{\ref{thm:palette_regularity}}\left(M_{i-1},\frac{\cE(M_{i-1})}{M_{i-1}^3} \right)\,.\\
		\end{align*}
		If we set~$s= \ceil{\frac{64}{\epsilon^4}}+1$, then we claim that~$M = M_{s},N = N_{s}$ suffices.
        
        First let~$\cA^0$ be an~$\epsilon$-regular equipartition of order~$t_0 \in [m,M_0]$, provided by Theorem~\ref{thm:palette_regularity} and then perform the following iterative procedure, starting with~$i=0$ (where~$t_{-1}=0$ for convenience).
        Given~$\cA^i$ with order~$t_i \in [t_{i-1},M_i]$ we may apply  Theorem~\ref{thm:palette_regularity} again to obtain a refinement~$\cA^{i+1}$ which is~$\frac{\cE(M_{i})}{M_{i}^3}$-regular with order~$t_{i+1} \in [t_i,M_{i+1}]$. In particular, at most~$\frac{\cE(M_{i-1})}{M_{i-1}^3} t_i^3$ of~$(i_1,i_2,i_3) \in [t_i]^3$ have~$(V_{i_1},V_{i_2},V_{i_3})$ fail to be~$\cE(M_{i-1})$-regular, since~$\frac{\cE(M_{i-1})}{M_{i-1}^3} \leq \cE(M_{i-1})$. 
        Let~$i$ be the first~$i$ so that~$q(\cA^i) - q(\cA^{i-1}) \leq \frac{\epsilon^4}{64}$, and set~$\cA = \cA_{i-1}$ and~$\cB = \cA_{i}$, with~$t = t_{i-1}$ and~$t\ell = t_{i}$. It remains to check items~\ref{it:part_count}-\ref{it:palette-reg-density}.

        Part~\ref{it:part_count} follows immediately by recalling that the~$t_i$ are increasing, so ~$m \leq t_0 \leq t \leq t \ell \leq M$.
        Part~\ref{it:eps0} also follows immediately using the monotonicity of~$\cE$, since all but~$(\frac{\cE(M_{i-2})}{M_{i-2}^3})[t_{i-1}]^3 \leq \cE(0)[t_{i-1}]^3$ of~$(i_1,i_2,i_3) \in [t_{i-1}]^3$ have~$(V_{i_1}V_{i_2},V_{i_3})$ as~$\cE(0)$-regular.
        For Part~\ref{it:epst}, we have that the partition~$\cB$ has at most~$\frac{\cE(M_{i-1})}{M_{i-1}^3} (t \ell)^3$ pairs~$\left((i_1,i_2,i_3),(j_1,j_2,j_3) \right) \in ([t]^3,[\ell]^3)$ with~$(V_{i_1,j_1},V_{i_2,j_2},V_{i_3,j_3})$ failing to be~$\cE(M_{i-1})$-regular.
        Since~$\frac{\cE(M_{i-1})}{M_{i-1}^3} (t \ell)^3 \leq \cE(M_{i-1}) \ell^3$, there is certainly no~$(i_1,i_2,i_3) \in [t]^3$ with more than~$\cE(t)\ell^3$ many~$(j_1,j_2,j_3) \in [\ell]^3$ with~$(V_{i_1,j_1},V_{i_2,j_2},V_{i_3,j_3})$ failing to be~$\cE(M_{t})$-regular.
        Finally, Part~\ref{it:palette-reg-density} follows by direct application of Lemma~\ref{lem:q_close}.
		
	\end{proof}

        Finally we may sample from the refinement~$\cB$ of~$\cA$ to find the model vertex sets we require. 
	
	\begin{proof}[Proof of Corollary \ref{cor:iterated_palette_regularity_app}]
		We apply Theorem~\ref{thm:iterated_palette_regularity} with~$\cE'(r) = \min\{\cE(r),\frac{1}{4t^3},\frac{\epsilon}{8}\}$ and~$m$ to obtain~$M_{\ref{thm:iterated_palette_regularity}}$ and~$N_{\ref{thm:iterated_palette_regularity}}$ and claim that ~$M=M_{\ref{thm:iterated_palette_regularity}}$,~$N=N_{\ref{thm:iterated_palette_regularity}}$, and~$\delta = \frac{1}{2M}$ suffices. To that end, let~$Q$ be any palette with at least~$N$ colors, so that Theorem~\ref{thm:iterated_palette_regularity} gives~$\cB$ refining~$\cA$ so that Parts~\ref{it:part_count}-\ref{it:palette-reg-density} hold. For each~$i \in [t] $ select~$j_i \in [\ell]$, independently uniformly at random, and set~$U_i = V_{i,j_i}$. We now show that parts~\ref{it:num_parts}-\ref{it:model_accuracy} are satisfied with positive probability. Parts~\ref{it:num_parts} and~\ref{it:model_size} are immediate and always hold. For Part~\ref{it:model_regularity},
        \begin{align*}
            \PP\left( (U_{i_1},U_{i_2},U_{i_3}) \text{ is not }\epsilon\text{-regular for some }(i_1,i_2,i_3) \in [t]^3 \right) \leq t^3\cE'(t) \leq \frac{1}{4}
        \end{align*}
        by applying the union bound and Part~\ref{it:epst} of Theorem ~\ref{thm:iterated_palette_regularity}. Meanwhile,
        \begin{align*}
            \EE (|(i_1,i_2,i_3) \in [t]^3 \text{ with }|d(U_{i_1},U_{i_2},U_{i_3}) - d(V_{i_1},V_{i_2},V_{i_3})  | \geq \epsilon  | ) \leq \frac{\epsilon}{8}t^3+\frac{\epsilon}{8}t^3
        \end{align*}
        by Part~\ref{it:palette-reg-density} of~\ref{thm:iterated_palette_regularity}, and therefore the probability that there are more than~$\epsilon t^3$~$(i_1,i_2,i_3) \in [t]^3$ with~$|d(U_{i_1},U_{i_2},U_{i_3}) - d(V_{i_1},V_{i_2},V_{i_3})  | \geq \epsilon$ cannot exceed~$\frac{1}{4}$. Then with probability at least~$\frac{1}{2}$ both Part~\ref{it:model_regularity} and~\ref{it:model_accuracy} are satisfied as well, so there exists such a choice of~$U_i$ and we are done.
	\end{proof}

\end{document}